\documentclass[margincite]{sl2art}
\usepackage{marginfix}

\makeatletter
\def\qed@warning{}
\makeatother

\usepackage{tikz}
\usepackage{tikz-cd}
\usetikzlibrary{braids}

\usetikzlibrary{
  spath3,
  hobby,
  decorations.pathreplacing,
  decorations.markings,
  shapes.geometric,
  calc,
  intersections
}

\tikzset{
  every path/.style={
    ultra thick,
  },
  show curve controls/.style={
    postaction=decorate,
    decoration={show path construction,
      curveto code={
        \draw [blue, dashed]
        (\tikzinputsegmentfirst) -- (\tikzinputsegmentsupporta)
        node [at end, draw, solid, red, inner sep=2pt]{};
        \draw [blue, dashed]
        (\tikzinputsegmentsupportb) -- (\tikzinputsegmentlast)
        node [at start, draw, solid, red, inner sep=2pt]{}
        node [at end, fill, blue, ellipse, inner sep=2pt]{}
        ;
      }
    }
  },
  show curve endpoints/.style={
    postaction=decorate,
    decoration={show path construction,
      curveto code={
        \node [fill, blue, ellipse, inner sep=2pt] at (\tikzinputsegmentlast) {}
        ;
      }
    }
  }
}
\usepackage{todonotes}

\usepackage{graphicx}
\graphicspath{{fig/}}

\usepackage{enumitem}
\usepackage{mathtools}
\usepackage{soul}

\usepackage{tabularx}

\newlist{thmenum}{enumerate}{1}
\setlist[thmenum]{label=(\alph*)}

\RequirePackage[noabbrev,capitalize]{cleveref}
\crefname{equation}{equation}{equations}
\Crefname{equation}{Equation}{Equations}

\def\leftfun#1{\mathopen{}\left#1}
\def\rightfun#1{\right#1}

\ExplSyntaxOn
\DeclareExpandableDocumentCommand{\IfEmptyTF}{mmm}
 {
    \tl_if_empty:nTF {#1} {#2} {#3}
 }
\ExplSyntaxOff

\newcommand{\kk}{\Bbbk}
\newcommand{\kkm}{\Bbbk^{\times}}
\newcommand{\ZZ}{\mathbb{Z}}

\newcommand{\tu}{\mathbf{1}}

\NewDocumentCommand{\qdim}{O{}}{\operatorname{dim}^{#1}}

\DeclareMathOperator{\id}{id}
\DeclareMathOperator{\Id}{Id}
\DeclareMathOperator{\End}{End}
\renewcommand{\hom}{\operatorname{Hom}}
\DeclareMathOperator{\im}{im}

\DeclareMathOperator{\Inv}{Inv}
\DeclareMathOperator{\Rep}{Rep}
\NewDocumentCommand{\chr}{m}{\widehat{#1}} %

\newcommand{\set}[1]{\left\{ \def\given{\ \middle| \ }  #1 \right\}  }

\newcommand{\defeq}{\mathrel{:=}}
\newcommand{\iso}{\cong}

\NewDocumentCommand{\qn}{}{\triangleleft}

\newcommand{\mer}{\mathfrak{m}}

\DeclareMathOperator{\lk}{lk}
\NewDocumentCommand{\comp}{m}{S^3 \setminus L}
\NewDocumentCommand{\surg}{m}{S^3(#1)}

\NewDocumentCommand{\homol}{D[]{2} m }{\operatorname{H}_#1(#2)}

\NewDocumentCommand{\cohomol}{D[]{2} m }{\operatorname{H}^#1(#2)}
\NewDocumentCommand{\cocycle}{D[]{2} m }{\operatorname{Z}^#1(#2)}
\NewDocumentCommand{\cobound}{D[]{2} m }{\operatorname{B}^#1(#2)}
\NewDocumentCommand{\cochain}{D[]{2} m }{\operatorname{C}^#1(#2)}

\NewDocumentCommand{\zcohomol}{D[]{2} m }{\operatorname{zH}^#1(#2)}
\NewDocumentCommand{\zcocycle}{D[]{2} m }{\operatorname{zZ}^#1(#2)}
\NewDocumentCommand{\zcobound}{D[]{2} m }{\operatorname{zB}^#1(#2)}
\NewDocumentCommand{\zcochain}{D[]{2} m }{\operatorname{zC}^#1(#2)}

\NewDocumentCommand{\zestts}{O{\zeta}}{\overset{\scriptstyle #1}{\otimes}}
\NewDocumentCommand{\warp}{}{\mathbin{\rtimes \mskip -14mu \bigcirc}}

\NewDocumentCommand{\linkinvname}{}{\mathcal{F}}
\NewDocumentCommand{\linkinv}{m m}{\linkinvname_{#2}\mathopen{}\left(#1\right)}

\NewDocumentCommand{\rtinvname}{}{\mathcal{Z}}
\NewDocumentCommand{\rtinv}{m m}{\rtinvname_{#2}\mathopen{}\left(#1\right)}

\NewDocumentCommand{\zestinvname}{}{\mathcal{J}}

\NewDocumentCommand{\zest}{O{\zeta} m }{%
  \IfEmptyTF{#2}{%
    \mathcal{J}_{#1}%
  }{%
    \mathcal{J}_{#1}\leftfun(#2\rightfun)%
  }
}
\NewDocumentCommand{\zestmini}{O{\zeta} m }{%
  \IfEmptyTF{#2}{%
    j_{#1}%
  }{%
    j_{#1}\leftfun(#2\rightfun)%
  }
}

\NewDocumentCommand{\tangcat}{O{A}}{\mathsf{Tang}^{#1}}
\NewDocumentCommand{\du}{}{\amalg} %
\NewDocumentCommand{\delooping}{m}{\mathsf{B}{#1}} %

\NewDocumentCommand{\ev}{}{\operatorname{eval}}
\NewDocumentCommand{\coev}{}{\operatorname{coeval}}

\NewDocumentCommand{\surgcol}{m O{\mathcal{C}}}{\Omega_{#1}^{#2}}
\NewDocumentCommand{\indset}{m}{I_{#1}}
\newcommand{\globaldim}{D}

\newcommand{\parcen}[1]{\phantom{(}#1\phantom{)}}
\newcommand{\br}{to [out=90,in=-90]}
\newcommand{\onecirc}[3]{
\begin{scope}[line width=1]
\draw (0,0) node[below] {$#1$} --(0,1.5) node[above] {$#2$};
\draw[fill=white] (0,.75) circle (.5) node {$#3$};
\end{scope}
}

\declaretheorem[style=theorem]{proposition}
\declaretheorem[style=theorem,sibling=proposition]{theorem}
\declaretheorem[style=theorem,sibling=proposition]{lemma}

\declaretheorem[style=theorem,sibling=proposition]{corollary}
\declaretheorem[style=definition,sibling=proposition]{definition}
\declaretheorem[style=definition,sibling=proposition]{remark}
\declaretheorem[style=definition,sibling=proposition]{example}

\usepackage{biblatex}
\title{Zesting and the relative complexity of Reshetikhin-Turaev invariants}
\author{Colleen Delaney}
\address{Purdue University}
\email{colleend@purdue.edu}
\author{Calvin McPhail-Snyder}
\address{Duke University}
\email{calvin@sl2.site}

\begin{document}

\begin{abstract}
We show that the computational complexity of Reshetikhin-Turaev invariants of simply colored links is preserved when their underlying ribbon fusion categories are related by the zesting construction. 
Zesting modifies an \(A\)-graded ribbon fusion category \(\mathcal{C}\) with additional algebraic data \(\zeta\) to produce a new category \(\mathcal{C}^{\zeta}\) whose link invariants are known to differ from those of \(\mathcal{C}\) by an invariant of \(A\)-colored links \(\mathcal{J}_{\zeta}\) depending only on \(\zeta\). 
Building on this understanding and on earlier work on quantum braid group representations under zesting, our result suggests how zesting contributes to the organization of (2+1)D topological quantum field theories and topological phases into complexity-theoretic hierarchies.
To prove our main result we develop a local formalism analogous to the Reshetikhin-Turaev construction to compute \emph{tangle} invariants \(\mathcal{J}_{\zeta}(T)\), which leads to a polynomial time algorithm to compute invariants of links \(\mathcal{J}_{\zeta}(L)\). A byproduct of our construction is an identification (up to a sign) of the link invariants \(\mathcal{J}_{\zeta}(L)\) as rack cocycle invariants, which may be of independent interest. 
Our formalism also extends to define invariants of closed \(3\)-manifolds with \(A\)-structure and we obtain similar complexity results for homotopy quantum field theories built from \(A\)-modular fusion categories.
\end{abstract}


\maketitle
\section{Introduction}

By definition extended \((2+1)\)D TQFTs (topological quantum field theories in three spacetime dimensions) define invariants of tangles, links, and 3-manifolds.
The Reshetikhin-Turaev construction \cite{Reshetikhin1990,Reshetikhin1991,Turaev1994} provides not only a \((2+1)\)D TQFT but also an explicit, concrete algorithm to compute the associated invariants using the algebraic data of a modular fusion category.
Since the introduction of quantum invariants mathematicians have been interested in understanding their computational complexity as it relates to their strength \cite{Jaeger1990}. 
More recently this question has been studied from the perspective of topological quantum computation.
For instance, the computational complexity of Reshetikhin-Turaev invariants is closely related to the images of the quantum braid group representations.
These characterize the exchange-generated quantum gates realized by the anyons in a topological phase of matter governed by the TQFT \cite{Kirby2004, Rowell2009, Kuperberg2015}.

Recent work has shown that quantum 3-manifold invariants are generically hard to compute \cite{Delaney2025, Samperton2025, Galindo2026}.
However, the landscape of quantum link invariants is still largely unexplored.
While the complexity of certain link coloring invariants in untwisted Dijkgraaf-Witten TQFTs has been quantified \cite{KroviRussell2015, SampertonKuperberg2021}, a careful complexity analysis for a given TQFT--whether for links or \(3\)-manifolds--will depend on fine-grained details about its structure.
In the absence of a full description of the algebraic data underlying a given TQFT needed to precisely compute these invariants, it is interesting to explore its complexity \emph{relative} to other TQFTs.
In particular one might ask when two different TQFT invariants have equivalent complexity.
More generally, one might try to form hierarchies or classes of topological phases of matter much like one organizes decision or counting problems into complexity classes. 

Here we make progress on answering this question:
we show that if two ribbon fusion categories \(\mathcal{C}\) and \(\mathcal{D}\) are related by the zesting construction with respect to their universal grading groups \cite{Delaney2020} then the complexity of computing their Reshetikhin-Turaev link invariants (colored by simple objects) is the same. This can be viewed as a companion result to \cite[Theorem 5.15]{Delaney2020}, which relates the quantum braid group representations and shows that the quantum gates generated by anyons related by zesting are projectively equivalent.

Our new result is a corollary of a more general fact about zesting and Reshetikhin-Turaev invariants.
Suppose we have \defemph{zesting data} \(\zeta\) for a modular tensor (or simply ribbon) category \(\mathcal{C}\).
It was previously observed in \cite{Delaney2021} that zesting multiplies the Reshetikhin-Turaev invariant of a framed link colored by simple objects of \(\mathcal{C}\) by a factor we call the \defemph{zest invariant} \(\zest{L}\).
This invariant depends only on the link and the input data $\zeta$ to the zesting construction and not the full details of the original category \(\mathcal{C}\).

Here we extend this result to tangles and give an independent definition of \(\zest{T}\) as the image of a 2-functor from a 2-category of shadow-colored tangles.
As in the original work of \textcite{Reshetikhin1990} this functor admits a local formulation that leads to a straightforward method to compute \(\zest{T}\).
By combining several well-known facts about the efficiency of performing arithmetic in \emph{fixed} number fields over $\mathbb{Q}$ and converting link diagrams to braid closures \cite{Vogel1990}, we show that our local formula gives a polynomial-time algorithm to compute the zest invariant of a link, proving our main theorem. 
Along the way we give a characterization of the zest invariant of a link by extracting a shadow quandle cocycle \(\phi\) \cite{Carter2003,Inoue2014} from the zesting data \(\zeta\), and show that when \(\mathcal{C}\) and \(\mathcal{C}^\zeta\) are pseudo-unitary the link invariant \(\zest{L}\) is equal to the quandle cocycle invariant determined by \(\phi\). (In general it is equal up to a trivial correction factor.)

Finally, we show that the Reshetikhin-Turaev-like construction of the zest invariant of tangles extends to an invariant of closed \(3\)-manifolds with \(A\)-structures for \(A\)-modular fusion categories. We show there is an analogous factorization and polynomial-time Turing equivalence of \(A\)-manifold invariants for \(A\)-modular categories related by zesting. We leave open questions about the relative complexity of the Reshetikhin-Turaev invariants of closed \(3\)-manifolds, as well as the story for more general grading groups beyond the universal one. 

\subsection*{Layout of the paper}
In \cref{sec: background} we review the zesting construction on ribbon fusion categories and its interplay with the Reshetikhin-Turaev construction. In \cref{sec:zest invariant} we introduce a functor between the \(2\)-category of shadow-colored, (oriented, framed) tangles and the delooping of the pointed, trivially graded subcategory of \(\mathcal{C}\) where the zesting data is valued. We define the zest invariant of a shadow-colored tangle as the image of this functor and show that on links it produces the invariant that first appeared in \cite{Delaney2021}. \cref{sec:zest invariant examples} is devoted to the computation of the zest invariant of a link. We first illustrate how to compute it in practice through an example, both as a morphism and a number. Using the formalism developed in \cref{sec:zest invariant} we give a local formula to compute the zest invariant for an arbitrary link and show--up to a sign that vanishes in the pseudo-unitary setting--that it is given by a shadow-rack cocycle invariant. We explain how the local formula gives a polynomial-time algorithm to compute the zest invariant, and therefore that zesting preserves the computational complexity of the Reshetkhin-Turaev link invariants. \cref{sec:zest manifold invariant} extends this story to define the zest invariant of \(3\)-manifolds with \(A\)-structure, which are also efficient to compute. Some details are relegated to appendices; some technical lemmas involving zesting that streamline the presentation are collected in \cref{sec:ribbon zesting conditions}, and the full details that connect the zesting data with shadow-rack cocycles can be found in \cref{sec:proof of cocycle condition}.

\subsection*{Acknowledgments} 
The authors thank Eric Rowell and Eric Samperton for many helpful conversations and C\'{e}sar Galindo for help with the proof of \cref{thm:ribbon condition reformulation}.
Part of this work was completed while CMS was a visiting scholar at UNC Chapel Hill and he thanks the Mathematics Department for their hospitality.
Part of this work was supported by U.S. National Science Foundation Grant No. 2530723.

Any opinions, findings, and conclusions or recommendations expressed in this material are those of the author(s) and do not necessarily reflect the views of the National Science Foundation.
\section{Background on zesting and Reshetikhin-Turaev invariants}
\label{sec: background}
We set notation and recall some basic notions about ribbon fusion categories in \cref{sec: preliminaries} before giving an overview of the ribbon zesting construction in \cref{sec:zesting summary}. \cref{sec:rt invariants and factorization} describes the relationship between the Reshetikhin-Turaev construction and the ribbon zesting construction and lays the groundwork for \cref{sec:zest invariant}.

\subsection{Preliminaries}
\label{sec: preliminaries}
In this paper we assume that \(\mathcal{C}\) is a ribbon fusion category over an algebraically closed field \(\kk\) of characteristic \(0\).
We write \(\tu\) for its unit object, \(\alpha\) for its associator, $\beta$ for its braiding, and \(\theta\) for its twists. The quantum dimension of an object $X$ is denoted by \(\dim(X)\).
We say \(\mathcal{C}\) is \defemph{pseudo-unitary} if the quantum dimensions of simple objects are equal to their Frobenius-Perron dimensions, which implies that \(\dim(X)\ge 1\) for all \(X \in \mathcal{C}\).

Assuming \(\kk\) is algebraically closed ensures that for any endomorphism \(f : X \to X\) of a simple object \(X\) there is a scalar \(\langle f \rangle \in \kk\)  defined by
 \[
   f = \langle f \rangle \id_{X}.
\]

\subsubsection{Invertible objects}
An object $X \in \mathcal{C}$ is called \defemph{invertible} if \(X \otimes X^* \cong \tu \), where $X^*$ is the dual of $X$. The set of isomorphism classes of invertible objects forms a group under fusion, which we denote by \(\Inv(\mathcal{C})\). At times we will abuse notation and fix a set of representatives of these invertible isomorphism classes and identify the elements of \(\Inv(\mathcal{C})\) with these representatives, e.g.~thinking of the identity as \(\tu\) and not \([\tu]\). It will be clear from context which of these meanings is being invoked. The subcategory generated by invertible objects is called the \defemph{pointed subcategory} \(\mathcal{C}_{pt}\). The Frobenius-Perron dimension of an invertible $X$ is always equal to 1, and its quantum dimension $\dim(X)$ is always \(1\) or \(-1\).

\subsubsection{Gradings}
We say \(\mathcal{C}\) is \defemph{\(G\)-graded} if there is a decomposition \(\mathcal{C}=\oplus_{g \in G} \mathcal{C}_g\) such that whenever \(X \in \mathcal{C}_g\) and \(Y \in \mathcal{C}_h\), \(X \otimes Y \in \mathcal{C}_{gh}\). 
  We call an object \(X\) \defemph{homogeneous} if \(X \in \mathcal{C}_g\) for some \(g \in G\) and will either write \(X_g\) or \(|X|:=g\) to indicate the grading of a homogeneous object.

The $G$-grading is \defemph{faithful} if \(\mathcal{C}_g\) is nonempty for each $g \in G$. Any faithful grading group of \(\mathcal{C}\) is necessarily finite (since \(\mathcal{C}\) is fusion) and abelian (since it is braided.)
Every fusion category \(\mathcal{C}\) admits a grading by its \defemph{universal grading group} \(U\), which is universal in the sense that any faithful grading of \(\mathcal{C}\) factors through \(U\) \cite[Theorem 3.5]{Gelaki2008}.
Later we will denote the universal grading group of \(\mathcal{C}\) by \(A\), since for us it is abelian.

\subsubsection{Diagrammatics}
We use string diagrams to depict morphisms in \(\mathcal{C}\), and our conventions are to compose morphisms down the page and to take the tensor product of morphisms (horizontal composition) from left to right. When no orientation is given to a strand in a diagram it is understood to be oriented down the page. All pictures are assumed to be blackboard-framed unless stated otherwise. 

We draw the braiding with the following conventions for crossings.

\begin{align}
\label{eq: braiding conventions}
\beta_{X,Y}=
  \begin{tikzpicture}[line width=1, baseline=12.5, scale=.75] 
    \draw[<-] (0,0) node[below] {$Y$} \br (1,1.5) node[above] {$Y$};
    \draw[white, line width=10] (1,0) \br (0,1.5);
    \draw[<-] (1,0) node[below] {$X$} \br (0,1.5) node[above] {$X$};
    \draw (0.5,-1) node[below] {positive crossing};
  \end{tikzpicture}
  \qquad& \qquad
  \beta_{Y,X}^{-1}=
  \begin{tikzpicture}[line width=1, baseline=12.5, scale=.75] 
    \draw[<-] (1,0) node[below] {$X$} \br (0,1.5) node[above] {$X$};
    \draw[white, line width=10] (0,0) \br (1,1.5);
    \draw[<-] (0,0) node[below] {$Y$} \br (1,1.5) node[above] {$Y$};
    \draw (0.5,-1) node[below] {negative crossing};
  \end{tikzpicture} 
\end{align}
Note that this convention is opposite to the one most common in knot theory. 
For the twist isomorphisms $\theta$ and their inverses we use the following pictures.

\begin{align}
\label{eq: twist conventions}
\theta_X = 
\begin{tikzpicture}[line width=1,baseline=.75cm, scale=.75]
	\draw (0,0) node[below] {$X$}-- (0,1) to [out=90,in=180] (.25,1.25) to [out=0,in=90] (.5,1) to [out=-90,in=0] (.25,.75);
	\draw[white,line width=10] (.25,.75) to [out=180,in=-90] (0,1)--(0,2);
	\draw (.25,.75) to [out=180,in=-90] (0,1)--(0,2) node[above] {$X$};
	\end{tikzpicture}
    \qquad & \qquad
    \theta_X^{-1}=
    \begin{tikzpicture}[line width=1,baseline=-.75cm, scale=.75, yscale=-1]
	\draw (0,0) node[above] {$X$}-- (0,1) to [out=90,in=180] (.25,1.25) to [out=0,in=90] (.5,1) to [out=-90,in=0] (.25,.75);
	\draw[white,line width=10] (.25,.75) to [out=180,in=-90] (0,1)--(0,2);
	\draw (.25,.75) to [out=180,in=-90] (0,1)--(0,2) node[below] {$X$};
	\end{tikzpicture}
\end{align}
We will use all of these same conventions for tangle diagrams in the ribbon 1-category of tangles.

\subsubsection{Double braidings}
\label{sec:prelim:double braidings}

Recall that the symmetric center 
\[
\mathcal{Z}_2(\mathcal{C})= \left \{ X \in \mathcal{C} \, \lvert \, \beta_{Y,X} \circ \beta_{X,Y}=\id_{X \otimes Y} \text{ for all } Y \in \mathcal{C} \right \}
\]
is the subcategory of objects which double braid trivially with all other objects in the category. An object $X$ is called \defemph{transparent} if \(X \in \mathcal{Z}_2(\mathcal{C})\) and \(\mathcal{C}\) is called \defemph{nondegenerate} if the only object in \(\mathcal{Z}_2(\mathcal{C})\) is \(\tu\).

The double braiding allows us to relate invertible objects to tensor automorphisms of the identity functor, hence to characters on the universal grading group.
For any simple object \(X \in \mathcal{C}\) and any invertible object \(\eta \in \Inv(\mathcal{C})\) there is a morphism \(\chi_\eta(X) : X \to X\) defined implicitly by $\beta_{\eta, X} \circ \beta_{X,\eta} =: \chi_\eta(X) \otimes \id_\eta$ as in the following picture:

\begin{align*}
\label{eq:chi-definition}
 \begin{tikzpicture}[line width=1, baseline=30, scale=.75] 
    \draw (0,0) node[below] {$\parcen{X}$} \br (1,1.5) node[above] {};
    \draw[white, line width=10] (1,0) \br (0,1.5);
    \draw (1,0) node[below] {$\parcen{\eta}$} \br (0,1.5) node[above] {};
    \begin{scope}[yshift=1.5cm]
        \draw (0,0) node[below] {} \br (1,1.5) node[above] {$\parcen{\eta}$};
    \draw[white, line width=10] (1,0) \br (0,1.5);%
    \draw (1,0) node[below] {} \br (0,1.5) node[above] {$\parcen{X}$};
    \end{scope}
    \end{tikzpicture}
    \quad
    = 
    \qquad
    \begin{tikzpicture}[line width=1, baseline=30, scale=.75] 
    \draw (0,0) node[below] {$\parcen{X}$}--(0,3) node[above] {$\parcen{X}$};
    \draw[fill=white] (-.5,1) rectangle node[] {\small $  \chi_\eta$} (0.5,2); 
    \draw (1,0) node[below] {$\parcen{\eta}$}--(1,3) node[above] {$\parcen{\eta}$};
    \end{tikzpicture}
\end{align*}
\begin{proposition}
  \label{thm:chi facts}
  For \(\eta\) an invertible object of \(\mathcal{C}\),
  \begin{thmenum}
    \item Each fixed \(\chi_{\eta}\) is a monoidal natural isomorphism.
    \item \(\chi : \operatorname{Inv}(\mathcal{C}) \to \operatorname{Aut}_{\otimes}(\id_{\mathcal{C}})\) is a group homomorphism.
    \item The kernel of \(\chi\) is the subgroup of transparent invertible objects.
    \item \(\chi_{\eta^{*}}(X) = \chi_{\eta}(X)^{-1}\).
      \qedhere
  \end{thmenum}
\end{proposition}

\begin{proof}
  Parts (a--c) are  \cite[Proposition 2.8]{Delaney2020}, which is essentially \cite[Lemma 8.22.9]{EGNO2015}.
  Part (d) is straightforward to derive using the spherical structure on \(\mathcal{C}\) (it suffices to know that the subcategory of invertibles is spherical).
\end{proof}

Any automorphism \(\psi \in \operatorname{Aut}_{\otimes}(\id_{\mathcal{C}})\) is determined by the scalar-valued function
\begin{equation*}
  f_{\psi}( X)
  \defeq 
  \langle \psi(X) \rangle
\end{equation*}
on simple objects of \(\mathcal{C}\).
It turns out that if \(A\) is the \emph{universal} grading of \(\mathcal{C}\) the number \(f_{\psi}(X)\) depends only on the degree of \(X\) (which is homogeneous because it is simple).
More succinctly, \cite[Proposition 4.14.3]{EGNO2015} says the map \(\psi \mapsto f_{\psi}\) defines a canonical isomorphism
\begin{equation*}
  \operatorname{Aut}_{\otimes}(\id_{\mathcal{C}})
  \to
  \widehat{A}.
\end{equation*}
We conclude:

\begin{corollary}
  \label{thm:grading-by-characters}
  Let \(A\) be the universal grading group of \(\mathcal{C}\).
  \begin{thmenum}
  \item \(\chi\) defines a surjective homomorphism \(\operatorname{Inv}(\mathcal{C}) \to \widehat{A}\).
    Explicitly, we have 
    \begin{equation}
      \label{eq:double braiding character definition}
      \chi_{\eta}(a) = \langle \chi_{\eta}(X_a) \rangle.
    \end{equation}
    where \(X_{a}\) is any simple object of \(\mathcal{C}\) in degree \(a\).
  \item If \(\mathcal{C}\) is nondegenerate then \(\operatorname{Inv}(\mathcal{C}) \iso \widehat{A}\).
    \qedhere
  \end{thmenum}
\end{corollary}

\begin{lemma}
  \label{thm:chi lemma}
  Take \(\mathcal{C}\) to be graded by \(A\) and let \(X\) be a homogeneous object with \(|X| = a\).
  Then
  \(
  \beta_{X, \eta}
  =
  \chi_{\eta}(a)
  \beta_{\eta, X}^{-1}
  .
  \)
\end{lemma}

\begin{proof}
  Obvious from the definition of \(\chi\) and \cref{thm:grading-by-characters}.
\end{proof}

\subsection{Zesting and zesting data}
\label{sec:zesting summary}
\begin{definition} 
  \label{def:zesting-conditions}
Let $\mathcal{C}=\bigoplus_{a \in A} \mathcal{C}_a$ be a ribbon fusion category with universal grading group \(A\). The data of a ribbon zesting $\zeta =(\lambda,\nu,t, \epsilon)$ consists of 

\begin{enumerate}[label=(\alph*)]
  \item 
    A function $\lambda : A \times A \to \Inv(\mathcal{C}_e)$ with $\lambda(e,a)=\lambda(a,e)=\tu$ for all $a \in A$, which is a  \(2\)-cocycle%
    \note{%
      When there are no transparent objects (say, when \(\mathcal{C}\) is modular) \(A\) and  \(\chr{\Inv(\mathcal{C})}\) are isomorphic.
      Then the claim that \(\lambda\) is a \(2\)-cocycle is equivalent to the equations \(\omega(a,b) \omega(ab,c) = \omega(b,c) \omega(a,bc)\) in \(\chr{\Inv(\mathcal{C})}\).
    }
    in the sense that for each \(a, b, c \in A\) there exist isomorphisms
    \[
      \lambda(a,b) \otimes \lambda(ab,c) \to \lambda(b,c) \otimes \lambda(a,bc).
    \]
    We write
    \[
      \omega(a,b;c) \defeq \chi_{\lambda(a,b)}(c) \in \kkm.
    \]
    for the values of the character determined by \(\lambda\) via \cref{thm:grading-by-characters}.
  \item A family of isomorphisms
    $$
    \nu(a,b,c): \lambda(a,b) \otimes \lambda(ab,c) \longrightarrow \lambda(b,c) \otimes \lambda(a,bc)
    $$ 
    satisfying the \defemph{associative zesting condition}%
    \begin{align}
      \label{eq:associative zesting}
      \begin{split}
    \left ( \nu(b,c,d) \otimes \id_{\lambda(a,bcd)} \right ) \circ \left ( \id_{\lambda(b,c)}\otimes \nu(a,bc,d) \right ) \circ \left ( \nu(a,b,c) \otimes \id_{\lambda(abc,d)} \right ) \\ =  \left (\id_{\lambda(a,b)} \otimes \nu(a,b,cd) \right ) \circ \left ( \beta^{-1}_{\lambda(a,b),\lambda(c,d)} \otimes \id_{\lambda(ab,cd)} \right ) \circ  \left (\id_{\lambda(c,d)} \nu(ab,c,d) \right )
    \end{split}
    \end{align}
    for all $a,b,c,d \in A$, or graphically
    
    \begin{align}
    \label{fig:associative zesting condition}
    \scalebox{.9}{
    \begin{tikzpicture}[line width=1,scale=1,baseline=1.9cm]
    \foreach \x in {0,1,2}
    {
    \draw(\x,0)--(\x,4);
    }
    \draw[fill=white] (-.25,.75) rectangle node {\small $\nu(b,c,d)$} (1.25,1.25) ;
    \draw[fill=white] (.75,1.75) rectangle node {\small $\nu(a,bc,d)$} (2.25,2.25);
    \draw[fill=white] (-.25,2.75) rectangle node {\small $\nu(a,b,c)$} (1.25,3.25) ;
    \draw (0,0) node[below] {\small $\lambda(c,d)$};
    \draw (1.06,0) node[below] {\small $\lambda(b,cd)$};
    \draw (2.25,0) node[below] {\small$ \lambda(a,bcd)$};
    \draw (0,4) node[above] {\small $\lambda(a,b)$};
    \draw (1.06,4) node[above] {\small $\lambda(ab,c)$};
    \draw (2.25,4) node[above] {\small $\lambda(abc,d)$};
    \end{tikzpicture}}
    =
    \scalebox{.9}{\begin{tikzpicture}[line width=1,scale=1, baseline=1.9cm]
    \foreach \x in {0,1}
    {
    \draw(\x,0)--(\x,1);
    \draw(\x,2.5)--(\x,4);
    }
   \draw(2,0)--(2,4);
   \begin{scope}[xshift=0cm,yshift=1cm]
   \draw (1,0)  \br (0,1.5);
   \draw[white, line width=10] (0,0) \br (1,1.5);
   \draw (0,0) \br (1,1.5);
   \end{scope}
   \draw[fill=white] (.75,.75) rectangle node {\small $\nu(a,b,cd)$} (2.25,1.25);
   \draw[fill=white] (.75,2.75) rectangle node {\small $\nu(ab,c,d)$} (2.25,3.25);
   \draw (0,0) node[below] {\small $\lambda(c,d)$};
   \draw (1.06,0) node[below] {\small $\lambda(b,cd)$};
   \draw (2.25,0) node[below] {\small$ \lambda(a,bcd)$};
   \draw (0,4) node[above] {\small $\lambda(a,b)$};
   \draw (1.06,4) node[above] {\small $\lambda(ab,c)$};
   \draw (2.25,4) node[above] {\small $\lambda(abc,d)$};
  \end{tikzpicture}}.
  \end{align}
    
    We impose the normalization conditions
    \[
      \nu(e,b,c) = \id_{\lambda(b,c)}
      ,
      \nu(a,e,c) = \id_{\lambda(a,c)}
      \text{, and }
      \nu(a,b,e) = \id_{\lambda(a,b)},
    \]
    which are consistent with the normalization \(\lambda(a,e) = \lambda(e,a) = \tu\).
  \item A family of isomorphisms $t(a,b) : \lambda(a,b) \to \lambda(b,a)$ making $\lambda$ into a symmetric 2-cocycle, normalized so that $t(e,a)=t(a,e)=\tu$ satisfying the \defemph{braided zesting conditions}
    \begin{align}
      \label{eq:braided zesting I}
      \begin{split}
     \nu(b,c,a) \circ \left ( \id_{\lambda(b,c)} \otimes t(a,bc) \right ) \circ  \nu (a,b,c)\\ = \left (t(a,c) \otimes \id_{\lambda(ca,b)} \right ) \circ \nu(b,a,c) \circ \left ( t(a,b) \otimes \id_{\lambda(ab,c)} \right ) 
     \end{split}
    \end{align}
    and
    \begin{align}
    \begin{split}
      \label{eq:braided zesting II}
      \omega(a,b;c) \left ( \nu(c,a,b)^{-1} \circ \left (\id_{\lambda(a,b)} \otimes t(ab,c) \right ) \circ  \nu(a,b,c)^{-1} \right )\\
      = \left (t(a,c) \otimes \id_{\lambda(ca,b)} \right ) \circ \nu(a,c,b)^{-1} \circ \left ( t(b,c) \otimes \id_{\lambda(a,bc)} \right ) 
      \end{split}
    \end{align}
    for all $a,b,c \in A$, or graphically
    \begin{align}
    \label{fig:braided zesting conditions}
    \scalebox{.9}{\begin{tikzpicture}[line width=1,baseline=1.9cm]
    \draw (0,0) node[below] {\small $\lambda(c,a)$}--(0,4) node[above] {\small $\lambda(a,b)$};
    \draw (1,0) node[below] {\small $\lambda(ca,b)$}--(1,4) node[above] {\small $\lambda(ab,c)$};
    \draw[fill=white] (-.25, .415) rectangle node [] {\small $\nu(b,c,a)$}(1.25,.915);
    \begin{scope}[yshift=2.66cm]
    \draw[fill=white] (-.25, .415) rectangle node [] {\small $\nu(a,b,c)$}(1.25,.915);
    \end{scope}
    \begin{scope}[xshift=1cm, yshift=1.33cm]
    \draw[fill=white] (0,.65) circle (.5) node [] {\scriptsize $t(a,bc)$};
    \end{scope}
    \end{tikzpicture}
    \quad = \quad
    \begin{tikzpicture}[line width=1,baseline=1.9cm]
    \draw (0,0) node[below] {\small $\lambda(c,a)$}--(0,4) node[above] {\small $\lambda(a,b)$};
    \draw (1,0) node[below] {\small $\lambda(ca,b)$}--(1,4) node[above] {\small $\lambda(ab,c)$};
    \draw[fill=white] (0,.65) circle (.45) node [] {\scriptsize $t(a,c)$};
    \begin{scope}[yshift=2.66cm]
    \draw[fill=white] (0,.65) circle (.45) node [] {\scriptsize $t(a,b)$};
    \end{scope}
    \begin{scope}[yshift=1.33cm]
    \draw[fill=white] (-.25, .415) rectangle node [] {\small $\nu(b,a,c)$}(1.25,.915);
    \end{scope}
    \end{tikzpicture}}
    \quad
    \text{and}
    \quad 
    \omega(a,b;c)\cdot
    \scalebox{.9}{\begin{tikzpicture}[line width=1,baseline=1.9cm]
    \draw (0,0) node[below] {\small $\lambda(c,a)$}--(0,4) node[above] {\small $\lambda(b,c)$};
    \draw (1,0) node[below] {\small $\lambda(ca,b)$}--(1,4) node[above] {\small $\lambda(a,bc)$};
    \draw[fill=white] (-.3, .415) rectangle node [] {\small $\nu(c,a,b)^{-1}$}(1.3,.915);
    \begin{scope}[yshift=2.66cm]
    \draw[fill=white] (-.3, .415) rectangle node [] {\small $\nu(a,b,c)^{-1}$}(1.3,.915);
    \end{scope}
    \begin{scope}[xshift=1cm, yshift=1.33cm]
    \draw[fill=white] (0,.65) circle (.5) node [] {\scriptsize $t(ab,c)$};
    \end{scope}
    \end{tikzpicture}
    \quad = \quad
    \begin{tikzpicture}[line width=1,baseline=1.9cm]
    \draw (0,0) node[below] {\small $\lambda(c,a)$}--(0,4) node[above] {\small $\lambda(b,c)$};
    \draw (1,0) node[below] {\small $\lambda(ca,b)$}--(1,4) node[above] {\small $\lambda(a,bc)$};
    \draw[fill=white] (0,.65) circle (.45) node [] {\scriptsize $t(a,c)$};
    \begin{scope}[yshift=2.66cm]
    \draw[fill=white] (0,.65) circle (.45) node [] {\scriptsize $t(b,c)$};
    \end{scope}
    \begin{scope}[yshift=1.33cm]
    \draw[fill=white] (-.3, .415) rectangle node [] {\small $\nu(a,c,b)^{-1}$}(1.3,.915);
    \end{scope}
    \end{tikzpicture}
    }.
    \end{align}
    These correspond to the condition denoted (BZ2) in \cite{Delaney2020}.
    In the more general context of \cite{Delaney2020} there is additional data \(j\) satisfying an additional condition (BZ1).
    For the universal grading group both can be dropped; see \cref{rem:dropping j}.

  \item A function \(\epsilon : A \to \kkm\) normalized so that \(\epsilon(e) = 1\) satisfying the  \defemph{twist zesting condition}
    \begin{equation}
      \label{eq:twist zesting epsilon}
      \frac{
        \epsilon(ab)
        }{
        \epsilon(a) \epsilon(b)
      }
      =
      \dim (\lambda(a,b))
    \end{equation}
    and \defemph{ribbon zesting condition}
    \begin{equation}
      \label{eq:ribbon zesting epsilon}
      \epsilon(a)^2 = 1
    \end{equation}
    for all \(a, b \in A\).
    \note{%
    In previous work the ribbon zesting parameter was stated in terms of the function \(f(a) = \epsilon(a) \langle t(a,a) \rangle \); we discuss this in \cref{sec:ribbon zesting conditions} and explain why \eqref{eq:twist zesting epsilon} always admits solutions.
    In many cases of interest (for example, any pseudo-unitary category) \(\dim (\lambda(a,b)) = 1\) uniformly, in which case the trivial solution \(\epsilon(a) = 1\) is the only choice that will result again in a pseudo-unitary category; we will revisit this fact throughout \cref{sec: background,sec:zest invariant}.
    }
\end{enumerate}
\end{definition}

The main result of \cite{Delaney2020} is the construction of a new category \(\mathcal{C}^{\zeta}\) from \(\mathcal{C}\) and \(\zeta\).
The new category \(\mathcal{C}^{\zeta}\) has the same objects as \(\mathcal{C}\) but its tensor product, braiding, and ribbon structure are all modified by \(\zeta\).
In particular, the tensor product of \(\mathcal{C}^{\zeta}\) is no longer strict, even if the original tensor product of \(\mathcal{C}\) is.

Here we use string diagrams to give a brief description of the structure isomorphisms of \(\mathcal{C}^{\zeta}\) when $\mathcal{C}$ is assumed to be strict.
Throughout \(X_a, Y_b, Z_c\) are homogeneous objects of degree \(a, b, c\), respectively.

\subsubsection{Zested monoidal structure and duals} 
\label{sec: zested monoidal and duals}
The zested tensor product is
\[
  X_{a} \zestts Y_{b} \defeq X_{a} \otimes Y_{b} \otimes \lambda(a,b)
\]
This is no longer strict, and the associator
\[
  (X_{a} \zestts Y_{b}) \zestts Z_{c}
  \overset{\alpha^{\zeta}}{\longrightarrow}
  X_{a} \zestts (Y_{b} \zestts Z_{c})
\]
is given by
\begin{align*}
  \alpha^{\zeta}_{X_a,Y_b,Z_c} = 
  \begin{tikzpicture}[line width=1, baseline = 1.125*.75cm, scale=.75, slblue]
    \draw[black] (1,0)--(1,2.25);
    \draw (4,0)--(4,2.25);
    \draw[black] (2,0)--(2,1.5);
    \draw (3,0)--(3,1.5);
    \draw (4,0)--(4,1.5);
    \draw[fill=white] (2.75,0 + .45) rectangle node {$\scriptstyle \nu(a,b,c)$} (4.25,0 + 1.05); 
    \draw (3,1.5) \br (2,2.25);
    \draw[white, line width= 10] (2,1.5) \br (3,2.25);
    \draw[black] (2,1.5) \br (3,2.25);
    \draw[black] (0,0) node[below] {$\parcen{X_a}$} -- (0,2.25) node[above] {$\parcen{X_a}$};
    \draw[black] (1,0) node[below] {$\parcen{Y_b}$};
    \draw[black] (2,0) node[below] {$Z_c$};
    \draw (3,0) node[below] {$\lambda(b,c)$};
    \draw (4.5,0) node[below] {$\lambda(a,bc)$};
    \draw[black] (1,2.25) node[above] {$\parcen{Y_b}$};
    \draw (2,2.25) node[above] {$\lambda(a,b)$};
    \draw[black] (3,2.25) node[above] {$\parcen{Z_c}$};
    \draw (4.25,2.25) node[above] {$\lambda(ab,c)$};
  \end{tikzpicture}
\end{align*}
Here we draw the extra morphisms added during zesting in \textcolor{slblue}{blue} for emphasis.

To express how the duality changes, it will be convenient to abbreviate
\begin{align}
    \lambda(a)
    &\defeq
    \lambda(a, a^{-1})
    \\
    \nu(a)
    &\defeq
    \nu(a, a^{-1}, a)
\end{align}

Since \(\lambda(a,e) = \tu\) we can view \(\nu(a)\) as an isomorphism \(\lambda(a) \to \lambda(a^{-1})\).
Using this notation the (left) dual in the zested category \(\mathcal{C}^{\zeta}\) is
\begin{equation}
  \label{eq:zested dual}
  \overline{X}_{a} = X_a^* \otimes \lambda(a)^*
\end{equation}
with (left) evaluation and coevaluation morphisms 
\begin{align*}
\coev^{\zeta}_{X}= \frac{1}{\dim(\lambda(a))}
    \begin{tikzpicture}[line width=1, scale=1.4, baseline=-.75*1.25cm]
    \draw[
    looseness=2] (1,-1) node[below] {$X_a$} to [out=90, in=90] (2,-1);
    \draw (2,-1) node[below] {$X_a^*$};
    \begin{scope}[xshift=1.75cm,slblue]
    \draw[
    looseness=2] (1,-1) node[below] {$\lambda(a)^*$} to [out=90, in=90] (2,-1);
    \draw (2,-1) node[below] {$\lambda(a)$};
    \end{scope}
    \end{tikzpicture} \quad & \quad  
    \ev^{\zeta}_{X}=
	\begin{tikzpicture}[line width=1, scale=1.4, baseline=-.25*1.25cm]
	\begin{scope}[xshift=.5cm,slblue]
   	 \draw[
   	 looseness=2] (1,0) node [above] {$\lambda(a)^*$} to [out=-90, in=-90] (2,0);
   	 \draw (2.125,0)  node[above] {$\lambda(a^{-1})$};
     	\draw[fill=white] (1.65,-.55) rectangle node {$\scriptstyle \nu(a)^{-1}$} (2.3,-.15);
 	\end{scope}
	\draw[line width=10,white, looseness=2] (1,0) to [out=-90, in=-90] (2,0); 
   	\draw[
   	looseness=2] (1,0) node[above] {$\parcen{X_a^*}$} to [out=-90, in=-90] (2,0) node[above] {$\parcen{X_a}$};
   	\end{tikzpicture} 
\end{align*}
The normalization factor of $\dim(\lambda(a)) = \dim(\lambda(a))^{-1}$ is required to ensure that the zig-zag axioms hold, i.e.~so that $\mathcal{C}^\zeta$ is (left) rigid. We note that this factor is \(1\) when \(\mathcal{C}\) is pseudo-unitary.

We will use the right duals and right evaluation and coevaluation morphisms induced by the zested pivotal structure.
Following \cite[Proposition 5.5]{Delaney2020}, we take this structure to be induced by the zested ribbon structure via the Drinfeld isomorphism. Although the zested ribbon structure is deferred to \cref{sec: zested braiding and twists}, and the definition of the zested Drinfeld isomorphism is given in \cite{Delaney2020}, when $\mathcal{C}$ is strict pivotal one can check that the induced pivotal structure \(\varphi_{X}^{\zeta} \colon X \to \overline{\overline{X}}\) is given by
\[
  \varphi_{X}^\zeta := 
  \frac{\epsilon(a)}{\dim(\lambda(a))}\quad 
  \begin{tikzpicture}[line width=1, scale=2, baseline= -.5cm*2]
    \begin{scope}[slblue]
      \draw[
      looseness=2] (1,-1) node [below] {$\lambda(a)$} to [out=90, in=90] (2,-1);
      \draw (2.125,-1) node[below] {$\lambda(a^{-1})$};
      \draw[fill=white]  (.8,-.5) rectangle node {$\scriptstyle \nu(a)^{-1}$}  (1.35,-.8);
    \end{scope}
    \draw[white, line width=10] (1.5,0)--(1.5,-1);
    \draw[%
    ] (1.5,0) node[above] {$X_a$} --(1.5,-1) node[below] {$X_a$} ;
  \end{tikzpicture}.
\]

This zested pivotal structure induces right evaluations and coevaluations
\begin{align*}
  (\coev'_{X})^{\zeta} 
  & =
  \left( \id_{\overline{X}} \zestts \left( \varphi^{\zeta}_{X}\right )^{-1} \right) \circ \coev_{\overline{X}}^{\zeta}
  \\
  (\ev'_{X})^{\zeta} 
  & =
  \ev_{\overline{X}}^{\zeta} \circ \left ( \varphi^{\zeta}_{X} \zestts \id_{\overline{X}} \right)
\end{align*}
or graphically 
\begin{align*}
  (\coev'_X)^{\zeta} &= \frac{1}{\epsilon(a) \dim(\lambda(a))}
  \begin{tikzpicture}[line width=1, scale=1.4, baseline=-.75*1.25cm]
    \begin{scope}[xshift=.5cm,slblue]
      \draw[%
      looseness=2] (1,-1) node [below] {$\lambda(a)^*$} to [out=90, in=90] (2,-1);
      \draw[fill=white] (1.75,-.8) rectangle node {$\scriptstyle \nu(a)$} (2.25,-.4);
      \draw (2.125,-1) node[below] {$\lambda(a^{-1})$};
    \end{scope}
    \draw[line width=10,white, looseness=2] (1,-1) to [out=90, in=90] (2,-1);
    \draw (2,-1) node[below] {$\parcen{X_a}$};
    \draw[%
    looseness=2] (1,-1) to [out=90, in=90] (2,-1);
    \draw (.875,-1) node[below] {$X_a^*$};
    \draw (2,-1) node[below] {$\parcen{X_a}$};
  \end{tikzpicture}
  \\
  (\ev'_X)^{\zeta} &= \epsilon(a)
  \begin{tikzpicture}[line width=1, scale=1.4, baseline=-.25*1.25cm]
    \draw[%
    looseness=2] (1,0) node[above] {$X_a$} to [out=-90, in=-90] (2,0);
    \draw (2,0) node[above] {$X_a^*$};
    \begin{scope}[xshift=1.75cm,slblue]
      \draw[%
      looseness=2] (1,0) node[above] {$\lambda(a)^*$} to [out=-90, in=-90] (2,0);
      \draw (2,0) node[above] {$\lambda(a)$};
    \end{scope}
  \end{tikzpicture}.
\end{align*}

The benefit of working with the cups and caps induced by the pivotal structure is that they can then be paired with the original evaluation/coevaluation morphisms to compute the zested trace.
For an endomorphism \(f \in \operatorname{Hom}(X_a,X_a)\) of a homogeneous object we have
\begin{align}
\label{eq: left zested trace}
\text{Tr}_L^\zeta(f) = \ev^\zeta_X \circ \left ( \id_{\overline{X}} \overset{\zeta}{\otimes} f \right ) \circ (\coev'_X)^{\zeta} =  \frac{\dim(\lambda(a^*))}{\epsilon(a)} \text{Tr}_L(f)
\end{align}
and 
\begin{align}
\label{eq: right zested trace}
\text{Tr}_R^\zeta(f) = (\ev'_X)^\zeta \circ \left ( f   \overset{\zeta}{\otimes} \id_{\overline{X}} \right ) \circ (\coev_X)^{\zeta} =  \frac{\epsilon(a)}{\dim(\lambda(a))} \text{Tr}_R(f)
\end{align}
Since \(\dim(\lambda(a))=\dim(\lambda(a)^*)\) the ribbon zesting condition \(\epsilon(a)^2=1\) confirms that the zested left and right traces are equal; this must be the case because \(\mathcal{C}^\zeta\) is ribbon, hence spherical.

\subsubsection{Zested braiding and ribbon isomorphisms}
\label{sec: zested braiding and twists}
The braiding 
\[
  X_a \zestts Y_b \overset{\beta^{\zeta}} \longrightarrow Y_b \zestts X_a
\]
    is given by
\begin{align*}
\beta^{\zeta}_{X_a,Y_b} =
 \begin{tikzpicture}[line width=1, baseline=12.5, scale=.75] 
    \draw (0,0) node[below] {$\parcen{Y}$} \br (1,1.5) node[above] {$\parcen{Y}$};
    \draw[white, line width=10] (1,0) \br (0,1.5);
    \draw (1,0) node[below] {$\parcen{X}$} \br (0,1.5) node[above] {$\parcen{X}$};
    \begin{scope}[xshift=2.25cm,slblue]
    \onecirc{\lambda(b,a)}{\lambda(a,b)}{\scriptstyle t(a,b)};
    \end{scope}
    \end{tikzpicture}
\end{align*}
and the twists are defined by
\[
  \theta_{X_a}^{\mathcal{C}^\zeta}
  \defeq
  \theta_{X_a}^{\mathcal{C}}
 \textcolor{slblue}{ t(a) \epsilon(a)}
\]
where \(t(a) := \langle t(a,a) \rangle \in \kkm\).

\begin{remark}
  \label{rem:dropping j}
  The paper \cite{Delaney2020} defines braided zestings for arbitrary faithful gradings \(G\) of \(\mathcal{C}\), not just the universal grading \(A\).
  At that level of generality braided zesting requires the additional data of a map \(j : G \to \operatorname{Aut}_{\otimes}(\Id_{\mathcal{C}})\), which amounts to the data of isomorphisms $j_a: Y \to Y$ on simple objects $Y$ for all $a \in G$.   
  (This \(j\) is unrelated to \(\zestmini{}\).)
  It has to satisfy a compatibility condition (BZ1) with respect to the double braiding.
  The more general definition of \(\beta^{\zeta}_{X_{a}, Y_{b}}\) includes an adjustment by the induced isomorphism \(j_{a}\) on \(Y\).
  For the \emph{universal} grading \(A\) the value \(j_{a}(Y)\) on a simple object \(Y\) depends only on the \(A\)-degree of \(Y\), and in fact each \(j_{a}\) can be identified with a homomorphism \(A \to \kkm\) by \cite[Proposition 4.14.3]{EGNO2015}.
  This means that (BZ1) is automatically satisfied and by \cite[Corollary 4.8]{Delaney2020} one can assume without loss of generality that \(j\) is trivial.
\end{remark}

\subsection{Zesting and the Reshetikhin-Turaev construction}
\label{sec:rt invariants and factorization}

Next we develop a reformulation of the zesting construction from the point of view of the Reshetkhin-Turaev topological quantum field theories (TQFTs) associated to \(\mathcal{C}\) and \(\mathcal{C}^\zeta\), which will be an essential ingredient in the later sections.

\begin{definition}
  For a ribbon category \(\mathcal{C}\), a \defemph{\(\mathcal{C}\)-colored tangle} is an oriented, framed tangle \(T\) along with a labeling of its connected components by objects of \(\mathcal{C}\).
  (More generally we can allow \(\kk\)-linear combinations of objects of \(\mathcal{C}\), as in the definition of the Reshetikhin-Turaev \(3\)-manifold invariants recalled in Section \ref{sec:surgery colors and A-mfld invariants}.)
  When \(\mathcal{C}\) is graded by an abelian group \(A\) (which we always assume is the universal grading group of \(\mathcal{C}\)), we say the coloring is \defemph{homogeneous} if each \(X_i\) is.

  \(\mathcal{C}\)-colored tangles are a ribbon category and the Reshetikhin-Turaev construction \cite{Reshetikhin1990} defines a ribbon functor \(\linkinvname\) from the category of \(\mathcal{C}\)-colored tangles to \(\mathcal{C}\).
  We call its value \(\linkinv{T}{\mathcal{C}}\) the \defemph{Reshetikhin-Turaev invariant} of \(T\).
\end{definition}

Suppose \(\mathcal{C}^{\zeta}\) is a ribbon zesting of \(\mathcal{C}\) with data \(\zeta = (\lambda, \nu, t, \epsilon)\).
Because the objects of \(\mathcal{C}^{\zeta}\) are the same as those of \(\mathcal{C}\) we can view any \(\mathcal{C}\)-colored tangle \(T\) as a \(\mathcal{C}^{\zeta}\)-colored tangle.
Assuming \(T\) is homogeneous, our goal in this section is to explain how to factor \(\linkinv{T}{\mathcal{C}^{\zeta}}\) into the original invariant \(\linkinv{T}{\mathcal{C}}\) and a morphism \(\zestmini{T}\) of \(\Inv(\mathcal{C}_e)\).
Later we will show that \(\zestmini{T}\) is in fact an invariant of \(T\) as an \(A\)-colored tangle.

While the objects of \(\mathcal{C}^{\zeta}\) are the same as \(\mathcal{C}\) the tensor product is no longer strictly associative.
To make the graphical calculus work in this context we take the convention that (unless otherwise noted) all tensor-products are left-associated.
This means that if \(T\) is a \(\mathcal{C}\)-colored tangle with
\[
  \linkinv{T}{\mathcal{C}} : X_1 \otimes \cdots \otimes X_n \to Y_1 \otimes \cdots \otimes Y_m
\]
then the invariant of the tangle with respect to the zested category is
\[
  \linkinv{T}{\mathcal{C}^{\zeta}} : ( \cdots ( (X_1 \zestts X_2)  \zestts X_3) \zestts \cdots ) 
  \to
  ( \cdots ( (Y_1 \zestts Y_2)  \zestts Y_3) \zestts \cdots ).
\]
If the coloring is homogeneous with \(|X_i| = a_i\) then 
\begin{equation}
  \label{eq:left associated zest tensor}
  \begin{aligned}
    &( \cdots ( (X_1 \zestts X_2)  \zestts X_3) \zestts \cdots )
    \\
    &=
    X_1 \otimes X_2 \otimes \lambda(a_1, a_2) \otimes X_3 \otimes \lambda(a_1 a_2, a_3) \otimes \cdots \otimes X_n \otimes \lambda(a_1 a_2 \cdots a_{n-1}, a_n)
  \end{aligned}
\end{equation}
is the original tensor product of the \(X_i\) interspersed with values of the \(\Inv(\mathcal{C}_e)\)-valued \(2\)-cocycle \(\lambda\).
Here we are using the assumption that our original category  \(\mathcal{C}\) is strict.
We are also assuming that all strands are oriented downwards: upward pointing strands have additional factors of \(\lambda\), as discussed in \cref{sec:zesting summary}.

In terms of link invariants we can think of \(\mathcal{C}^{\zeta}\) as being obtained as a sort of product of \(\mathcal{C}\) and an invertible piece.
However, as in \cref{eq:left associated zest tensor} the ``product'' is obtained by interweaving the original objects \(X_i\) and invertible objects \(\lambda(a_j, a_k)\) from the zesting data.

\begin{theorem}
  \label{thm:factorization existence}
  Let \(T\) be a homogeneous \(\mathcal{C}\)-colored tangle with Reshetikhin-Turaev invariant
  \(
    \linkinv{T}{\mathcal{C}}.
  \)
  Then there is a morphism
  \(
  j_{\zeta}(T)
  \)
  of \(\Inv(\mathcal{C}_e)\) so that the zested Reshetikhin-Turaev invariant factors up to braiding as
  \begin{equation}
    \label{eq:warp factorization motivation}
    \linkinv{T}{\mathcal{C}^{\zeta}}
    =
     \begin{tikzpicture}[line width=1, baseline =.85*2cm, scale=.85]
    \draw[slblue, looseness=.5] (1.5,0) \br (4.75,1.5);
    \draw[slblue, looseness=.5] (2.5,0) \br (5.75,1.5);
    \draw[slblue, looseness=.5] (4.5,0) \br (7.75,1.5);
    \draw[slblue, looseness=.5] (4.75,2.5) \br (1.5,4);
    \draw[slblue, looseness=.5] (5.75,2.5) \br (2.5,4);
    \draw[slblue, looseness=.5] (7.75,2.5) \br (4.5,4);
    \draw[slblue,fill=white] (4.5, 1.5) rectangle node {$\zestmini{T}$} (8,2.5);
    \foreach \x in {1,2,4}{
   	\draw[white, line width=10] (\x,0)--(\x,4);
    	\draw(\x,0)--(\x,4);
    }
    \draw[fill=white] (.75,1.5) rectangle node {$\mathcal{F}_{\mathcal{C}}(T)$} (4.25,2.5);
    \draw (1,4) node[above] {$\scriptstyle X_1$};
    \draw (2,4) node[above] {$\scriptstyle X_2$};
    \draw (4,4) node[above] {$\scriptstyle X_n$};
    \draw (3,1.25) node {$\cdots$};
    \draw (3,2.75) node {$\cdots$};
    \draw (3.25, 4) node[above] {$\scriptstyle \cdots$};
    \draw[slblue] (6.5,1.25) node {$\cdots$};
    \draw[slblue] (6.5,2.75) node {$\cdots$};
    \end{tikzpicture}
  \end{equation}
\end{theorem}
We note that here \(T\) is a tangle that may be arbitrarily oriented.
In \cref{sec:zest invariant} we show that \(\zestmini{T}\) is in fact an invariant of the underlying colored tangle, arising from a Reshetikhin-Turaev-like functor
\(\zest{}\).
In our later notation \(\zestmini{T} = \zest{e, T}\).

To prove this theorem we want to use some algebraic properties of the almost-tensor product in \cref{eq:warp factorization motivation}.

\begin{definition}
  Suppose \(f : X_1 \cdots X_n \to Y_{1} \cdots Y_m\) and  \(j : \lambda_{1} \cdots \lambda_{n} \to \mu_{1} \cdots \mu_{m}\) are morphisms of \(\mathcal{C}\).
  Then their \defemph{warp product}%
  \note{%
    In weaving, the ``warp'' refers to the vertical yarns through which the ``weft'' is woven.
    The ``warp product'' involves both the warp (the interweaving of the source and target objects of the two morphisms) and a tensor product.
  }
  is the morphism
  \[
    \begin{aligned}
      f \warp j
      \colon
      &X_1 \otimes \lambda_1 \otimes X_2 \otimes \lambda_2 \otimes X_3 \otimes \lambda_3 \otimes \cdots \otimes X_n \otimes \lambda_n
      \\
      &\to
      Y_1 \otimes \mu_1 \otimes Y_2 \otimes \mu_2 \otimes Y_3 \otimes \mu_3 \otimes \cdots \otimes Y_m \otimes \mu_m
    \end{aligned}
  \]
  defined by
  \begin{equation}
    f \warp j \defeq 
    \begin{tikzpicture}[line width=1, baseline =.85*2cm, scale=.85]
    \draw[slblue, looseness=.5] (1.5,0) \br (4.75,1.5);
    \draw[slblue, looseness=.5] (2.5,0) \br (5.75,1.5);
    \draw[slblue, looseness=.5] (4.5,0) \br (7.75,1.5);
    \draw[slblue, looseness=.5] (4.75,2.5) \br (1.5,4);
    \draw[slblue, looseness=.5] (5.75,2.5) \br (2.5,4);
    \draw[slblue, looseness=.5] (7.75,2.5) \br (4.5,4);
    \draw[slblue,fill=white] (4.5, 1.5) rectangle node {$j$} (8,2.5);
    \foreach \x in {1,2,4}{
   	\draw[white, line width=10] (\x,0)--(\x,4);
    	\draw(\x,0)--(\x,4);
    }
    \draw[fill=white] (.75,1.5) rectangle node {$f$} (4.25,2.5);
    \draw (1,4) node[above] {$\scriptstyle X_1$};
    \draw (2,4) node[above] {$\scriptstyle X_2$};
    \draw (4,4) node[above] {$\scriptstyle X_n$};
    \draw (1,0) node[below] {$\scriptstyle Y_1$};
    \draw (2,0) node[below] {$\scriptstyle Y_2$};
    \draw (4,0) node[below] {$\scriptstyle Y_m$};
    \draw (3,1.25) node {$\cdots$};
    \draw (3,2.75) node {$\cdots$};
    \draw (3.25, 0) node[below] {$\scriptstyle \cdots$};
    \draw (3.25, 4) node[above] {$\scriptstyle \cdots$};
    \draw[slblue] (1.5,4) node[above] {$\scriptstyle \lambda_1$};
    \draw[slblue] (2.5,4) node[above] {$\scriptstyle \lambda_2$};
    \draw[slblue] (4.75,4) node[above] {$\scriptstyle \lambda_m$};
    \draw[slblue] (1.5,0) node[below] {$\scriptstyle \phantom{Y}\mu_1\phantom{Y}$};
    \draw[slblue] (2.5,0) node[below] {$\scriptstyle \phantom{Y}\mu_2\phantom{Y}$};
    \draw[slblue] (4.675,0) node[below] {$\scriptstyle \phantom{Y}\mu_m\phantom{Y}$};
    \draw[slblue] (6.5,1.25) node {$\cdots$};
    \draw[slblue] (6.5,2.75) node {$\cdots$};
    \end{tikzpicture}
  \end{equation}

  Note that this definition depends on a choice of factorization of the sources and targets of \(f\) and \(j\); in practice this will always be clear from the context.%
  \note{%
    We could also consider a slightly more general warp product with an object \(\lambda_0\) inserted to the left of \(X_{1}\).
  }
  Here we have drawn the diagram in two colors for clarity; later we will always choose \(j\) to lie in an invertible subcategory of \(\mathcal{C}\) but this is not needed for the definition to make sense.
\end{definition}
We can now restate \cref{thm:factorization existence}.
Our claim is that there is a morphism \(\zestmini{T}\) of \(\Inv(\mathcal{C}_e)\) with 
\[
  \linkinv{T}{\mathcal{C}^{\zeta}}
  =
  \linkinv{T}{\mathcal{C}}
  \warp
  \zestmini{T}.
\]
Before we prove \cref{thm:factorization existence} we will need the following observation. 

\begin{proposition}
  \label{thm:warp product composition compatibility}
Warp products are compatible with composition: when \(f_1, f_2, j_1, j_2\) are appropriately composable,
\[
  (f_2 \warp j_2) \circ (f_1 \warp j_1)
  =
  (f_2 \circ f_1) \warp (j_2 \circ j_1)
  \qedhere
\]
\end{proposition}
\begin{proof}
  Examine the picture:
  
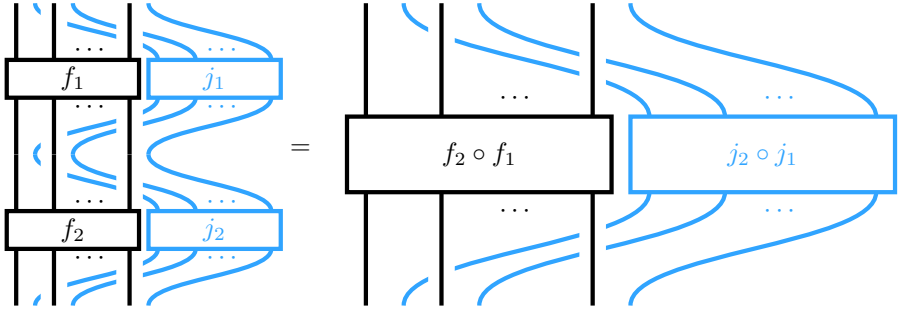
\begin{align*}
    \begin{tikzpicture}[line width=1, baseline =0cm, scale=.5]
      \draw[slblue, looseness=.5] (1.5,0) \br (4.75,1.5);
      \draw[slblue, looseness=.5] (2.5,0) \br (5.75,1.5);
      \draw[slblue, looseness=.5] (4.5,0) \br (7.75,1.5);
      \draw[slblue, looseness=.5] (4.75,2.5) \br (1.5,4);
      \draw[slblue, looseness=.5] (5.75,2.5) \br (2.5,4);
      \draw[slblue, looseness=.5] (7.75,2.5) \br (4.5,4);
      \draw[slblue,fill=white] (4.5, 1.5) rectangle node {$j_1$} (8,2.5);
      \foreach \x in {1,2,4}{
        \draw[white, line width=10] (\x,0)--(\x,4);
        \draw(\x,0)--(\x,4);
      }
      \draw[fill=white] (.75,1.5) rectangle node {$f_1$} (4.25,2.5);
      \draw (3,1.25) node {$\cdots$};
      \draw (3,2.75) node {$\cdots$};
      \draw[slblue] (6.5,1.25) node {$\cdots$};
      \draw[slblue] (6.5,2.75) node {$\cdots$};
      \begin{scope}[yshift=-4cm]
        \draw[slblue, looseness=.5] (1.5,0) \br (4.75,1.5);
        \draw[slblue, looseness=.5] (2.5,0) \br (5.75,1.5);
        \draw[slblue, looseness=.5] (4.5,0) \br (7.75,1.5);
        \draw[slblue, looseness=.5] (4.75,2.5) \br (1.5,4);
        \draw[slblue, looseness=.5] (5.75,2.5) \br (2.5,4);
        \draw[slblue, looseness=.5] (7.75,2.5) \br (4.5,4);
        \draw[slblue,fill=white] (4.5, 1.5) rectangle node {$j_2$} (8,2.5);
        \foreach \x in {1,2,4}{
          \draw[white, line width=10] (\x,0)--(\x,4);
          \draw(\x,0)--(\x,4);
        }
        \draw[fill=white] (.75,1.5) rectangle node {$f_2$} (4.25,2.5);
        \draw (3,1.25) node {$\cdots$};
        \draw (3,2.75) node {$\cdots$};
        \draw[slblue] (6.5,1.25) node {$\cdots$};
        \draw[slblue] (6.5,2.75) node {$\cdots$};
      \end{scope}
    \end{tikzpicture}
    = \quad
    \begin{tikzpicture}[line width=1, baseline =2cm, scale=1]
      \draw[slblue, looseness=.5] (1.5,0) \br (4.75,1.5);
      \draw[slblue, looseness=.5] (2.5,0) \br (5.75,1.5);
      \draw[slblue, looseness=.5] (4.5,0) \br (7.75,1.5);
      \draw[slblue, looseness=.5] (4.75,2.5) \br (1.5,4);
      \draw[slblue, looseness=.5] (5.75,2.5) \br (2.5,4);
      \draw[slblue, looseness=.5] (7.75,2.5) \br (4.5,4);
      \draw[slblue,fill=white] (4.5, 1.5) rectangle node {$j_2 \circ j_1$} (8,2.5);
      \foreach \x in {1,2,4}{
        \draw[white, line width=10] (\x,0)--(\x,4);
        \draw(\x,0)--(\x,4);
      }
      \draw[fill=white] (.75,1.5) rectangle node {$f_2\circ f_1$} (4.25,2.5);
      \draw (3,1.25) node {$\cdots$};
      \draw (3,2.75) node {$\cdots$};
      \draw[slblue] (6.5,1.25) node {$\cdots$};
      \draw[slblue] (6.5,2.75) node {$\cdots$};
    \end{tikzpicture} \quad .
  \end{align*}
\end{proof}

\begin{proof}[Proof of \cref{thm:factorization existence}]
  This follows from the same argument as \cite[Theorem 2]{Delaney2021}, which is a version of our result for links.
  We can choose a Morse diagram for \(T\) and decompose it into a vertical product of elementary tangle generators (caps, cups, and braidings)
  \[
    T = T_k \circ \cdots \circ T_1.
  \]
  
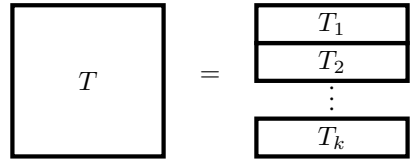
\begin{align*}
    \label{fig:vertical tangle decomposition}
    \begin{tikzpicture}[line width=1,scale=.5, baseline=.5*2cm]
    \draw[fill=white] (0,0) rectangle node {$T$} (4,4);
    \end{tikzpicture}
    \quad
    =
    \quad
    \begin{tikzpicture}[line width=1,scale=.5, baseline=.5*2cm]
    \draw[fill=white] (0,0) rectangle node {$T_k$} (4,1);
    \draw[fill=white] (0,2) rectangle node {$T_2$} (4,3);
    \draw[fill=white] (0,3) rectangle node {$T_1$} (4,4);
    \draw (2,1.75) node {$\vdots$};
    \end{tikzpicture}
  \end{align*}
  Since we are only decomposing along vertical composition \(\circ\) (not along horizontal composition \(\otimes\)) our generators include diagrams like
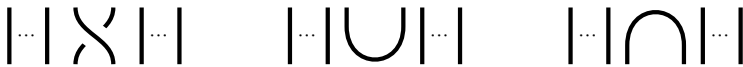
\begin{align*}
  \begin{tikzpicture}[line width=1, baseline=.75cm, scale=.5]
    \draw (0,0)--(0,1.5);
    \draw (1,0)--(1,1.5);
    \draw (0.5,.75) node {$\scriptstyle \cdots$};
    \begin{scope}[xshift=1.75cm]
    \draw (0,0) \br (1,1.5);
    \draw[white, line width=10] (1,0) \br (0,1.5);
    \draw (1,0) \br (0,1.5);
    \end{scope}
    \begin{scope}[xshift=3.5cm]
    \draw (0,0)--(0,1.5);
    \draw (1,0)--(1,1.5);
    \draw (0.5,.75) node {$\scriptstyle \cdots$};
    \end{scope}
    \end{tikzpicture}
    \qquad \qquad
    \begin{tikzpicture}[line width=1, baseline=.75cm, scale=.5]
    \draw (0,0)--(0,1.5);
    \draw (1,0)--(1,1.5);
    \draw (0.5,.75) node {$\scriptstyle  \cdots$};
    \begin{scope}[xshift=1.75cm,yscale=-1,yshift=-1.5cm]
    \draw[looseness=2] (-.25,0)--(-.25,.5) to [out=90, in=90] (1.25,.5) --(1.25,0);
    \end{scope}
    \begin{scope}[xshift=3.5cm]
    \draw (0,0)--(0,1.5);
    \draw (1,0)--(1,1.5);
    \draw (0.5,.75) node {$\scriptstyle  \cdots$};
    \end{scope}
  \end{tikzpicture}
    \qquad \qquad
    \begin{tikzpicture}[line width=1, baseline=.75cm, scale=.5]
    \draw (0,0)--(0,1.5);
    \draw (1,0)--(1,1.5);
    \draw (0.5,.75) node {$\scriptstyle  \cdots$};
    \begin{scope}[xshift=1.75cm]
    \draw[looseness=2] (-.25,0)--(-.25,.5) to [out=90, in=90] (1.25,.5) --(1.25,0);
    \end{scope}
    \begin{scope}[xshift=3.5cm]
    \draw (0,0)--(0,1.5);
    \draw (1,0)--(1,1.5);
    \draw (0.5,.75) node {$\scriptstyle  \cdots$};
    \end{scope}
  \end{tikzpicture}
\end{align*}
  Because \(\warp\) is compatible with composition it suffices to show that
  \[
    \linkinv{T_i}{\mathcal{C}^{\zeta}} = \linkinv{T_i}{\mathcal{C}} \warp \zestmini{T_i}
  \]
  for each generator \(T_i\).

  Let \(B_i\) be the tangle with source \((X_1, \dots, X_n)\) with a braid generator on strands \(i, i+1\).
  Because of our left-association convention we can easily see from the picture 
 
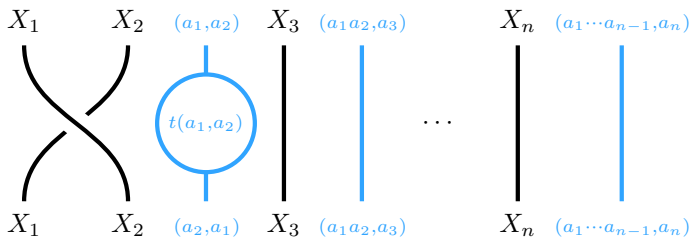
\begin{align*}
    \begin{tikzpicture}[line width=1, baseline = 0, scale=1.375]
    \draw[black] (2.75,-.75) node[below] {$\parcen{X_n}$}--(2.75,.75)node[above] {$\parcen{X_n}$};
    \draw[slblue] (-.25,-.75) node[below] {$\parcen{\scriptstyle (a_2,a_1)}$} --(-.25,.75) node[above] {$\parcen{\scriptstyle(a_1,a_2)}$};
    \draw [slblue, fill=white] (-.25,0)  circle (.47) node {$\scriptstyle t(a_1,a_2)$};
    \draw[black] (-2,-.75) node[below] {$\parcen{X_1}$} \br (-1,.75) node[above] {$\parcen{X_2}$};
     \draw[white, line width=10] (-1,-.75) \br (-2,.75);
         \draw[black] (-1,-.75) node[below] {$\parcen{X_2}$} \br (-2,.75) node[above] {$\parcen{X_1}$};
    \draw[slblue] (1.25,-.75) node[below] {$\parcen{\scriptstyle (a_1a_2,a_3)}$} -- (1.25,.75) node[above] {$\parcen{\scriptstyle (a_1a_2,a_3)}$};
    \draw[slblue] (3.75,-.75) node[below] {$\parcen{\scriptstyle(a_1\cdots a_{n-1},a_n)}$} -- (3.75,.75) node[above] {$\parcen{\scriptstyle(a_1\cdots a_{n-1},a_n)}$};
    \draw[black] (.5,-.75) node[below] {$\parcen{X_3}$} -- (.5,.75) node[above] {$\parcen{X_3}$};
    \draw[black] (2, 0) node {$\cdots$};
    \end{tikzpicture}
  \end{align*}  
that
  \(
    \linkinv{B_1}{\mathcal{C}^{\zeta}} = \linkinv{B_1}{\mathcal{C}} \warp \zestmini{B_1}
  \)
  for the morphism
  \begin{align*}
    \zestmini{B_1} \colon
    &\lambda(a_1, a_2) \otimes \lambda(a_1a_2, a_3) \otimes \cdots \otimes \lambda(a_1 \cdots a_{n-1}, a_n)
    \\
    &\to
    \lambda(a_2, a_1) \otimes \lambda(a_1a_2, a_3) \otimes \cdots \otimes \lambda(a_1 \cdots a_{n-1}, a_n)
    \\
    \zestmini{B_1} 
    &=
    t(a_1, a_2) \otimes \id.
  \end{align*}
  
 For the parenthesization
  \[
    \left( \cdots \left(X_1 \zestts  \cdots \zestts \left(X_i \zestts X_{i+1}\right) \right) \zestts \cdots \zestts X_n\right)
  \]
  the braiding on strands \(i, i+1\) is similarly given by the braiding of \(\mathcal{C}\) and \(t(a_i, a_{i+1})\), but our convention is that all tensor products are left-associated.
  Therefore \(\linkinv{B_i}{\mathcal{C}^{\zeta}}\) is given by
  
  
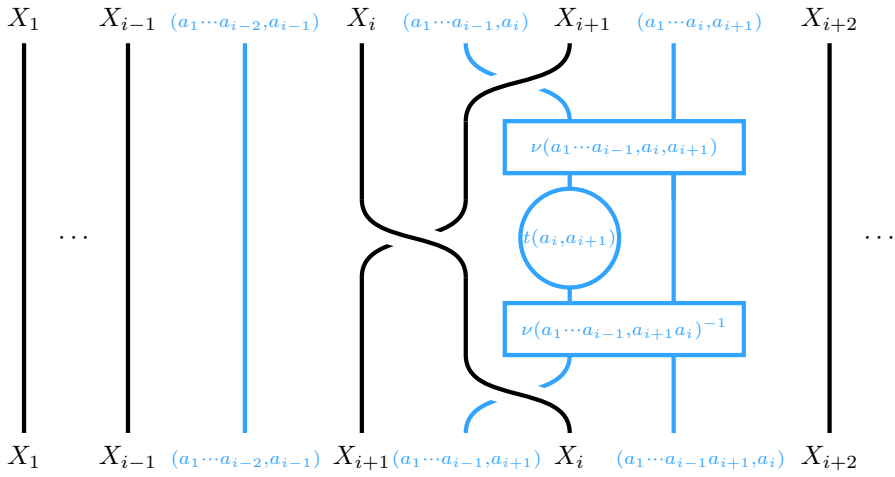
\begin{align*}
    \begin{tikzpicture}[line width=1, baseline = 0, scale=1.375, slblue]
    \draw[black] (1,0)--(1,0 + 1.5);
    \draw (4,0)--(4,0 + 1.5);
    \draw[black] (2,-.75)--(2,-1.5);
    \draw[black] (1,-2.25)--(1,-.75);
    \draw (4,-2.25)-- (4,-.75);
    \draw[black] (2,0)--(2,0 + .75);
    \draw (3,-1.5)--(3,0 + .75);
    \draw (4,-.75)--(4,0 + .75);
    \draw[fill=white] (2.375,0 + -1.5) rectangle node {$\scriptstyle \nu(a_1\cdots a_{i-1},a_{i+1}a_i)^{-1}$} (4.675,-1);
    \draw[fill=white] (2.375,0 + .25) rectangle node {$\scriptstyle \nu(a_1\cdots a_{i-1},a_i,a_{i+1})$} (4.675,0 + .75); 
    \draw (3,0+.75)  to [out=90, in=-90] (2,0 + 1.5);
    \draw[white, line width=10] (2,0+.75)  to [out=90, in=-90] (3,0 + 1.5);
    \draw[black] (2,0+.75) to [out=90, in=-90] (3,0 + 1.5);
    \draw[black] (1,-.75) \br (2,0);
    \draw[white, line width=10] (2,-.75) \br (1,0);
    \draw[black] (2,-.75) \br (1,0);
    \draw [fill=white] (3,-.75/2)  circle (.475) node {$\scriptstyle t(a_i,a_{i+1})$};
    \draw (2,-2.25) \br (3,-1.5);
    \draw[white, line width= 10] (3,-2.25) \br (2,-1.5);
    \draw[black] (3,-2.25) \br (2,-1.5);
    \draw[black] (-2.25,-2.25) node[below] {$\parcen{X_1}$} -- (-2.25,1.5) node[above] {$\parcen{X_1}$};
    \draw[black] (-1.75, -2.25+3.75/2) node {$\cdots$};
    \draw[black] (-1.25,-2.25) node[below] {$\parcen{X_{i-1}}$} -- (-1.25,1.5) node[above] {$\parcen{X_{i-1}}$};
    \draw (-.125,-2.25) node[below] {$\parcen{\scriptstyle (a_1\cdots a_{i-2},a_{i-1})}$} -- (-.125,1.5) node[above] {$\scriptstyle (a_1\cdots a_{i-2},a_{i-1})$};
    \draw[black] (1,-2.25) node[below] {$\parcen{X_{i+1}}$};
    \draw (2,-2.25) node[below] {$\parcen{\scriptstyle (a_1\cdots a_{i-1},a_{i+1})}$};
    \draw[black] (3,-2.25) node[below] {$\parcen{X_i}$};
    \draw (4.25,-2.25) node[below] {$\parcen{\scriptstyle(a_1\cdots a_{i-1}a_{i+1},a_i)}$};
    \draw[black] (1,1.5) node[above] {$\parcen{X_i}$};
    \draw (2,1.5) node[above] { $ \scriptstyle (a_1\cdots a_{i-1},a_i)$};
    \draw[black] (3.125,1.5) node[above] {$\parcen{X_{i+1}}$};
    \draw (4.25,1.5) node[above] {$\scriptstyle (a_1\cdots a_i,a_{i+1})$};
      \draw[black] (5.5,-2.25) node[below] {$\parcen{X_{i+2}}$} -- (5.5,1.5) node[above] {$\parcen{X_{i+2}}$};
    \draw[black] (6, -2.25+3.75/2) node {$\cdots$};
    \end{tikzpicture}
  \end{align*}

  By pulling the \textcolor{slblue}{blue} strands under the black strands to the right of the diagram we obtain the required factorization.
  A similar computation works for the cups and caps.

  If our tangles include upward-oriented strands we need to be slightly more careful.
  For example,
  \begin{equation*}
    \linkinv{
      \begin{tikzpicture}[line width=1, baseline=.75cm, scale=1]
        \draw[->] (0,0) -- (0,1.5) node[above] {$X_{1}$} ;
        \draw (1,0) -- (1,1.5) node[above] {$X_{2}$};
      \end{tikzpicture}
    }{\mathcal{C}^{\zeta}}
    \quad =
    \begin{tikzpicture}[line width=1, baseline=.75cm, scale=1, slblue]
      \draw[->, black] (0,0) -- (0,1.5) node[above] {$X_{1}$} ;
      \draw[->] (1,0) -- (1,1.5) node[above] {$(a_{1}, a_{1}^{-1})$} ;
      \draw[black] (2,0) -- (2,1.5) node[above] {$X_{2}$};
      \draw (3,0) -- (3,1.5) node[above] {$(a_{1}^{-1}, a_{2})$} ;
    \end{tikzpicture}
  \end{equation*}
  as the zested dual \eqref{eq:zested dual} includes an extra object.
  Similarly, when braiding with an upward oriented strand, we braid \(\lambda(a_{1}, a_{1}^{-1})^{*}\) over a strand labeled by \(X_{2}\).
  This is not a problem, because we can use \cref{thm:chi lemma} and \cref{thm:chi facts}(d) to find that
  \begin{equation*}
    \mathcal{F}_{\mathcal{C}^\zeta} \left ( 
    \begin{tikzpicture}[line width=1, baseline=12.5, scale=.75] 
    \draw[<-] (0,0) node[below] {$X_2$} \br (1,1.5) node[above] {$X_2$};
    \draw[white, line width=10] (1,0) \br (0,1.5);
    \draw[->] (1,0) node[below] {$X_1$} \br (0,1.5) node[above] {$X_1$};
  \end{tikzpicture} \,\,
  \right ) = 
  \begin{tikzpicture}[line width=1,baseline=12.5, scale=.75]
  \draw (0,0) node[below] {$\phantom{a_1^{-1}}X_2\phantom{a_1^{-1}}$} to [out=90, in=-90] (2,1.5) node[above] {$X_2$};
  \draw[white, line width=10] (1,0) to [out=90, in=-90] (0,1.5);
  \draw[->] (1,0) node[below] {$\phantom{a_1^{-1}}X_1\phantom{a_1^{-1}}$} to [out=90, in=-90] (0,1.5) node[above] {$X_1$};
  \draw[white, line width=10] (2,0) to [out=90, in=-90] (1,1.5);
  \draw[slblue,->] (2,0) node[below] {$\phantom{a_1^{-1}}(a_1)\phantom{a_1^{-1}}$} to [out=90, in=-90] (1,1.5) node[above] {$(a_1)$};
  \draw[slblue] (3.5,0) node[below] {$(a_2, a_1^{-1})$} --(3.5,1.5) node[above] {$(a_1^{-1},a_2)$};
  \draw[slblue, fill=white] (3.5,.75) circle (.33) node {$\scriptstyle t$};
  \end{tikzpicture}
  = \omega(a_1,a_1^{-1},a_2) 
    \begin{tikzpicture}[line width=1,baseline=12.5, scale=.75]
     \draw[slblue,->] (2,0) node[below] {$\phantom{a_1^{-1}}(a_1)\phantom{a_1^{-1}}$} to [out=90, in=-90] (1,1.5) node[above] {$(a_1)$};
  \draw[white, line width=10] (0,0) to [out=90, in=-90] (2,1.5);
  \draw (0,0) node[below] {$\phantom{a_1^{-1}}X_2\phantom{a_1^{-1}}$} to [out=90, in=-90] (2,1.5) node[above] {$X_2$};
  \draw[white, line width=10] (1,0) to [out=90, in=-90] (0,1.5);
  \draw[->] (1,0) node[below] {$\phantom{a_1^{-1}}X_1\phantom{a_1^{-1}}$} to [out=90, in=-90] (0,1.5) node[above] {$X_1$};
  \draw[slblue] (3.5,0) node[below] {$(a_2, a_1^{-1})$} --(3.5,1.5) node[above] {$(a_1^{-1},a_2)$};
  \draw[slblue, fill=white] (3.5,.75) circle (.33) node {$\scriptstyle t$};
  \end{tikzpicture}
  \end{equation*}
\end{proof}

Our proof shows that the behavior of \(\warp\) under tensor products is more complicated than for composition.
The issue is that \(\mathcal{C}^{\zeta}\) is no longer strict, so we have to keep track of parenthesizations.
In \cref{sec:zest invariant} we show how to do this by using tangle diagrams with regions (not just tangles) colored by \(A\).%
\note{%
  Algebraically, the Reshetikhin-Turaev-inspired monoidal functor \(\zest{}\) underlying \(\zestmini{}\) is really a 2-functor in the sense that the tensor product has domains.
  To keep track of them we introduce region colors.
}
We conclude this section with a motivating example.

\begin{example}
  \label{ex:not monoidal}
  Consider objects \(X,Y\) of \(\mathcal{C}\) with \(|X| = a, |Y| = b\) and \(B\) the tangle braiding \(X\) over \(Y\).
  Then
  \[
    \linkinv{B}{\mathcal{C}} = \beta_{X,Y}
    =
    \begin{tikzpicture}[line width=1, baseline=12.5, scale=.75] 
    \draw (0,0) node[below] {$Y$} \br (1,1.5) node[above] {$Y$};
    \draw[white, line width=10] (1,0) \br (0,1.5);
    \draw (1,0) node[below] {$X$} \br (0,1.5) node[above] {$X$};
    \end{tikzpicture}
  \]
  and
  \[
    \linkinv{B}{\mathcal{C}^{\zeta}} = \beta_{X,Y} \otimes t(a,b)
    =
    \begin{tikzpicture}[line width=1, baseline=12.5, scale=.75] 
    \draw (0,0) node[below] {$\parcen{Y}$} \br (1,1.5) node[above] {$\parcen{Y}$};
    \draw[white, line width=10] (1,0) \br (0,1.5);
    \draw (1,0) node[below] {$\parcen{X}$} \br (0,1.5) node[above] {$\parcen{X}$};
    \begin{scope}[xshift=2.5cm,slblue]
    \onecirc{\lambda(b,a)}{\lambda(a,b)}{\scriptstyle t(a,b)};
    \end{scope}
    \end{tikzpicture}
  \]
  so \(\zestmini{B} = t(a,b)\).

  Now let \(W\) be another object with \(|W| = i\) and let 
  \[
    S =
    \begin{tikzpicture}[line width=1, baseline = 10]
      \draw (0,0) -- (0,1);
      \draw (0,1) node[above] {$W$};
    \end{tikzpicture}
  \]
  be the tangle with a single downward-oriented strand labeled by \(W\).
  Clearly
  \[
    \linkinv{S}{\mathcal{C}}
    =
    \linkinv{S}{\mathcal{C}^{\zeta}}
    =
    \id_W
  \]
  and
  \[
    \linkinv{S \du B}{\mathcal{C}}
    =
    \linkinv{S}{\mathcal{C}}
    \otimes
    \linkinv{B}{\mathcal{C}}
    =
    \id_{W} \otimes \beta_{X,Y}.
  \]
  On the other hand,
  \begin{equation}
    \label{eq:shifted braiding value}
    \linkinv{S \zestts B}{\mathcal{C}^{\zeta}}
    =
    \begin{tikzpicture}[line width=1, baseline = -7.5*3/2, scale=.75, slblue]
    \draw[black] (1,0)--(1,0 + 1.5);
    \draw (4,0)--(4,0 + 1.5);
    \draw[black] (2,-.75)--(2,-1.5);
    \draw[black] (1,-2.25)--(1,-.75);
    \draw (4,-2.25)-- (4,-.75);
    \draw[black] (2,0)--(2,0 + .75);
    \draw (3,-1.5)--(3,0 + .75);
    \draw (4,-.75)--(4,0 + .75);
    \draw[fill=white] (2.6,0 + -1.45) rectangle node {$\scriptstyle \nu(i,b,a)^{-1}$} (4.4,-.85);
    \draw[fill=white] (2.6,0 + .1) rectangle node {$\scriptstyle \nu(i,a,b)$} (4.4,0 + .7); 
    \draw (3,0+.75)  to [out=90, in=-90] (2,0 + 1.5);
    \draw[white, line width=10] (2,0+.75)  to [out=90, in=-90] (3,0 + 1.5);
    \draw[black] (2,0+.75) to [out=90, in=-90] (3,0 + 1.5);
    \draw[black] (1,-.75) \br (2,0);
    \draw[white, line width=10] (2,-.75) \br (1,0);
    \draw[black] (2,-.75) \br (1,0);
    \draw [fill=white] (3,-.75/2)  circle (.25) node {$t$};
    \draw (2,-2.25) \br (3,-1.5);
    \draw[white, line width= 10] (3,-2.25) \br (2,-1.5);
    \draw[black] (3,-2.25) \br (2,-1.5);
    \draw[black] (0,-2.25) node[below] {$\parcen{W}$} -- (0,1.5) node[above] {$\parcen{W}$};
    \draw[black] (1,-2.25) node[below] {$\parcen{Y}$};
    \draw (2,-2.25) node[below] {$\lambda(i,b)$};
    \draw[black] (3,-2.25) node[below] {$\parcen{X}$};
    \draw (4.25,-2.25) node[below] {$\lambda(ib,a)$};
    \draw[black] (1,1.5) node[above] {$\parcen{X}$};
    \draw (2,1.5) node[above] {$\lambda(i,a)$};
    \draw[black] (3,1.5) node[above] {$\parcen{Y}$};
    \draw (4.25,1.5) node[above] {$\lambda(ia,b)$};
    \end{tikzpicture}
  \end{equation}
  because the tensor product is no longer strict, so
  \[
    \zestmini{S \du B} =
    \nu(i,b,a)^{-1}
    \circ
    ( t(a,b) \otimes \id_{\lambda(i,a,b)} )
    \circ
    \nu(i,a,b),
  \]
  which is in general \emph{not} equal to \(\zestmini{S} \otimes \zestmini{B}\).
  We see the issue is the presence of the \(i\)-colored strand in the background, which introduces some associators that are in general non-trivial.
  In the next section we explain how to handle these by using diagrams with region labels.
\end{example}

\section{The zest invariant of a tangle}
\label{sec:zest invariant}
In this section we define a colored tangle invariant \(\zest{}\) and show it agrees with the morphism \(\zestmini{}\) appearing in \cref{thm:factorization existence}.
We also give an explicit, diagrammatic procedure for computing \(\zest{}\) that closely resembles the construction of quandle cocycle invariants \cite{Carter2003}.

Throughout this section we fix a category \(\mathcal{C}\) with universal grading group \(A\) and ribbon zesting data \(\zeta\).
As before we write \(\mathcal{C}^{\zeta}\) for its zest.

\subsection{\texorpdfstring{\(A\)}{A}-colored tangles}

\begin{definition}
  Let \(A\) be an abelian group.
  An \defemph{\(A\)-colored tangle} is an oriented, framed tangle \(T\) with each of its components labeled by an element of \(A\).
\end{definition}

\(A\)-colorings have a topological interpretation as cohomology classes on the tangle complement in $\mathbb{R}^2 \times I$ (\cref{thm:A coloring cohomology}).
Just as for \(\mathcal{C}\)-colored tangles, \(A\)-colored tangles form a category: we can compose two tangles if the \(A\)-colorings of the components agree.
The same is true for diagrams of \(A\)-colored tangles.
It turns out to be useful to consider an additional coloring on the regions.

\begin{marginfigure}
  \centering
  \begin{tikzpicture}[line width=1, baseline = 10]
    \draw[<-] (0,0) -- (0,1);
    \draw (0,1) node[above] {$a$};
    \draw[accent] (-0.25, 0.5) node[left] {$i$};
    \draw[accent] (0.25, 0.5) node[right] {$ia$};
  \end{tikzpicture}
  \quad \text{ and } \quad
  \begin{tikzpicture}[line width=1, baseline = 10]
    \draw[->] (0,0) -- (0,1);
    \draw (0,1) node[above] {$a$};
    \draw[accent] (-0.25, 0.5) node[left] {$i$};
    \draw[accent] (0.25, 0.5) node[right] {$ia^{-1}$};
  \end{tikzpicture}
  \caption{
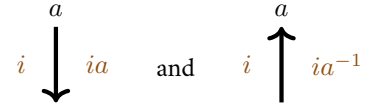
    A shadow coloring satisfies these relations at every edge, depending on its orientation.
    Here the black label corresponds to the edge and the \textcolor{accent}{gold} labels to the regions.
  }%
  \label{fig:coloring-rule-region}
\end{marginfigure}

\begin{definition}
  \label{def: shadow coloring}
  The \defemph{regions} of a tangle diagram \(D\) are the connected components of its complement.
  If \(D\) is an \(A\)-colored tangle diagram, a \defemph{shadow coloring}\note{%
    This is the usual name in the quandle literature \cite{Carter2001}.
    Shadow colorings of tangle diagrams are unrelated to Turaev's shadows \cite[Chapter IX]{Turaev1994}.
  }
  of \(D\) is a labeling of the regions of \(D\) with elements of \(A\) subject to the rules in \cref{fig:coloring-rule-region}.
  For a given \(A\)-coloring of a tangle diagram shadow colorings are naturally in bijection with \(A\): once we fix the color of a single region the rules in \cref{fig:coloring-rule-region} determine the remaining ones.
  Our convention is to view a shadow colored tangle as an \(A\)-colored tangle (diagram) along with a choice of element \(i \in A\) assigned to the leftmost region.
  For a link this is the unbounded region.
  We denote an arbitrary shadow-colored tangle as \((T,i)\) and call \(i\) the \defemph{base color} of (the region coloring of) \((T,i)\).

  When \(T\) has left-hand color \(i\) and right-hand color \(j\) we say it has \defemph{total color} \(|T| = i^{-1}j\).
  With this notation, the disjoint union of region-colored tangles can be written
  \[
    (i, T_1 \du T_2)
    =
    (i, T_1) \du (i|T_1|, T_2)
  \]
  as seen from the picture
  \begin{align*}
  (i, T_1 \du T_2) =    \begin{tikzpicture}[line width=1, baseline =3/4*.75cm, scale=.75]
  \draw (0,0) --(0,1.5);
  \draw (1,0)--(1,1.5);
  \draw[fill=white] (-.25,.45) rectangle node {$T_1$} (1.25,1.05);
  \draw[accent] (-.25,.75) node[left] {$\scriptstyle i$};
  \draw[accent] (1.25,.75) node[right] {$\scriptstyle i|T_1|$};
  \draw[accent] (4.00,.75) node[right] {$\scriptstyle i|T_1||T_2|$};
  \draw (0.5,.225) node {$\cdots$};
  \draw (0.5,1.275) node {$\cdots$};
  \begin{scope}[xshift=2.75cm]
  \draw (0,0) --(0,1.5);
  \draw (1,0)--(1,1.5);
  \draw[fill=white] (-.25,.45) rectangle node {$T_2$} (1.25,1.05);
  \draw (0.5,.225) node {$\cdots$};
  \draw (0.5,1.275) node {$\cdots$};
  \end{scope}
  \end{tikzpicture}
  \end{align*}
  with the region colors written in {\color{accent} gold}. Note that the total color \(|T_1 \du T_2| = |T_1||T_2|\)  is multiplicative under disjoint union of tangles. 
\end{definition}

\begin{remark}
  Algebraically we can think of an \(A\)-colored link as a coloring by the conjugation quandle of \(A\), which is trivial because \(A\) is abelian.
  Our definition of shadow coloring is then the usual one for quandles, with \(A\) acting on itself by multiplication.
\end{remark}

\begin{definition}
\label{def: tangle 2-category}
  There is a \(2\)-category \(\tangcat\) of oriented, framed \(A\)-shadow colored tangles with:
  \begin{description}
    \item[objects]
      elements of \(A\)
    \item[\(1\)-morphisms]
      \(i \to ia\) are lists \((i;(a_1, \epsilon_1), \dots, (a_n, \epsilon_n))\) of oriented points labeled with elements of \(A\) with total color \(a_1^{\epsilon_1} \cdots a_n^{\epsilon_n} = a\).
      For clarity we include the domain \(i\) of this morphism (i.e.\@ the left-hand region color).
      When it does not cause confusion we will abbreviate \(a = (a,1)\) and \(a^{*} = (a, -1)\) for downward and upward oriented points, respectively.%
      \note{
        Note that \(a^* = (a, -1)\) and \((a^{-1},1)\) are \emph{not} the same \(1\)-morphism, although they correspond to the same region coloring.
      }
    A general 1-morphism in \(\tangcat\) looks like the following picture, with a positive \(\epsilon\) drawn oriented downwards and a negative \(\epsilon\) upwards.
    \begin{center}
    \begin{tikzpicture}[line width=1]
    \draw (0,0) node[left=-.5cm] {${\color{accent} i}$};
    \draw (5,0) node[right=.5cm] {${\color{accent} ia}$};
    \foreach \x in {1,2,3}{
    \draw[fill=black] (\x,0) node[above=.5cm] {$a_{\x}$} circle (.0625);
    }
    \draw (4,0) node[] {$\cdots$};
    \draw[fill=black] (5,0) node[above=.5cm] {$a_n$} circle (.0625);
    \foreach \x in {1,3,5}{
    \draw[->] (\x,.25) -- (\x,-.25);
    }
      \draw[->] (2,-.25) -- (2,.25);
    \end{tikzpicture}
    \end{center}
    
    \item[\(2\)-morphisms]
      \((i; (a_1, \epsilon_1), \dots, (a_n, \epsilon_n)) \to (i;(b_1, \epsilon_1'), \dots, (b_m, \epsilon_m'))\) are \(A\)-colored, oriented, framed tangles \((i,T)\) with base color \(i\).

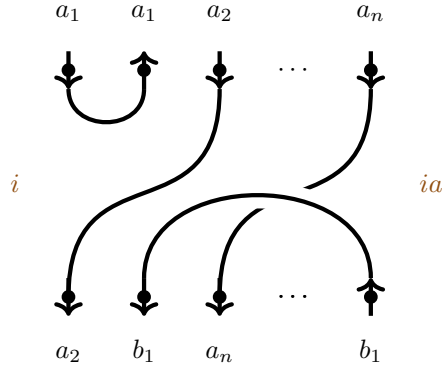
    \begin{center}
    \begin{tikzpicture}[line width=1]
    \draw (0,-1.5) node[left=-.5cm] {${\color{accent} i}$};
    \draw (5,-1.5) node[right=.5cm] {${\color{accent} ia}$};
    \draw[fill=black] (1,0) node[above=.5cm] {$a_{1}$} circle (.0625);
    \draw[fill=black] (2,0) node[above=.5cm] {$a_{1}$} circle (.0625);
    \draw[fill=black] (3,0) node[above=.5cm] {$a_{2}$} circle (.0625);
    \draw (4,0) node[] {$\cdots$};
    \draw[fill=black] (5,0) node[above=.5cm] {$a_n$} circle (.0625);
    \foreach \x in {1,3,5}{
    \draw[->] (\x,.25) -- (\x,-.25);
    }
      \draw[->] (2,-.25) -- (2,.25);
      
    \begin{scope}[yshift=-3cm]
  
    \draw[fill=black] (1,0) node[above=-1cm] {$a_2$} circle (.0625);
    \draw[fill=black] (3,0) node[above=-1cm] {$a_n$} circle (.0625);
    \draw (4,0) node[] {$\cdots$};
     \draw[fill=black] (2,0) node[above=-1cm] {$b_1$} circle (.0625);
    \draw (4,0) node[] {$\cdots$};
    \draw[fill=black] (5,0) node[above=-1cm] {$b_1$} circle (.0625);
    \foreach \x in {1,2,3}{
    \draw[->] (\x,.25) -- (\x,-.25);
    }
    \draw[->] (5,-.25) -- (5,.25);
    \end{scope}
    \draw[looseness=1.5] (1,-.25) to [out=-90,in=-90] (2,-.25);
    \draw[looseness=1.5] (3,-.25) to [out=-90, in=90] (1,-3);
      \draw[looseness=1.5] (5,-.25) to [out=-90, in=90] (3,-3);
      \draw[white, line width=10, looseness=1.25] (2,-2.75) to [out=90,in=90] (5,-2.75);
       \draw[looseness=1.25] (2,-2.75) to [out=90,in=90] (5,-2.75);
    \end{tikzpicture}
    \end{center}
      
  \end{description}
  The composition \(\du\) of \(1\)-morphisms is concatenation, with appropriate restrictions on the base colors: for example
  \[
    (i;a) \du (ia;b) = (i; a, b)
  \]
  but \((i; a) \du (i; b)\) is not defined.
  This corresponds to the disjoint union in the monoidal \(1\)-category of tangles, while the composition of \(2\)-morphisms is the usual composition of tangles.
  For this reason we write \(\du\) for the composition of \(1\)-morphisms.
\end{definition}

\begin{remark}
Since the region labels of any given 2-morphism in \(\tangcat\) are determined by the base color and the underlying colored tangle, in diagrams we will often suppress all region labels except the base color and perhaps the rightmost region color, as in the pictures above illustrating \cref{def: tangle 2-category}.
\end{remark}

\subsection{Definition of the zest invariant}
\label{sec:zest invariant def}
In this section we construct a \(2\)-functor \(\tangcat \to \delooping \Inv(\mathcal{C}_{e})\) which produces the invariant of (oriented, framed) shadow colored tangles that arises in the zesting construction. The target for this functor is the delooping \(\delooping \Inv(\mathcal{C}_e)\) of the pointed subcategory \(\Inv(\mathcal{C}_e)\), i.e. the \(2\)-category that arises from viewing a monoidal \(1\)-category as a \(2\)-category with a single object. 

Given ribbon zesting data $\zeta=(\lambda,\nu,t,\epsilon)$ with $\lambda$ valued in \(\Inv(\mathcal{C}_e)\) we want to define a \(2\)-functor \(\zestinvname : \tangcat \to \delooping\Inv(\mathcal{C}_e)\) that satisfies the factorization property of \cref{thm:factorization existence}; this will also establish that \(\zestinvname\) is a tangle invariant and allow us to compute it locally in a manner analogous to the usual Reshetikhin-Turaev tangle invariants.
To do this we need to choose the values of \(\zestinvname\) on the generators of  \(\tangcat\).

This is trivial to describe on generating objects: since \(\delooping{\Inv}(\mathcal{C}_e)\) has a single object \(\bullet\), we have \(\zest{i} = \bullet\) for all objects \(i \in A\). Next we consider \(\zestinvname\) on the \(1\)-morphisms of \(\tangcat\); recall that these are lists of colored, oriented points, which are usually viewed as the objects of a tangle category.

\subsubsection{The \texorpdfstring{\(2\)}{2}-functor \texorpdfstring{\(\mathcal{J}_\zeta\)}{𝒥\_ζ} on \texorpdfstring{\(1\)}{1}-morphisms}
\label{sec:zest invariant def on 1 morphisms}
For a generating \(1\)-morphism we define%
\note{
  Here \(\lambda(a,a^{-1})^*\) is the dual of \(\lambda(a,a^{-1})\) in \(\mathcal{C}\).
}
\begin{equation}
  \label{eq:zest object rules}
  \begin{aligned}
    \zest{i;a}
    &\defeq
    \lambda(i,a)
    \\
    \zest{i;a^*}
    &\defeq
    \lambda(a, a^{-1})^* \otimes \lambda(i,a^{-1})
  \end{aligned}
\end{equation}
and extend to arbitrary \(1\)-morphisms via horizontal composition.
For example,
\begin{equation*}
  \zest{i;a_1, a_2}
  =
  \zest{(i;a_1) \du (ia_1; a_2)}
  =
  \zest{i; a_1}
  \otimes
  \zest{i a_1; a_2}
  =
  \lambda(i,a_1) \otimes \lambda(ia_1, a_2)
\end{equation*}
and
\begin{equation*}
  \zest{i; a_1^*, a_2}
  =
  \lambda(a_1, a_1^{-1})^*
  \otimes
  \lambda(i,a_1)
  \otimes
  \lambda(ia_1, a_2).
\end{equation*}
These values are essentially determined by \cref{thm:factorization existence}: for example, if \(W\) and \(X\) are homogeneous objects of \(\mathcal{C}\) with degrees \(|W| = i\), \(|X| = a\), then
\[
  \linkinv{W \du X}{\mathcal{C}^{\zeta}}
  =
  W \otimes X \otimes \lambda(i,a).
\]
We think of a strand with degree \(a\) next to a region labeled by \(i\) as coming from \(W \zestts X\), so, for our factorization to hold, we want to choose \(\zest{i; a}\) so that
\[
  \linkinv{W \du X}{\mathcal{C}^{\zeta}}
  =
  \linkinv{W \otimes X}{\mathcal{C}}
  \otimes
  \zest{i;a}
\]
hence \(\zest{i;a} \defeq \lambda(i,a)\), which gives the first rule in \cref{eq:zest object rules}.
The second comes from a similar analysis of duality of objects in \(\mathcal{C}^{\zeta}\).

\subsubsection{The \texorpdfstring{\(2\)}{2}-functor \texorpdfstring{\(\mathcal{J}_\zeta\)}{𝒥\_ζ} on \texorpdfstring{\(2\)}{2}-morphisms}
\label{sec:zest invariant def on 2 morphisms}
Something similar works for tangles.
Given an \(A\)-colored tangle \((i, T)\) we can choose a homogeneous \(\mathcal{C}\)-coloring \(\widetilde T\) of \(T\) inducing the \(A\)-coloring and an object \(W\) of \(\mathcal{C}\) with degree \(i\).
(Here we use the faithfulness of the \(A\)-grading of \(\mathcal{C}\).)
As in \cref{ex:not monoidal} let \(S\) be a downward-oriented strand colored by \(W\).
Then the requirement that
\(
  \linkinv{S \du \widetilde T}{\mathcal{C}^\zeta}
  =
  \linkinv{S\du \widetilde T}{\mathcal{C}}
  \warp
  \zest{i; T}
\)
or graphically
\[
  \linkinv{S \du \widetilde T}{\mathcal{C}^\zeta}
    =
     \begin{tikzpicture}[line width=1, baseline =.85*2cm, scale=.85]
     \draw[slblue, looseness=.5] (1.5,0) \br (4.75,1.5);
    \draw[slblue, looseness=.5] (2.5,0) \br (5.75,1.5);
    \draw[slblue, looseness=.5] (4.5,0) \br (7.75,1.5);
    \draw[slblue, looseness=.5] (4.75,2.5) \br (1.5,4);
    \draw[slblue, looseness=.5] (5.75,2.5) \br (2.5,4);
    \draw[slblue, looseness=.5] (7.75,2.5) \br (4.5,4);
    \draw[slblue,fill=white] (4.5, 1.5) rectangle node {$\zest{i; T}$} (8,2.5);
    \foreach \x in {0,1,2,4}{
   	\draw[white, line width=10] (\x,0)--(\x,4);
    	\draw(\x,0)--(\x,4);
    }
    \draw[fill=white] (.75,1.5) rectangle node {$\mathcal{F}_{\mathcal{C}}(\widetilde{T})$} (4.25,2.5);
    \draw (0,4) node[above] {$\scriptstyle W$};
    \draw (1,4) node[above] {$\scriptstyle X_1$};
    \draw (2,4) node[above] {$\scriptstyle X_2$};
    \draw (4,4) node[above] {$\scriptstyle X_m$};
    \draw (3,1.25) node {$\cdots$};
    \draw (3,2.75) node {$\cdots$};
    \draw (3.25, 4) node[above] {$\scriptstyle \cdots$};
    \draw[slblue] (6.5,1.25) node {$\cdots$};
    \draw[slblue] (6.5,2.75) node {$\cdots$};
    \end{tikzpicture}
\]
completely determines \(\zest{}\).

We already did this computation for downward-oriented braid generators in \cref{ex:not monoidal} and concluded that the value of \(\zest{}\) on a crossing is 
\[
\label{eq: def phi}
 \phi_i(a_1,a_2) \defeq \nu(i,a_2,a_1)^{-1} \circ \left ( t(a,b) \otimes \id_{\lambda(i,a_1a_2)} \right ) \circ \nu(i, a_1, a_2)
\]
or graphically
\begin{equation}
  \label{eq:zested braiding}
  \zest{
    \begin{tikzpicture}[line width=1, baseline=12.5, scale=.75] 
    \draw[<-] (0,0) node[below] {$a_2$} \br (1,1.5) node[above] {$a_2$};
    \draw[white, line width=10] (1,0) \br (0,1.5);
    \draw[<-] (1,0) node[below] {$a_1$} \br (0,1.5) node[above] {$a_1$};
    \draw[accent] (-.5,.75) node {$i$};
  \end{tikzpicture} \,\,
  }
  \quad = \quad  \begin{tikzpicture}[line width=1,scale=1, baseline=.75cm, slblue]
  \draw (1,0) node[below] {$\scriptstyle \lambda(i,a_2)$}--(1,1.5) node[above] { $\scriptstyle \lambda(i,a_1)$};
  \draw (2,0) node[below] {$\scriptstyle \phantom{a} \lambda(ia_2,a_1)$}--(2,1.5) node[above] {$\scriptstyle \phantom{a}\lambda(ia_1,a_2)$};
  \draw[fill=white] (.75,.5) rectangle node {$\scriptstyle \phi_i(a_1,a_2)$} (2.25,1);
  \end{tikzpicture} \quad := \quad
   \begin{tikzpicture}[line width=1,scale=1, baseline=2cm, slblue]
   \draw (1,0) node[below] {$\scriptstyle \lambda(i,a_2)$}--(1,4) node[above] { $\scriptstyle \lambda(i,a_1)$};
  \draw (2,0) node[below] {$\scriptstyle \phantom{a} \lambda(ia_2,a_1)$}--(2,4) node[above] {$\scriptstyle \phantom{a}\lambda(ia_1,a_2)$};
    \draw[fill=white] (.675,.75) rectangle node {$\scriptstyle \nu(i,a_2,a_1)^{-1}$} (2.375,1.25);
    \draw[fill=white] (.675,2.75) rectangle node {$\scriptstyle \nu(i,a_1,a_2)$} (2.375,3.25);
    \draw[fill=white] (1,2) circle (.25) node {$t$};
  \end{tikzpicture} \quad.
\end{equation}
Then for the other orientations of positive crossing we have:
\[
  \zest{
    \begin{tikzpicture}[line width=1, baseline=12.5, scale=.75] 
    \draw[->] (0,0) node[below] {$a_2$} \br (1,1.5) node[above] {$a_2$};
    \draw[white, line width=10] (1,0) \br (0,1.5);
    \draw[<-] (1,0) node[below] {$a_1$} \br (0,1.5) node[above] {$a_1$};
    \draw[accent] (-.5,.75) node {$i$};
  \end{tikzpicture} \,\,
 }
  \quad =
 \begin{tikzpicture}[line width=1,scale=1, baseline=1cm, slblue]
  \draw (1,0) node[below] {$\scriptstyle \lambda(i,a_2^{-1})$}--(1,1.5) to [out=90, in=-90] (0,3) node[above] { $\scriptstyle \lambda(i,a_1)$};
  \draw (2,0) node[below] {$\scriptstyle \phantom{a_2^{-1}} \lambda(ia_2^{-1},a_1)$}--(2,3) node[above] {$\scriptstyle \phantom{a}\lambda(ia_1,a_2^{-1})$};
   \draw[fill=white] (.75,.5) rectangle node {$\scriptstyle \phi_i(a_1,a_2^{-1})$} (2.25,1);
  \draw[white, line width=10] (0,0) node[below] {$\scriptstyle \phantom{a_2^{-1}} \lambda(a_2)$} to [out = 90, in=-90] (1,3) node[above] {$\scriptstyle \lambda(a_2)$};
  \draw[->] (0,0) node[below] {$\scriptstyle \phantom{a_2^{-1}}\lambda(a_2)\phantom{a_2^{-1}}$} to [out = 90, in=-90] (1,3) node[above] {$\scriptstyle \lambda(a_2)$};
  \end{tikzpicture}
\]
\[
  \zest{
    \begin{tikzpicture}[line width=1, baseline=12.5, scale=.75] 
    \draw[<-] (0,0) node[below] {$a_2$} \br (1,1.5) node[above] {$a_2$};
    \draw[white, line width=10] (1,0) \br (0,1.5);
    \draw[->] (1,0) node[below] {$a_1$} \br (0,1.5) node[above] {$a_1$};
    \draw[accent] (-.5,.75) node {$i$};
  \end{tikzpicture} \,\,
 }
  \quad = \omega(a_1,a_1^{-1};a_2)^{-1}
  \begin{tikzpicture}[line width=1,scale=1, baseline=1cm, slblue]
  \draw (0,0) node[below] {$\scriptstyle \phantom{a_1^{-1}}\lambda(i,a_2)\phantom{a_1^{-1}}$} to [out=90, in=-90] (1,1.5)  -- (1,3) node[above] { $\scriptstyle \lambda(i,a_1^{-1})$};
  \draw (2,0) node[below] {$\scriptstyle \phantom{a_1^{-1}} \lambda(ia_2,a_1^{-1})$}--(2,3) node[above] {$\scriptstyle \phantom{a_1^{-1}a_1^{-1}}\lambda(ia_1^{-1},a_2)$};
   \draw[fill=white] (2.25,2) rectangle node {$\scriptstyle \phi_i(a_1^{-1},a_2)$} (.75,2.5);
  \draw[white, line width=10] (1,0) to [out = 90, in=-90] (0,3);
  \draw[->] (1,0) node[below] {$\scriptstyle \phantom{a_1^{-1}}\lambda(a_1)\phantom{a_1^{-1}}$} to [out = 90, in=-90] (0,3) node[above] {$\scriptstyle \lambda(a_1)$};
  \end{tikzpicture}
\]

\[
  \zest{
    \begin{tikzpicture}[line width=1, baseline=12.5, scale=.75] 
    \draw[->] (0,0) node[below] {$a_2$} \br (1,1.5) node[above] {$a_2$};
    \draw[white, line width=10] (1,0) \br (0,1.5);
    \draw[->] (1,0) node[below] {$a_1$} \br (0,1.5) node[above] {$a_1$};
    \draw[accent] (-.5,.75) node {$i$};
  \end{tikzpicture} \,\,
 }
  \quad = \omega(a_1,a_1^{-1};a_2)
  \begin{tikzpicture}[line width=1,scale=1, baseline=2.25cm, slblue]
  \draw (0,0) node[below] {$\scriptstyle \lambda(i,a_2^{-1})$} to [out=90, in=-90] (1,1.5)  -- (1,3) to [out=90,in=-90] (0,4.5) node[above] { $\scriptstyle \lambda(i,a_1^{-1})$};
  \draw (2,0) node[below] {$\scriptstyle \phantom{a_1^{-1}} \lambda(ia_2^{-1},a_1^{-1})$}--(2,4.5) node[above] {$\scriptstyle \phantom{a_1^{-1}}\lambda(ia_1^{-1},a_2^{-1})$};
   \draw[fill=white] (2.5,2) rectangle node {$\scriptstyle \phi_i(a_1^{-1},a_2^{-1})$} (.5,2.5);
  \draw[white, line width=10] (1,0) to [out = 90, in=-90] (0,3);
  \draw[->] (1,0) node[below] {$\scriptstyle \phantom{a_1^{-1}}\lambda(a_1)\phantom{a_1^{-1}}$} to [out = 90, in=-90] (-1,4.5) node[above] {$\scriptstyle \lambda(a_1)$};
  \draw[white, line width=10] (-1,0) to [out=90,in=-90] (1,4.5);
  \draw[->] (-1,0) node[below] {$\scriptstyle \phantom{a_1^{-1}}\lambda(a_2)\phantom{a_1^{-1}}$} to [out=90,in=-90] (1,4.5) node[above] {$\scriptstyle \lambda(a_2)$};
  \end{tikzpicture}
\]
Here we have used the same trick as in the proof of \cref{thm:factorization existence}: we can swap the sign of a crossing involving an invertible object at the cost of a scalar $\omega$.

Something similar works for cups and caps.
Using the rigid structure described in \cref{sec: zested monoidal and duals} we assign the value of  \(\zest{}\) on a left-oriented cap 
\[
  (i;) \to (i; a,a^*) = \lambda(i,a) \otimes \lambda(a,a^{-1})^* \otimes \lambda(ia,a^{-1})
\]
as in \cref{eq:coev-left-zest},
\begin{align}
   \zest{
  \begin{tikzpicture}[line width=1, scale=1.25, baseline=-.75*1.25cm]
  \draw[<-, looseness=2] (1,-1) to [out=90, in=90] (2,-1);
  \draw (1,-1) node[below] {$a$};
  \draw[accent] (0.75,-0.75) node[left] {$i$};
  \end{tikzpicture}\quad
  }
  \quad = \quad \frac{1}{\dim(\lambda(a))}
  \begin{tikzpicture}[line width=1, slblue, scale = 1.25, baseline = -1.875cm]
  \draw[] (0,-3.1) node[below] {$\scriptstyle \phantom{a^{-1}}\lambda(i,a)\phantom{a^{-1}}$}\br (1, -1.6);
  \draw[white, line width=10] (1,-3.1) node[below] {$\scriptstyle \lambda(a,a^{-1})^*$} \br (0, -1.6);
  \draw[] (1,-3.1) \br (0, -1.6);
  \draw (0, -1.6)--(0,-1);
  \draw (.875,-3.1) node[below] {$\scriptstyle \lambda(a,a^{-1})^*$};
  \draw[fill=white] (.75,-1.6) rectangle node {$\scriptstyle \nu(i,a,a^{-1})^{-1}$} (2.25,-1);
  \draw[] (2,-3.1) node[below] {$\scriptstyle \lambda(ia,a^{-1})$}--(2,-1.6);
   \draw[looseness=2] (0,-1) to [out=90, in=90] (1,-1);
  \end{tikzpicture}
  \label{eq:coev-left-zest}
\end{align}
 while the value of the left-oriented cup
\[
  (i; a^*, a) = \lambda(a, a^{-1})^* \otimes \lambda(i, a^{-1}) \otimes \lambda(ia^{-1}, a) \to (i;)
\]
is given by \cref{eq:ev-left-zest}.
\begin{align}
     \zest{
    \begin{tikzpicture}[line width=1, scale=1.25, baseline=-.25*1.25cm]
    \draw[<-, looseness=2] (1,0) to [out=-90, in=-90] (2,0) node[above] {$a$};
    \draw[accent] (0.75,-0.25) node[left] {$i$};
    \end{tikzpicture}\quad
    }
    \quad = \quad 
    \begin{tikzpicture}[line width=1, slblue, scale = 1.25, baseline = -1.875cm]
    \draw[looseness=2] (1,0) node[above] {$\scriptstyle \lambda(i,a^{-1})$}--(1,-1.5) to [out=-90, in=-90] (0,-1.5)--(0,0) node[above] {$\scriptstyle \lambda(a,a^{-1})^*$};
    \draw (2,0) node[above] {$\scriptstyle \lambda(ia^{-1},a)$}--(2,-.75);
    \draw[fill=white] (.75,-1) rectangle node {$\scriptstyle \nu(i,a^{-1},a)$} (2.25,-.4);
     \draw[fill=white] (.625,-1.75) rectangle node {$\scriptstyle \nu(a)^{-1}$} (1.35,-1.25);
    \end{tikzpicture}
  \label{eq:ev-left-zest}
\end{align}

The right-oriented cap is the map
\[
  (i;) \to (i;a^*, a) = \lambda(a,a^{-1})^* \otimes \lambda(i,a^{-1}) \otimes \lambda(ia^{-1}, a)
\]
given by \cref{eq:coev-right-zest},
\begin{align}
\label{eq:coev-right-zest}
    \zest{
    \begin{tikzpicture}[line width=1, scale=1.25, baseline=-.75*1.25cm]
    \draw[->, looseness=2] (1,-1) to [out=90, in=90] (2,-1);
    \draw (2,-1) node[below] {$a$};
    \draw[accent] (0.75,-0.75) node[left] {$i$};
    \end{tikzpicture}\quad
    }
    \quad = \quad \frac{1}{\dim(\lambda(a)) \epsilon(a)}
    \begin{tikzpicture}[line width=1, slblue, scale = 1.25, baseline = -1.875cm]
    \draw[looseness=2] (0,-2.5) node[below] {$\scriptstyle \lambda(a,a^{-1})^*$}--(0,-1) to [out=90, in=90] (1,-1)--(1,-2.5)node[below] {$\scriptstyle \lambda(i,a^{-1})$};
    \draw[] (2,-1.75)--(2,-2.5) node[below] {$\scriptstyle \lambda(ia^{-1},a)$};
    \draw[fill=white] (.75,-2.1) rectangle node {$\scriptstyle \nu(i,a^{-1},a)^{-1}$} (2.25,-1.5);
     \draw[fill=white] (.75,-1.1) rectangle node {$\scriptstyle \nu(a)$} (1.25,-.6);
    \end{tikzpicture}
\end{align}
while the right-oriented cup is the map
\[
  (i; a, a^*) = \lambda(i, a) \otimes \lambda(a, a^{-1})^* \otimes \lambda(ia, a^{-1}) \to (i;)
\]
shown in \cref{eq:ev-right-zest}.
\begin{align}
    \zest{
    \begin{tikzpicture}[line width=1, scale=1.25, baseline=-.25*1.25cm]
    \draw[->, looseness=2] (1,0) node[above] {$a$} to [out=-90, in=-90] (2,0);
    \draw[accent] (0.75,-0.25) node[left] {$i$};
    \end{tikzpicture}\quad
    }
    \quad = \quad \epsilon(a)
    \begin{tikzpicture}[line width=1, slblue, scale = 1.25, baseline = -1.875cm]
    \draw (1,-1.5) \br (0, 0) node[above] {$\scriptstyle \lambda(i,a)$};
    \draw[white, line width=10] (0,-1.5) \br (1,0);
    \draw[] (0,-1.5 ) \br (1,0);
    \draw (.875,0) node[above] {$\scriptstyle \lambda(a,a^{-1})^*$};
    \draw[fill=white] (.75,-2.1) rectangle node {$\scriptstyle \nu(i,a,a^{-1})$} (2.25,-1.5);
    \draw (2,-1.5)--(2,0) node[above] {$\scriptstyle \lambda(ia,a^{-1})$};
    \draw[looseness=2] (0,-1.5)--(0,-2.1) to [out=-90,in=-90] (1,-2.1);
    \end{tikzpicture}
  \label{eq:ev-right-zest}
\end{align}

Finally, it is easy to see that
\begin{align}
\label{eq: zested twist}
  \zest{ 
	\begin{tikzpicture}[line width=1,baseline=1cm]
	\draw (0,0) node[below] {$a$}-- (0,1) to [out=90,in=180] (.25,1.25) to [out=0,in=90] (.5,1) to [out=-90,in=0] (.25,.75);
	\draw[white,line width=10] (.25,.75) to [out=180,in=-90] (0,1)--(0,2);
	\draw (.25,.75) to [out=180,in=-90] (0,1)--(0,2) node[above] {$a$};
	\draw[accent] (-.25,1) node[left] {$i$};
	\end{tikzpicture}
}
  =
   t(a) \epsilon(a) 
    \begin{tikzpicture}[line width=1,baseline=1cm,slblue]
	\draw (0,0) node[below] {$ \scriptstyle \lambda(i,a)$} -- (0,2) node[above] {$ \scriptstyle \lambda(i,a)$};
	\end{tikzpicture}
\end{align}
regardless of the value of the region color \(\textcolor{accent}{i}\).

\begin{definition}
  Let \(T\) be an \(A\)-colored tangle and \(i \in A\) a base color.
  We define \(\zest{i;T}\) to take the values in \cref{eq:zest object rules,eq:zested braiding,eq:coev-right-zest,eq:ev-right-zest,eq:coev-left-zest,eq:ev-left-zest} on tangle generators and then extend it to all \(A\)-colored tangles by
  \begin{align*}
    \zest{i, T' \circ T} &= \zest{i, T'} \circ \zest{i, T}
    \\
    \zest{i, T_1 \du T_2} &= \zest{i, T_1} \circ \zest{i|T_1|, T_2}
  \end{align*}
\end{definition}

We now have a procedure to compute \(\zest{i,T}\) for any \(A\)-colored tangle \(T\):
choose a diagram \(D\) of \(T\) in Morse position (i.e.\@ decompose \(T\) into a product of disjoint unions of tangle generators) and then compose.
We give an example in \cref{sec:zest invariant examples}.
In the next subsection we show that the resulting morphism does not depend on the choice of diagram and satisfies \cref{thm:factorization existence}.

\subsection{Properties of the zest invariant}

\begin{theorem}
  \label{thm:tangle relations}
  \(\zest{i,T}\) is a tangle invariant: it is independent of the diagram used to present \(T\).
\end{theorem}
 
It is clear that the usual generators and relations for the category of oriented ribbon tangles become generators for \(\tangcat\) by allowing arbitrary shadow colorings by \(A\).
Thus we can use the standard argument to prove independence of the choice of diagram:
we show that the value of \(\zest{}\) is invariant under colored, oriented, framed Reidemeister moves. 
To do so we need two lemmas.

\begin{lemma}
  \label{thm:factorization still holds}
  Let \(S\) be a tangle with a single downward-oriented strand colored by a homogeneous object \(W\) with \(|W| = i\), which exists because the grading of \(\mathcal{C}\) is faithful.
  Let \(T\) be any homogeneous \(\mathcal{C}\)-colored tangle \(T\).
  Then
  \[
    \zestmini{S \du T}
    =
    \zest{i,T}
  \]
  where \(\zest{i,T}\) is the value on the underlying \(A\)-colored tangle with leftmost shadow color \(i\).
  In particular \(\zest{}\) satisfies \cref{thm:factorization existence} in the sense that
  \[
    \linkinv{S \du T}{\mathcal{C}^{\zeta}}
    =
    \linkinv{S \du T}{\mathcal{C}}
    \warp
    \zest{i,T}.
    \qedhere
  \]
\end{lemma}

\begin{proof}
  Because the warp product respects vertical composition it suffices to show the factorization for the tangle generators with parallel vertical strands used in the proof of \cref{thm:factorization existence}.
  This immediately follows from the definition of \(\zest{}\).
  For example, recall the braid generator \(B_i\) from the proof of \cref{thm:factorization existence}.
  As an \(A\)-colored tangle we have
  \[
    (e, B_i) =
    (e, S_{a_1})
    \du
    (a_1, S_{a_2})
    \du
    \cdots
    \du
    (a_{1} \cdots a_{i-1}, \beta_{a_i, a_{i+1}})
    \du
    \cdots
    \du
    (a_1 \cdots a_{n-1}, S_{a_n})
  \]
  where \(S_a\) is a single downward-oriented strand colored by \(a\) and \(\beta_{a,b}\) is a braid generator with strands colored by \(a\) and \(b\). (Here we are recycling the notation $\beta$ for the braiding in the tangle category.)
  It is clear that
  \begin{align*}
    &\zest{e, B_i}
    \\
    &=
    \zest{e, S_{a_1}}
    \circ
    \zest{a_1, S_{a_2}}
    \circ
    \cdots
    \circ
    \zest{a_{1} \cdots a_{i-1}, \beta_{a_i, a_{i+1}}}
    \circ
    \cdots
    \circ
    \zest{a_1 \cdots a_{n-1}, S_{a_n}}
    \\
    &=
    \zestmini{B_i}
  \end{align*}
  and the same analysis works for the other generators.
\end{proof}

\begin{lemma}
  \label{thm:identities preserved}
  Let \(T\) be a homogeneous \(A\)-colored tangle that is framed isotopic to an identity tangle.
  Then \( \zest{i,T} \) is an identity map.
\end{lemma}

\begin{proof}
  As in \cref{thm:factorization still holds} let \(S\) be a tangle with a single downward-oriented strand colored by a homogeneous object \(W\) with \(|W| = i\).
  Similarly choose a homogeneous \(\mathcal{C}\)-coloring \(\tilde{T}\) inducing the \(A\)-coloring of \(T\).
  By hypothesis \(\mathcal{C}\) is a ribbon category, and one of the main results of \cite{Delaney2020} is that \(\mathcal{C}^{\zeta}\) is a ribbon category as well.
  The Reshetikhin-Turaev construction then says that, because \(S \du \tilde{T}\) is isotopic to an identity tangle,
  \[
    \linkinv{S \du \tilde{T}}{\mathcal{C}} = \id_{W \otimes X_1 \otimes \cdots \otimes X_n}
  \]
  and similarly \(\linkinv{S \du \tilde{T}}{\mathcal{C}^{\zeta}}\) is the identity map on the object in \cref{eq:left associated zest tensor}.
  (Here we again assume all strands are oriented downwards: if not there will be some extra factors of \(\lambda\) that do not change the argument.)
  By \cref{thm:factorization still holds} we have
  \[
    \linkinv{S \du \tilde{T}}{\mathcal{C}^{\zeta}}
    =
    \linkinv{S \du \tilde{T}}{\mathcal{C}}
    \warp
    \zest{i,T}
  \]
  which is only possible if 
  \[
    \zest{i,T}
    =
    \id_{ \lambda(i, a_{1}) \otimes \lambda(i a_1, a_2) \otimes \lambda(i a_1a_2, a_3) \otimes \cdots \otimes \lambda(i a_1 a_2 \cdots a_{n-1}, a_n) }.
    \qedhere
  \]
\end{proof}

\begin{figure}
  \begin{center}
  \begin{tikzpicture}[line width=1, scale=.5, baseline=1cm]
    \draw (0,4) node[above] {$a$};
    \draw (0,0) node[below] {$a$};
    \draw[accent] (-1,2) node[left] {$i$};
    \draw (0,0) -- (0,1) to [out=90,in=180] (.25,1.25) to [out=0,in=90] (.5,1) to [out=-90,in=0] (.25,.75);
    \draw[white,line width=7] (.25,.75) to [out=180,in=-90] (0,1)--(0,2);
    \draw (.25,.75) to [out=180,in=-90] (0,1)--(0,2);
    \begin{scope}[yscale=-1,yshift=-4cm]
    \draw (0,0) -- (0,1) to [out=90,in=180] (.25,1.25) to [out=0,in=90] (.5,1) to [out=-90,in=0] (.25,.75);
    \draw[white,line width=7] (.25,.75) to [out=180,in=-90] (0,1)--(0,2);
    \draw (.25,.75) to [out=180,in=-90] (0,1)--(0,2);
    \end{scope}
  \end{tikzpicture}
  \qquad
  \qquad
  \begin{tikzpicture}[line width=1, scale=.5, baseline=.75cm]
    \draw (0,3) node[above] {$a$};
    \draw (1,3) node[above] {$b$};
    \draw (0,0) node[below] {$\phantom{b}a\phantom{b}$};
    \draw (1,0) node[below] {$b$};
    \draw[accent] (-1,1.5) node[] {$i$};
    \begin{scope}[xscale=-1,xshift=-1cm]
    \draw (0,0) \br (1,1.5);
    \draw[white, line width=10] (1,0) \br (0,1.5);
    \draw (1,0) \br (0,1.5);
    \begin{scope}[yshift=1.5cm]
    \draw (1,0) \br (0,1.5);
    \draw[white, line width=10] (0,0) \br (1,1.5);
    \draw (0,0) \br (1,1.5);
    \end{scope}
    \end{scope}
  \end{tikzpicture}
  \qquad
  \qquad
  \begin{tikzpicture}[line width=1, scale=.5,baseline=2.25cm] 
    \draw (0,0) \br (2,3);
    \draw[white, line width=10] (1,0) \br (0,1.5);
    \draw (1,0) \br (0,1.5);
    \draw (2,0) --(2,1.5);
    \begin{scope}[yshift=3cm]
    \draw (0,0) \br (1,1.5);
    \draw (2,0) --(2,1.5);
    \end{scope}
    \begin{scope}[yshift=1.5cm]
    \draw (0,0)--(0,1.5);
    \draw[white, line width=10] (2,0) \br (0,3);
    \draw (2,0) \br (0,3);
    \end{scope}
    \begin{scope}[yscale=-1,yshift=-9cm]
    \draw (0,0)  \br (2,3);
    \draw[white, line width=10] (1,0) \br (0,1.5);
    \draw (1,0)  \br (0,1.5);
    \draw (2,0) --(2,1.5);
    \begin{scope}[yshift=3cm]
    \draw (0,0) \br (1,1.5);
    \draw (2,0) --(2,1.5);
    \end{scope}
    \begin{scope}[yshift=1.5cm]
    \draw (0,0)--(0,1.5);
    \draw[white, line width=10] (2,0) \br (0,3);
    \draw (2,0) \br (0,3);
    \end{scope}
    \end{scope}
    \draw (0,0) node[below] {$\phantom{b}a\phantom{b}$};
    \draw (1,0) node[below] {$b$};
    \draw (2,0) node[below] {$\phantom{b}c\phantom{b}$};
    \draw (0,9) node[above] {$a$};
    \draw (1,9) node[above] {$b$};
    \draw (2,9) node[above] {$c$};
    \draw[accent] (-1,4.5) node[] {$i$};
  \end{tikzpicture}
   \end{center}
  \caption{Colored, oriented, framed Reidemeister moves as relators}
  \label{fig:Reidemeister-relators}
\end{figure}
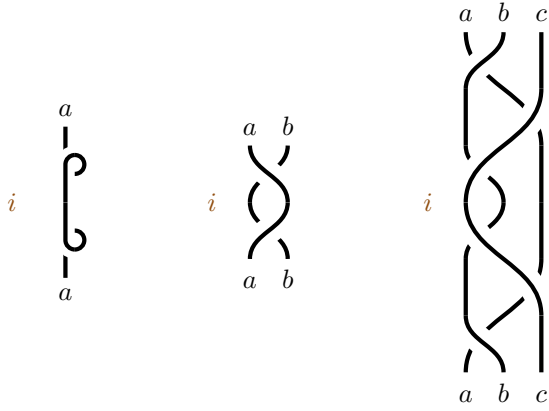

\begin{proof}[Proof of \cref{thm:tangle relations}]
  It is convenient to express the relations between the tangle generators as relators: we give a list of diagrams on which \(\zest{}\) is required to be the identity map.
  For example, the usual Reidemeister moves are given as relators in \cref{fig:Reidemeister-relators}.
  Since we are working with oriented tangles we need additional relations for critical points and other orientations of braidings.
  Let \(\mathfrak{R}\) be the set of shadow colorings \((i,R)\) of any relator diagram \(R\) appearing in \cref{fig:Reidemeister-relators} or in Figures 2.2, 2.3, 2.4, 2.8, or 2.8 of \cite[Chapter XII]{Kassel1995}. %
  
  A straightforward extension of \cite[Theorem XII.2.2]{Kassel1995} to \(\tangcat\) reduces the claim to showing that \(\zest{i, R}\) is an identity map for every \((i, R) \in \mathfrak{R}\).
  Since by definition every \(R\) is isotopic to an identity tangle we can apply \cref{thm:identities preserved}.
\end{proof}

\begin{remark}
  In principle one could repeat the computations in \cite{Delaney2020} and independently show that each \(\zest{i,R}\) is an identity map.
  This would give a proof that \(\zest{}\) is well-defined without reference to a background category \(\mathcal{C}\).
\end{remark}

\begin{theorem}
  \label{thm:base color independence}
  Let \(L\) be an \(A\)-colored link.
  Then unlike for tangles \(\zest{i,L}\) is independent of the choice of base color \(i\), so it gives an invariant of the underlying \(A\)-colored link.
\end{theorem}

\begin{proof}
  Let \(S\) be a single \(i\)-colored strand.
  Then by isotopy invariance
  \[
    \zest{e, S \du L}
    = 
    \zest{e, L \du S}.
  \]
  Factoring both sides gives
  \[
    \zest{e, S \du L}
    =
    \zest{e, S }
    \zest{i, L}
    =
    \id_{\tu}
    \zest{i, L}
  \]
  and
  \[
    \zest{e, L \du S}
    =
    \zest{e, L}
    \zest{i, S}
    =
    \zest{e, L}
    \id_{\tu}
    .
    \qedhere
  \]
\end{proof}
The idea behind the proof is given below:
\[
 \zest{ 
\begin{tikzpicture}[scale=.5, baseline=0cm]
\begin{scope}[xshift=-3.5cm]
 \draw[->] (0,2) node[above] {$i$} -- (0,-2);
 \draw[accent] (0,0) node[left, xshift=-.25cm] {$e$};
\end{scope}
\path[spath/save=trefoil]
(0,2) .. controls +(2.2,0) and +(120:-2.2) ..
(210:2) .. controls +(120:2.2) and +(60:2.2) ..
(-30:2) .. controls +(60:-2.2) and +(-2.2,0) .. (0,2);
\tikzset{
  every trefoil component/.style={draw},
  spath/knot={trefoil}{15pt}{1,3,5},
}
\draw[->] (-1.85, -.6);
\draw (0,2) node[above] {$a$};
\end{tikzpicture}
}
= 
\zest{
\begin{tikzpicture}[scale=.5, baseline=0cm]
\begin{scope}[xshift=3.5cm]
 \draw[->] (0,2) node[above] {$i$} -- (0,-2);
\end{scope}
\draw[accent] (-1.75,0) node[left, xshift=-.25cm] {$e$};
\path[spath/save=trefoil]
(0,2) .. controls +(2.2,0) and +(120:-2.2) ..
(210:2) .. controls +(120:2.2) and +(60:2.2) ..
(-30:2) .. controls +(60:-2.2) and +(-2.2,0) .. (0,2);
\tikzset{
  every trefoil component/.style={draw},
  spath/knot={trefoil}{15pt}{1,3,5},
}
\draw[->] (-1.85, -.6);
\draw (0,2) node[above] {$a$};
\end{tikzpicture}
}.
\]
The tangles on both sides are isotopic so by \cref{thm:tangle relations} they have the same value under \(\zest{}\).

\begin{proposition}
  \label{thm:qdims square to one}
  The value of a closed loop with either orientation is given by  \(\epsilon\):
  \[
    \zest{
      \begin{tikzpicture}[line width=1, baseline=-.125cm, scale=.5]
        \draw (0,0) circle (.5);
        \draw[to-] (-.5,.125) --(-.5,0) node[left] {$a$} ;
      \end{tikzpicture} 
    }
    =
    \zest{
      \begin{tikzpicture}[line width=1, baseline=-.125cm, scale=.5]
        \draw (0,0) circle (.5);
        \draw[-to] (-.5,0) node[left] {$a$} --(-.5,-.125)  ;
      \end{tikzpicture} 
    }
    =
    \epsilon(a).
    \qedhere
  \]
\end{proposition}

\begin{proof}
  Composing the morphisms in \cref{eq:coev-left-zest,eq:ev-right-zest} shows the value of a counterclockwise \(a\)-colored loop is \(\epsilon(a)\) for any value of the surrounding region color \(i\).
  For a clockwise loop we compose \cref{eq:coev-right-zest,eq:ev-left-zest} and similarly obtain  \(\epsilon(a)^{-1}\).
  By hypothesis \(\epsilon(a) = \epsilon(a)^{-1}\) for a ribbon zesting.
\end{proof}

\section{Efficient computation of the zest invariant}
\label{sec:zest invariant examples}
Now we show how to use the results of the previous section to read off the zest invariant of a tangle directly from an $A$-colored diagram.

\subsection{An example that shows how to compute the zest invariant}
\begin{example}
\label{ex: whitehead}
The closure of the braid below is an orientation of the Whitehead link. 

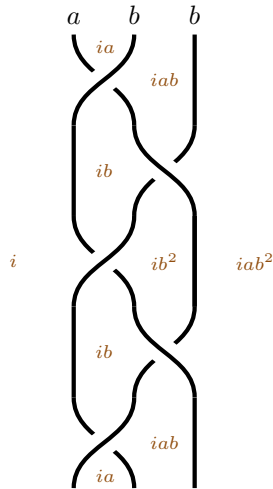
\begin{align*}
  \begin{tikzpicture}[line width=1, baseline=.8*4.5cm, scale=.8]
    \begin{scope}[xscale=-1,xshift=-1cm]
    \draw (0,0)  \br (1,1.5) ;
    \draw[white, line width=10] (1,0) \br (0,1.5);
    \draw[%
    ] (1,0)  \br (0,1.5);
    \draw (-1,0)--(-1,1.5);
    \end{scope}
    \begin{scope}[xshift=1cm,yshift=1.5cm]
    \draw (0,0)  \br (1,1.5) ;
    \draw[white, line width=10] (1,0) \br (0,1.5);
    \draw[%
    ] (1,0)  \br (0,1.5);
    \draw (1,1.5)--(1,3);
    \draw[%
    ] (-1,0)--(-1,1.5);
    \end{scope}
    \begin{scope}[xscale=-1,xshift=-1cm,yshift=3cm]
    \draw[%
    ]  (0,0)  \br (1,1.5) ;
    \draw[white, line width=10] (1,0) \br (0,1.5);
    \draw(1,0)  \br (0,1.5);
    \draw (-1,0)--(-1,1.5);
    \end{scope}
    \begin{scope}[xshift=1cm,yshift=4.5cm]
    \draw[%
    ] (0,0)  \br (1,1.5) ;
    \draw[white, line width=10] (1,0) \br (0,1.5);
    \draw[%
    ] (1,0)  \br (0,1.5);
    \draw (-1,0)--(-1,1.5);
    \end{scope}
    \begin{scope}[xscale=-1,xshift=-1cm,yshift=6cm]
    \draw[] (0,0)  \br (1,1.5) node[above] {$a$};
    \draw[white, line width=10] (1,0) \br (0,1.5);
    \draw[%
    ] (1,0)  \br (0,1.5) node[above] {$b$};
    \draw (-1,0)--(-1,1.5) node[above] {$b$}; 
    \end{scope}
    \begin{scope}[accent]
    \draw (-1,3.75) node {$\scriptstyle i$};
    \draw (.525,7.3) node {$\scriptstyle ia$};
    \draw (.525,.2) node {$\scriptstyle ia$};
    \draw (.5,5.25) node {$\scriptstyle ib$};
    \draw (.5,2.25) node {$\scriptstyle ib$};
    \draw (1.5,6.75) node {$\scriptstyle iab$};
    \draw (1.5,3.75) node {$\scriptstyle ib^2$};
    \draw (1.5,.75) node {$\scriptstyle iab$};
    \draw (3,3.75) node {$\scriptstyle iab^2$};
    \end{scope}
  \end{tikzpicture}
\end{align*}

The value of the braid under \(\zest{}\) is an endomorphism of
\[
  \zest{i; a, b, b}
  =
 \lambda(i, a) \otimes \lambda(ia, b) \otimes \lambda(iab, b)
\]
specifically
\begin{align}
\label{eq: whitehead braid}
  \phi_{i}(b, a)^{-1} \circ 
  \phi_{ib}(b,a) \circ
  \phi_{i}(b,b)^{-1} \circ
  \phi_{ib}(a,b) \circ
  \phi_{i}(a, b)^{-1}
\end{align}
where by abuse of notation we have suppressed the tensor products with identity morphisms. 
On the closure of the braid the value of \(\zest{}\) is
\[
 \epsilon(a)\epsilon(b)^2 \times
  \phi_{i}(b, a)^{-1} \circ 
  \phi_{ib}(b,a) \circ
  \phi_{i}(b,b)^{-1} \circ
  \phi_{ib}(a,b) \circ
  \phi_{i}(a, b)^{-1}
\]
where the additional factors of \(\epsilon\) come from e.g. cupping off to the right using e.g.~\cref{eq:coev-left-zest} and \cref{eq:ev-right-zest}. 
\end{example}
Note that the value on the trace of the braid only ever differs by a sign from the value of the underlying braid.\note{
This is consistent with \cite[Equation 11]{Delaney2021} in light of \cref{sec:ribbon zesting conditions}.
}

\subsection{The numerical value of the zest invariant}
\label{sec:numerical}
To this point we have discussed the zest invariant $\zest{}$ as taking values in the 2-morphisms of \(\delooping \Inv(\mathcal{C})\)
(the 1-morphisms of \(\Inv(\mathcal{C})\)).
However, in practice one works with the associative and braided zesting data $\nu$ and $t$ concretely as \(\kkm\)-valued functions rather than morphisms, in which case the evaluation of $\zest{}$ on a tangle is a scalar in \(\kkm\).

In slightly more detail, one starts with \(\lambda: A \times A \to \Inv(\mathcal{C}_e)\) which is a symmetric 2-cocycle \emph{on the level of objects}. That is, 
\begin{equation*}
  \begin{aligned}
    \lambda(a,b) \otimes \lambda(ab,c) &= \lambda(b,c) \otimes \lambda(a,bc) \\
    \lambda(a,b) &= \lambda(b,a)
  \end{aligned}
  \quad \text{ for all } a,b, c \in A.
\end{equation*}
Then 
\[
  \nu(a,b,c) \in \End(\lambda(a,b) \otimes \lambda(ab,c))=\End(\lambda(b,c) \otimes \lambda(a,bc))
\]
and
\[
  \nu(a,b,c) = \langle \nu(a,b,c) \rangle 
  \id_{\lambda(a,b) \otimes \lambda(a,bc)}
  .
\]
Similarly, \(t(a,b) \in \End(\lambda(a,b))\) and \(t(a,b) = \langle t(a,b) \rangle \id_{\lambda(a,b)}\). By abuse of notation, we conflate \(\nu\) with \(\langle \nu \rangle\) and \(t\) with \(\langle t \rangle \). Since \(\epsilon\) was introduced as a scalar-valued function from the get-go there is no change.
Finally, by imposing the equations
\begin{equation*}
  \lambda(a,b) \otimes \lambda(c,d)
  = 
  \lambda(c,d) \otimes \lambda(a,b)
\end{equation*}
one can identify the braiding \(\beta\) on the strict subcategory $\mathcal{C}_e \cap \mathcal{C}_{pt} \simeq \mathcal{C}(B,q)$ with a scalar-valued function.
Given these identifications one views the ribbon zesting data as consisting of cochains
\begin{align*}
  \nu &\colon A \times A \times A \to \kkm  \\
  t &\colon A \times A \to \kkm\\
  \epsilon &\colon A \to \{\pm 1\}\\
\end{align*} 
satisfying the equations
\begin{align}
\label{eqs: zesting equations}
  \nu(b,c,d) \nu(a,bc,d) \nu(a,b,c) &= \nu(a,b,cd) \beta^{-1}_{\lambda(a,b),\lambda(c,d)} \nu(ab,c,d) \\
  \nu(b,c,a)t(a,bc)\nu(a,b,c) &= t(a,c)\nu(b,a,c)t(a,b)\\
  \omega(a,b;c)\nu(c,a,b)^{-1}t(ab,c)\nu(a,b,c)^{-1}&=t(a,c)\nu(a,c,b)^{-1}t(b,c) \\
  \epsilon(ab)&=\dim(\lambda(a,b)) \epsilon(a)\epsilon(b)
\end{align}
for all \(a, b, c \in A\).
These equations can essentially be read off of the associative, braided\note{
When zesting with respect to a non-universal grading of $\mathcal{C}$, there is a similar way to unpack the data of the $j$ isomorphisms, see \cref{rem:dropping j}. 
}
, and ribbon zesting conditions from \cref{def:zesting-conditions}.
The one subtlety is that the braiding on the (strict) subcategory $\mathcal{C}_e \cap \mathcal{C}_{pt} \simeq \mathcal{C}(B,q)$ is given by the bicharacter $\beta$ on an abelian group $B$ associated to a quadratic form $q$ in the usual way per the classification of pointed braided fusion categories via metric groups \cite[Section 8.4]{EGNO2015}. 

In particular the morphisms $\phi_i(a,b)^{\pm 1}$ introduced in \cref{eq: def phi} as the value of $\zest{}$ on crossings become scalar-valued functions
\[
\phi_i(a,b)= \nu(i,b,a)^{-1}t(a,b)\nu(i,a,b).
\]
The point of these identifications is that now one can compute a numerical value for \(\zest{}\) on an \(A\)-shadow-colored tangle by simply multiplying the numerical values assigned to crossings, cups, and caps. For example, we revisit \cref{ex: whitehead}:

\begin{example}
  \label{ex: whitehead numerical}
  The numerical value of the zest invariant of the braid from \cref{ex: whitehead} comes from multiplying the scalar factors of $\phi$ at each crossing, or equivalently by reading \cref{eq: whitehead braid} as a product of evaluations of scalar-valued functions. One finds that all the contributions from $\nu$ terms and all but one of the $t$ terms cancel to give
  \[
  \phi_{i}(b, a)^{-1}
  \phi_{ib}(b,a)
  \phi_{i}(b,b)^{-1}
  \phi_{ib}(a,b)
  \phi_{i}(a, b)^{-1} = t(b,b)^{-1}
  \]
  One can check that this locally-computed formula is consistent with the computation in \cite[Equation 26]{Delaney2021} after taking into account the change in notation.\note{Regrettably there is a typo in the equation in the arXiv version; compare with the factor of \(t(i,i)^{-1}\).
  }

  Taking the closure, we find the numerical value of the zest invariant of our colored, oriented Whitehead link to be
  \[
    \epsilon(a)\epsilon(b)^2 \phi_{i}(b, a)^{-1}
  \phi_{ib}(b,a)
  \phi_{i}(b,b)^{-1}
  \phi_{ib}(a,b)
  \phi_{i}(a, b)^{-1} = \epsilon(a)t(b,b)^{-1}.
  \]
  This is consistent with \cite[Equation 26]{Delaney2021}, which assumed that \(\mathcal{C}\) was unitary, in which case \(\epsilon\) is a homomorphism and \(\epsilon(ab^2)=\epsilon(a)\).
\end{example}

For links in particular, thinking of \(\zest{}\) as a numerical invariant gives a connection between our construction and the link invariants arising from quandle cohomology, specifically to \defemph{shadow rack cocycle invariants}. 

\subsection{Zest invariants and shadow rack cocycle invariants}
\label{sec:quandle invariants}
Our description of the zest invariant of framed links is quite similar in spirit to the construction of shadow rack cocycle invariants \cite{Carter2003}.
We recall a special case in \cref{sec:shadowrackcocycleinvariants}.
We then show in \cref{thm: zest invariant is a cocycle invariant} that when $\mathcal{C}$ and $\mathcal{C}^{\zeta}$ are pseudo-unitary the associated zest invariants are exactly given by certain shadow rack cocycle invariants.
This represents a new connection between invariants that arise from discrete algebraic structures and those that arise from topological quantum field theories.

\subsubsection{Shadow rack cocycle invariants from an abelian group}
\label{sec:shadowrackcocycleinvariants}
We recall \cite{zbMATH00165698} that a \defemph{rack} $X$ is a set with a binary operation $\qn \colon X \times X \to X$ such that
\begin{enumerate}
  \item for all $a,b \in X$ there is a unique $c \in X$ such that $a = c \qn b$, and
  \item for all $a,b,c \in X$, \((a \qn b) \qn c = (a \qn c) \qn (b \qn c)\).
\end{enumerate}

We can think of a rack as describing the algebraic properties of colorings of tangle diagrams:
arcs are assigned elements of \(X\) subject to a relation involving \(\qn\) at each crossing; the rule for positive crossings is in \cref{fig: shadow color diagram}.
The axioms above ensure that Reidemeister II and III moves (respectively) act bijectively on the colorings.

A \defemph{shadow rack} \((X,Y)\) is a rack \(X\) and a set \(Y\) with an action $\cdot \colon Y \times X \to Y$ satisfying
\[
(y \cdot a) \cdot b = (y \cdot b \qn a) \cdot a \text{ for all } y \in Y \text{ and } a,b \in X.
\]
This definition is designed so that colorings of a tangle diagram by \(X\) extend to colorings of its regions by \(Y\) with the colors of adjacent regions satisfying a relation at each edge.
The condition above ensures the colorings are well-defined around each crossing, as in \cref{fig: shadow color diagram}.

\begin{marginfigure}
  \begin{tikzpicture}[line width=1, baseline=12.5, scale=2] 
    \draw[<-] (0,0) node[below] {$b \qn a$} \br (1,1.5) node[above] {$b$};
    \draw[white, line width=10] (1,0) \br (0,1.5);
    \draw[<-] (1,0) node[below] {$a$} \br (0,1.5) node[above] {$a$};
    \draw (-0.25,.75) node {$y$};
    \draw (1.25,.75) node {$(y \cdot a) \cdot b$};
    \draw (0.5,1.25) node {$y \cdot a$};
    \draw (0.5,.25) node {$y \cdot (b \qn a)$};
  \end{tikzpicture}
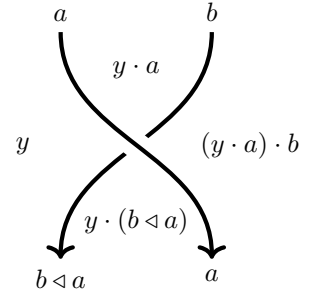
  \caption{Shadow colorings around a crossing.}
  \label{fig: shadow color diagram}
\end{marginfigure}

\begin{example}
  \label{ex: conjugation quandle}
  Let \(G\) be a group.
  Conjugation
  \begin{equation*}
    g \qn h \defeq g h g^{-1}
  \end{equation*}
  makes \(G\) a rack.
  If \(S\) is a set with \(G\)-action \((s, g) \mapsto s \cdot g\) then \((G,S)\) is a shadow rack because
  \begin{equation*}
  (s \cdot h \qn g) \cdot g
  =
  (s \cdot g h g^{-1}) \cdot g
  =
  s \cdot gh
  =
  s \cdot g \cdot h.
  \qedhere
  \end{equation*}
\end{example}

\begin{example}
  \label{ex: abelian shadow}
  In  \cref{ex: conjugation quandle} one can take \(G = A\) to be an abelian group so \(\qn\) is trivial.
  Setting \(Y = A\) and \(a \cdot b = ab\) makes \((A,A)\) a shadow rack.
\end{example}
This simple example of a shadow rack is the one underlying our construction; shadow rack colorings of tangle diagrams by \((A,A)\) are precisely the same as the shadow colorings of \cref{def: shadow coloring}.

One can define cohomology groups for racks whose \(2\)-cocycles give link invariants \cite{Carter2003}.
There are similar cohomology groups for shadow racks \cite{Inoue2014,Cazet2022}.
We give the definitions exclusively for the shadow rack \((A,A)\) of \cref{ex: abelian shadow}.

\begin{definition}
  For the shadow rack \((A,A)\) of \cref{ex: abelian shadow} a \defemph{shadow rack \(2\)-cocycle} with values in \(\kkm\) is a function \(\omega: A \times A^2 \to \kkm\) satisfying the \defemph{shadow rack \(2\)-cocycle condition}
  \begin{equation}
    \label{eq:cocycle relation}
    \frac{
      \omega_{i}(a_2, a_3)
    }{
      \omega_{ia_1}(a_2, a_3)
    }
    \left[
      \frac{
        \omega_{i}(a_1, a_3)
      }{
        \omega_{ia_2}(a_1, a_3)
      }
    \right]^{-1}
    \frac{
      \omega_{i}(a_1, a_2)
    }{
      \omega_{ia_3}(a_1, a_2)
    }
    =
    1
    \quad\text{for all } i,a_1,a_2, a_3 \in A.
  \end{equation}
\end{definition}

Then if \(D\) is an \(A\)-colored tangle diagram, the \defemph{weight} of a crossing \(c\) of \(D\) is defined to be 
\[
  W(c, \omega) =
  \begin{cases}
    \omega_{i}(a,b) & c \text{ positive}
    \\
    \omega_{i}(b,a)^{-1} & c \text{ negative}
  \end{cases}
\]
where \(i\) is the left-hand color and \(a,b\) are the incoming colors, as shown below:

\begin{align}
\label{eq: weight rules}
  W \left (
  \begin{tikzpicture}[line width=1, baseline=12.5, scale=.75] 
    \draw[<-] (0,0) node[below] {$b$} \br (1,1.5) node[above] {$b$};
    \draw[white, line width=10] (1,0) \br (0,1.5);
    \draw[<-] (1,0) node[below] {$a$} \br (0,1.5) node[above] {$a$};
    \draw (-.5,.75) node {$i$};
  \end{tikzpicture}, \quad \omega
  \right )
  = \omega_i(a,b)
  \qquad
  W \left (
  \begin{tikzpicture}[line width=1, baseline=12.5, scale=.75] 
    \draw[<-] (1,0) node[below] {$a$} \br (0,1.5) node[above] {$a$};
    \draw[white, line width=10] (0,0) \br (1,1.5);
    \draw[<-] (0,0) node[below] {$b$} \br (1,1.5) node[above] {$b$};
    \draw (-.5,.75) node {$i$};
  \end{tikzpicture}, \quad \omega
  \right )
  = \omega_i(b,a)^{-1}
\end{align}
These assignments are also called \defemph{Boltzmann weights} \cite{Carter2003}.
(Recall that we have taken the opposite convention on crossing signs to the one in the quandle cohomology literature.)

\begin{definition}
  Let \(D\) be a shadow colored tangle diagram.
  Its \defemph{shadow rack cocycle invariant} is given by
  \[
    Z(D,\omega) \defeq \prod_{c \in D} W(c, \omega). \qedhere
  \]
\end{definition}

\begin{theorem}
  The scalar \(Z(D, \omega)\) is an invariant \(Z(L, \omega)\) of the underlying colored,%
  \note{
    When \(A\) is finite the set of \(A\)-colorings of \(L\) is also finite, so one can eliminate the dependence on the coloring by taking a sum or product of the \(Z(L)\) over all colorings.
    This approach is common in the literature, but in the context of zesting it is more natural to take the \(A\)-coloring as data.
  }
  oriented, framed link \(L\) of \(D\).
\end{theorem}

\begin{proof}
  It is clear the weights are chosen so that \(Z(D, \omega)\) is invariant under the RII move by definition. Invariance under the RIII move is equivalent to the cocycle condition \eqref{eq:cocycle relation}, as one can check by comparing the assignment of weights to the braids below.
  
 \[
   Z \left (
    \begin{tikzpicture}[line width=1, scale=.75,baseline=1.6875cm] 
    \draw[<-] (0,0) \br (2,3);
    \draw[white, line width=10] (1,0) \br (0,1.5);
    \draw[<-] (1,0) \br (0,1.5);
    \draw[<-] (2,0) --(2,1.5);
    \begin{scope}[yshift=3cm]
    \draw (0,0) \br (1,1.5);
    \draw (2,0) --(2,1.5);
    \end{scope}
    \begin{scope}[yshift=1.5cm]
    \draw (0,0)--(0,1.5);
    \draw[white, line width=10] (2,0) \br (0,3);
    \draw (2,0) \br (0,3);
    \end{scope}
    \draw (0,0) node[below] {$\phantom{b}a\phantom{b}$};
    \draw (1,0) node[below] {$b$};
    \draw (2,0) node[below] {$\phantom{b}c\phantom{b}$};
    \draw (0,4.5) node[above] {$a$};
    \draw (1,4.5) node[above] {$b$};
    \draw (2,4.5) node[above] {$c$};
    \draw (-1,2.25) node[] {$i$};
    \draw (.75, 2.25) node[] {$ib$}; 
  \end{tikzpicture}
  \right ) = \omega_i(b,c) \omega_{ib}(a,c) \omega_i(a,b)
\]

\[
Z \left (
  \begin{tikzpicture}[line width=1, scale=.75, xscale = -1, baseline=1.6875cm] 
    \begin{scope}[yshift=1.5cm]
    \draw (2,0) \br (0,3);
    \draw (0,0)--(0,1.5);
    \end{scope}
    \draw[<-] (1,0) \br (0,1.5);
    \draw[<-] (2,0) --(2,1.5);
    \draw[white, line width=10] (0,0) \br (2,3);
    \draw[<-] (0,0) \br (2,3);
    \begin{scope}[yshift=3cm]
    \draw (2,0) --(2,1.5);
    \draw[white, line width=10] (0,0) \br (1,1.5);
    \draw (0,0) \br (1,1.5);
    \end{scope}
    \draw (0,0) node[below] {$\phantom{b}c\phantom{b}$};
    \draw (1,0) node[below] {$b$};
    \draw (2,0) node[below] {$\phantom{b}a\phantom{b}$};
    \draw (0,4.5) node[above] {$c$};
    \draw (1,4.5) node[above] {$b$};
    \draw (2,4.5) node[above] {$a$};
    \draw (1.5,1) node[] {$ic$};
    \draw (1.5, 3.5) node[] {$ia$}; 
    \draw (3, 2.25) node[] {$i$};
  \end{tikzpicture}
  \right ) = \omega_{ic}(a,b)\omega_i(a,c)\omega_{ia}(b,c)
\]
\end{proof}

\subsubsection{Braided zesting defines a shadow rack 2-cocycle} 
\label{sec:zesting defines rack cocycle}

Now we show that the assignment of weights to crossings defined by zesting defines a shadow rack 2-cocycle.

Consider ribbon zesting data \(\zeta=(\lambda, \nu, t, \epsilon)\) for a ribbon fusion category \(\mathcal{C}\), where $\nu$, $t$, and $\epsilon$ are viewed as scalar-valued functions as explained in \cref{sec:numerical}. Then for a fixed, $i,a,b \in A$, recall from \cref{eq:zested braiding} that a positive crossing is assigned

\[
 \zest{
    \begin{tikzpicture}[line width=1, baseline=12.5, scale=.75] 
    \draw[<-] (0,0) node[below] {$b$} \br (1,1.5) node[above] {$b$};
    \draw[white, line width=10] (1,0) \br (0,1.5);
    \draw[<-] (1,0) node[below] {$\phantom{b}a\phantom{b}$} \br (0,1.5) node[above] {$a$};
    \draw[accent] (-.5,.75) node {$i$};
  \end{tikzpicture}
  }
  =
  \phi_{i}(a, b) \defeq \nu(i,b,a)^{-1} t(a,b)\nu(i,a,b) \in \kkm.
\]
 
and a negative crossing is assigned
\[
 \zest{
    \begin{tikzpicture}[line width=1, baseline=12.5, scale=.75] 
    \draw[<-] (1,0) node[below] {$\phantom{b}a\phantom{b}$} \br (0,1.5) node[above] {$a$};
    \draw[white, line width=10] (0,0) \br (1,1.5);
    \draw[<-] (0,0) node[below] {$b$} \br (1,1.5) node[above] {$b$};
    \draw[accent] (-.5,.75) node {$i$};
  \end{tikzpicture}
  }
  =
  \phi_{i}(a, b)^{-1} \defeq \nu(i,a,b)^{-1} t(a,b)^{-1}\nu(i,b,a) \in \kkm.
\]

\begin{theorem}
\label{thm: zest invariant is rack cocycle invariant}
  The function $\phi_i(a,b): A \times A^2 \to \kkm$ is a shadow rack 2-cocycle.
\end{theorem}

This theorem is a special case of \cref{thm:tangle relations}: the rack \(2\)-cocycle condition is equivalent to checking that \(\zest{}\) maps the Reidemeister III relator to an identity map.
More constructively, in \cref{sec:proof of cocycle condition} we give a direct, but somewhat tedious proof that the shadow rack 2-cocycle condition follows from the associative zesting condition \cref{eq:associative zesting} and the braided zesting condition \cref{eq:braided zesting I}.

The data of a braided zesting defines a shadow rack 2-cocycle but in general the zest invariant of a framed link will only be equal to the shadow rack 2-cocycle invariant up to a sign.
This is because of the nontrivial contributions to the zest invariant from cups and caps, consistent with the fact that the shadow rack 2-cocycle only sees crossing data and does not involve the ribbon zesting data $\epsilon$. 
One can track the contributing factors of \(\epsilon\) from maxima and minima, but if $D$ is an $A$-colored link diagram given as the trace closure of a (downward oriented) $m$-strand braid, it is simpler to state
\[
  \zest{D} = \prod_{j=1}^m \epsilon(a_j) \prod_{c \in D} \phi_{i_c}(a_c,b_c)^{\sigma(c)} = \prod_{j=1}^m \epsilon(a_j) Z_\phi(D),
\]
where $a_j$ is the color of the $j$th strand of the underlying braid.
The following theorem is an immediate corollary of this discussion.

\begin{theorem}
\label{thm: zest invariant is a cocycle invariant}
Let \(\zeta=(\lambda,\nu,t, 1)\) be any ribbon zesting data with $\epsilon \equiv 1$. Then 
\[
  \zest{L} = Z(L, \phi),
\]
i.e.~the zest invariant of an oriented, framed, $A$-colored link $L$ is the shadow rack cocycle invariant that arises from $A$ and the Boltzmann weights \(\phi_i(a,b)\).
\end{theorem}

\begin{corollary}
\label{cor: pseudounitary}
Let \(\mathcal{C}\) and \(\mathcal{C}^{\zeta}\) be pseudo-unitary ribbon fusion categories related by ribbon zesting. Then the zest invariant of an \(A\)-colored framed link is equal to the shadow-rack cocycle invariant:
\[
  \zest{L} = Z_{\phi}(L). \qedhere
\]
\end{corollary}
\begin{proof}
  By \cref{thm:ribbon condition reformulation}(c) \(\mathcal{C}^{\zeta}\) is pseudo-unitary if and only if \(\epsilon(a) = 1\) for all \(a \in A\).
  This can be read directly from \cref{thm:qdims square to one} or from the formulae for the zested trace in \cref{eq: left zested trace,eq: right zested trace}.
\end{proof}

That is, the zest invariant coincides with the shadow rack cocycle invariant whenever the zesting relates two pseudo-unitary ribbon fusion categories, i.e.~whenever the quantum dimensions of simple objects are all positive. For most applications to physics and topological quantum computing one is only interested in \emph{unitary} modular fusion categories, and for those purposes there is no significant loss of generality in assuming pseudo-unitarity, which is weaker than unitarity. 

\subsubsection{Zesting and shadow rack cohomology} 
\label{sec: zest cohomology}

We showed above that (up to a small correction) the zest invariant can be computed as a shadow rack cocycle invariant.
This suggests there may be a cohomology theory underlying zesting.\note{
Indeed there are some group-cohomological ways to classify different pieces of the zesting data in \cite{Delaney2020}, but a full classification and holistic cohomology theory is not known.  
}
To further support this we show that cohomologous \(2\)-cocycles correspond to equivalent zesting data.

A \defemph{shadow rack \(2\)-coboundary} is a function of the form
\[
  (i; a_1, a_2) \mapsto
  \frac{
    \tau_{i}(a_1)
    \tau_{ia_1}(a_2)
    }{
    \tau_{ia_2}(a_1)
    \tau_{i}(a_2)
  }
\]
for some \(\tau : A \times A \to \kkm\).
It is easy to see that changing a shadow rack 2-cocycle \(\omega\) by a \(2\)-coboundary does not change the value of \(Z(L, \omega)\) for any link \(L\).

The analogous notion for zesting data is given in \cite[Definition 3.8(2)]{Delaney2020}.
We say the associative zesting data \((\lambda, \nu)\) and \((\lambda', \nu')\) are \defemph{equivalent} if there is a family of isomorphisms \(g(a_1, a_2) : \lambda(a_1, a_2) \to \lambda'(a_1, a_2)\) with
\[
   \left (  (g(a_2, a_3) \otimes g(a_1, a_2 a_3)) \right ) \circ \nu(a_1, a_2, a_3)
  =
  \nu'(a_1, a_2, a_3)
  \circ \left ( g(a_1, a_2) \otimes g(a_1 a_2, a_3) \right ).
\]
Thus when one works with $\nu$ as \(\kkm\)-valued 3-cochain on $A$ the data of such isomorphisms $g$ becomes a $2$-cochain on $A$, and two choices of associative zesting data $\nu$ and $\nu'$ are equivalent if there exists $g:A \times A \to \kkm$ such that
\[
\nu'(a_1,a_2,a_3) = \frac{g(a_2,a_3)g(a_1,a_2a_3)}{g(a_1,a_2)g(a_1a_2,a_3)} \nu(a_1,a_2,a_3),
\]
in which case we can quickly check that their associated shadow rack 2-cocycles $\phi$ and $\phi'$ are related by
\[
\phi'_i(a_1,a_2) = \phi_i(a_1,a_2) \frac{g(i,a_2)g(ia_2,a_1)}{g(i,a_1)g(ia_1,a_2)}.
\]
To do this we need the fact that \(g(a,b)=g(b,a)\), which holds because $\lambda$ is a symmetric \(2\)-cocycle on $A$.
This shows that the \(2\)-cocycle \(\phi'\) induced by the equivalence of zesting data changes \(\phi\) by a \(2\)-coboundary.

One can similarly check that an equivalence of braided zesting data (with respect to fixed associative zesting data $(\lambda, \nu$) induces a shadow rack 2-coboundary. For a fixed $\lambda$ and $\nu$ two choices of braided zesting data $t$ and $t'$ are equivalent if they are related by an alternating bicharacter, i.e.~ a bicharacter $\rho:A \times A \to \kkm$ with $\rho(a,a)=1$ \cite[Definition 4.6]{Delaney2020}. 

\begin{proposition}
\label{prop: coboundary}
Let $\lambda$ be a 2-cocycle valued in $\Inv(\mathcal{C}_e)$ as in \cref{def:zesting-conditions}(a).
\begin{thmenum}
  \item If \(\nu\) and \(\nu'\) are equivalent choices of associative zesting data with respect to \(\lambda\), then the shadow rack 2-cocycles \(\phi\) and \(\phi'\) differ by a 2-coboundary.
  \item For a fixed choice of associative zesting data \(\nu\) and equivalent choices of braided zesting data \(t\) and \(t'\), the shadow rack 2-cocycles \(\phi\) and \(\phi'\) differ by a 2-coboundary.
    \qedhere
\end{thmenum}
\end{proposition}
\qedhere

It would be ideal to be able to compare triples of braided zesting data $(\lambda, \nu, t)$ and $(\lambda',\nu',t')$ simultaneously, but this is beyond our scope for now.

One might hope to use the tools of rack cohomology to classify zestings, since equivalent associative (or braided) zesting data correspond to the same cohomology class. 
Unfortunately, the converse does not appear to be true: not every shadow rack 2-cocycle on \(A\) comes from zesting.
The issue is that the braided zesting relations in \cref{def:zesting-conditions}(c) imply \eqref{eq:cocycle relation} but in general are stronger because they assert relations between \(\phi_{i}(ab, c)\), \(\phi_{i}(a,c)\), and \(\phi_{i}(b,c)\).

In any case, we have shown that certain rack cocycle invariants appear quite naturally in Reshetikhin-Turaev TQFTs related by zesting.
To our knowledge this represents a new connection between invariants of links that arise from discrete algebraic structures like quandles and racks and those that are computed by a $(2+1)$D TQFT or ribbon fusion category.
Another place where they make contact is in the $(2+1)$D Dijkgraaf-Witten TQFT, where the invariants of framed links colored by a restricted class of simple objects are given by quandle coloring invariants \cite[Section 4.3.3]{Bonderson2019}.

\subsection{Computing the zest invariant of a link is easy}
An almost immediate consequence of the local formula for the zest invariant of a link is that it is efficient to compute as a function of the size of the link. 
Let $D$ be any link diagram representing $L$. Let $n(D)$ denote the number of components of $D$, and $c(D)$ the number of crossings of $D$; we take the quantity $n(D)+c(D)$ to measure the size of $D$.

\begin{theorem}
\label{thm: efficient algorithm}
Given ribbon zesting data $\zeta=(\lambda,\nu,t,\epsilon)$ and a diagram $D$ of an $A$-colored link $L$, the zest invariant $\mathcal{J}_{\zeta}(L)$ can be computed in polynomial time as a function of the size of $D$.
\end{theorem}

The algorithm used in the computation of \(\mathcal{J}_{\zeta}(L)\) is simply the local formula for \(\mathcal{J}_{\zeta}(L)\), but there are some minor technical details needed to give a careful proof of \cref{thm: efficient algorithm}, and we give the full details shortly in \cref{sec:zest link algorithm proof}. The main significance of \cref{thm: efficient algorithm} is that the complexity of the exact computation of Reshetikhin-Turaev invariants of framed links colored by simple objects is preserved under ribbon zesting:

\begin{corollary}
\label{thm: RT link complexity preserved}
Let \(\mathcal{C}\) and \(\mathcal{C}^\zeta\) be ribbon fusion categories related by zesting with data \(\zeta=(\lambda,\nu,t,\epsilon)\), and suppose that \(L\) is a framed, oriented link colored by \(\text{Irr}(\mathcal{C})=\text{Irr}(\mathcal{C}^\zeta)\). Then \(\mathcal{F}_{\mathcal{C}^\zeta}(L)\) is polynomial-time Turing equivalent to \(\mathcal{F}_{\mathcal{C}}(L)\).
\end{corollary}
\qedhere

In other words, without knowing precisely the absolute complexity of computing Reshetikhin-Turaev invariants in \(\mathcal{C}\), one can conclude that it is the same as that for \(\mathcal{C}^\zeta\). 

When \(\mathcal{C}\) and \(\mathcal{C}^\zeta\) are both modular and the Reshetikhin-Turaev construction defines invariants of closed 3-manifolds, one can also ask about the relative complexity of these problems. In \cref{sec:zest manifold invariant} we show that an analogous statement to \cref{thm: RT link complexity preserved} holds for \(A\)-manifold invariants, although we don't expect it to hold for 3-manifold invariants generally. In any case, from the point of view of topological quantum computation based on anyon braiding, the complexity of link invariants is more salient.

\subsubsection{Proof of \cref{thm: efficient algorithm}}
\label{sec:zest link algorithm proof}

Let $D$ be a diagram of an oriented, framed link $L$ with components colored by \(A\).
It will become clear later that without loss of generality we may assume that $D$ is diagrammatically non-split, i.e.~that the graph underlying $D$ is connected, since the value of the zest invariant is given by the product of the invariants of the connected components. 

To prove \cref{thm: efficient algorithm} we first show that converting from an arbitrary \(A\)-colored link diagram $D$ to an \(A\)-shadow-colored braid closure $\hat{B}$ can be done in polynomial time in the size of $D$, so that the local formula 

\[\mathcal{J}(\hat{B}) = \prod_{j=1}^m \epsilon(a_j) \prod_{c \in \hat{B}} \phi_{i_c}(a_c,b_c)^{\sigma(c)}  \]
from \cref{sec:zesting defines rack cocycle} can be used to compute the zest invariant of $L$.

\begin{theorem}[Vogel's algorithm \cite{Vogel1990}]
\label{lem:vogel}
A link diagram $D$ can be converted into a braid closure $\hat{B}$ in polynomial time in $n(D)+c(D)$. 
\end{theorem}
\begin{proof}  Using Vogel's algorithm, $D$ can be converted into a braid closure \(\hat{B}\) in at most $s(D)^2$ steps, where $s(D)$ is the number of Seifert circles. 
By \cite[Lemma 2]{Cui2016}, $s(D) \le n(D)+c(D)$, and so representing $L$ as a braid closure requires no more than $\left (n(D)+c(D)\right )^2$ steps. Moreover,  $c(\hat{B})=c(D)$ and $n(\hat{B})=n(D)+2$. 
\end{proof}

\begin{proposition}
\label{prop: induced shadow coloring}
Suppose that the $A$-colored link diagram is the trace closure $\hat{B}$ of a braid diagram $B$ on $m$-strands obtained from a link diagram $D$ via Vogel's algorithm. Then an $A$-shadow coloring can be computed in polynomial time in $n(D)+c(D)$.
\end{proposition}
\begin{proof}
By \cref{thm:base color independence}, $\mathcal{J}_{\zeta}(i,L)$ is independent of the choice of base color $i \in A$,
so for the purposes of computing $\mathcal{J}_{\zeta}(i,L)$ we may as well take $i=e$.
(However, the following argument works for any $i\in A$.)

Starting at the top of the braid $B$ and moving left to right, the region labels are computed strand by strand by multiplying the base color with the label of each strand according to its orientation in $m$ steps. By \cite{Vogel1990}, the number of strands $m$ in $\hat{B}$ is at most $n(D)+ (s(D) - 1)(s(D) - 2)$, which by \cite[Lemma 2]{Cui2016} is $O(\left ( n(D)+c(D)\right )^2)$. 

Then moving down the braid, each crossing introduces a new region in the interior of the diagram whose color is computed according to the rule in \cref{eq:zest object rules}, independent of the sign of the crossing. This takes $c(\hat{B})=c(D)$ steps.

Since the $A$-colored link diagram was in braid closure form to begin with, there are no additional region labels introduced in passing from $B$ to $\hat{B}$.
\end{proof}

We may thus happily assume our link diagrams are \(A\)-shadow colored braid closures without introducing any significant complexity beyond the computation of the zest invariant via our local formula as illustrated in  \cref{sec:zest invariant examples}.

\begin{remark} It is also possible to efficiently compute the induced shadow coloring before converting to a braid closure--for example by constructing the dual plane graph to $D$, picking a vertex and coloring it with an element of $A$, and then using a greedy algorithm to compute the induced vertex coloring from a spanning tree. However, then one has to do the bookkeeping of keeping track of how the shadow colored $L$ induces a shadow coloring of a braid closure in Vogel's algorithm, which we would like to avoid.
\end{remark}

It is clear from the structure of the formula
\[
  \mathcal{J}(\hat{B}) = \prod_{j=1}^m \epsilon(a_j) \prod_{c \in \hat{B}} \phi_{i_c}(a_c,b_c)^{\sigma(c)}
\]
that it is given by the product of $m$ signs $\{\pm 1\}$ and the product of $c(\hat{B})$ many evaluations of $\phi$ on shadow-colored crossings of $\hat{B}$.
This formula is computed from the zesting data $\nu$, $t$, and $\epsilon$, which arise as solutions to the algebraic equations \cref{eqs: zesting equations}, and whose values we may assume lie in a fixed finite extension $K$ of $\mathbb{Q}$ which is pre-computed independently of $D$ or $L$.
There are polynomial time algorithms to identify a primitive element $\alpha$ for which $K=\mathbb{Q}(\alpha)$ and to perform field operations in $\mathbb{Q}(\alpha)$, see for example \cite{Lenstra1992}.
As a consequence we can perform the multiplications in the formula for $\mathcal{J}(\hat{B})$ for free.

\section{Zested \texorpdfstring{\(A\)}{A}-manifold invariants}
\label{sec:zest manifold invariant}

So far we have been interested in the factorization of the Reshetikhin-Turaev invariants of homogeneously colored tangles under zesting.
When \(\mathcal{C}\) is modular the Reshetikhin-Turaev construction also provides invariants of 3-manifolds, and when \(\mathcal{C}^\zeta\) is also modular it is natural to ask whether a similar factorization holds for its \(3\)-manifold invariants. 

However, when \(\mathcal{C}\) is nondegenerate, not every choice of zesting data $\zeta$ with respect to the universal grading will result in \(\mathcal{C}^\zeta\) being nondegenerate (unless \(\mathcal{C}_{pt}\) is symmetric by \cite[Corollary 5.13]{Delaney2020}). Moreover, it is not hard to see from the description of Reshetikhin-Turaev \(3\)-manifold invariants via Kirby-colored link surgery diagrams that a factorization is likely too much to hope for in general:
while a link colored by a homogeneous object in $\mathcal{C}^\zeta$ will factorize (\cref{thm:manifold invariants factor}(b)), a link colored by an arbitrary object will not.
We explain this point later in \cref{rmk: zested 3-mfld invariant}.
Instead it is natural to look at invariants of \(3\)-manifolds with \(A\)-structure.
Under zesting these admit a natural factorization into a product of the original Reshetikhin-Turaev invariant and an \(A\)-manifold invariant that depends only on the zesting data.
We explore the invariants of \(A\)-manifolds that arise in this manner in the rest of this section, and comment on their complexity when applicable.

As a first step, we recall that for colored tangles it was important to restrict to colorings homogeneous with respect to the \(A\)-grading, which have the following topological interpretation.
(This works more generally for any abelian group but here we consider only the case where \(A\) is the universal grading group of our category \(\mathcal{C}\).)

\begin{proposition}
  \label{thm:A coloring cohomology}
  An \(A\)-coloring of a link is equivalent to a choice of cohomology class \(\omega \in \cohomol[1]{\comp L; A}\).
\end{proposition}
\begin{proof}
  Suppose \(L\) has \(n\) components.
  Recall that \(\homol[1]{\comp L; A} \iso \ZZ^n\), with generators corresponding to the meridians  \(\mer_1, \dots, \mer_n\) of  each component.
Because \(\homol[0]{\comp L; \ZZ} \iso \ZZ\) we can think of a class \(\omega \in \cohomol[1]{\comp L; A}\) as an element of \(\hom_\ZZ(\homol[1]{\comp L; A}, A) \iso A^n\), with the value of \(\omega\) on the meridian \(\mer_i\) of component \(i\) corresponding to the label on that component.
  This gives a bijective correspondence between labelings and cohomology classes.
  The orientations make sure that we do not confuse \(\omega(\mer_i)\) and \(\omega(\mer_i)^{-1}\).
\end{proof}

For \(3\)-manifolds this suggests the analogous notion to an \(A\)-grading:
\begin{definition}
  Let \(A\) be an abelian group.
  An \defemph{\(A\)-manifold} is a compact oriented \(3\)-manifold with a choice of cohomology class \(\omega \in \cohomol[1]{M; A}\).
  We call \(\omega\) an \defemph{\(A\)-structure} on \(M\).
\end{definition}
Our surgery invariants of \(A\)-manifolds are a special case of Turaev's \defemph{homotopy quantum field theory} \cite{Turaev2010}, which for completeness we summarize in \cref{sec:surgery colors and A-mfld invariants}.
First we recall some standard material on surgery presentations of \(A\)-manifolds.

\subsection{Surgery presentations of  \texorpdfstring{\(A\)}{A}-manifolds}

Surgery presentations of \(3\) and \(4\)-manifolds are a central tool in low-dimensional topology.
We recall the construction briefly and refer to the textbook \cite{Gompf1999} for more information.

\begin{definition}
  Let \(L\) be a framed link in \(S^3\) with \(c\) components.
  By removing a regular neighborhood of \(L\) we obtain a manifold \(E\) with boundary a disjoint union \(T_1 \du \cdots \du T_c\) of tori.
  The framing of the component \(K_i\) of \(L\) determines a curve \(\mathfrak{f}_i \subset T_i\) in the boundary component \(T_i\).
  We can glue a solid torus \(V_i\) to each \(T_i\) in such a way that the meridian of \(V_i\) is sent to \(\mathfrak{f}_i\), and this completely determines how to attach \(V_i\).
  (The meridian of \(V\) is homologically trivial, so another way to say this is that we attach \(V_i\) to kill the homology class of \(\mathfrak{f}_i\).)
  The resulting compact oriented \(3\)-manifold is said to be obtained by \defemph{surgery on \(L\)}.
  We denote it by \(\surg{L}\).
  It depends only on \(L\) and its framing (i.e.\@ the class of each \(\mathfrak{f}_i \in \homol[1]{T_i, \ZZ}\)).
\end{definition}

\begin{figure}[h]
  \label{fig:blow-up moves}
  \begin{center}
  \begin{tikzpicture}[scale=.5,line width=1,baseline=0cm, decoration={
    markings, mark=at position 0.5 with {\arrow{<}}}]
  \draw (-.5,0)--(-.5,-2.5);
    \draw (-1,0)--(-1,-2.5);
       \draw (1.25,0)--(1.25,-2.5);
  \draw[white, line width=10] [domain=230:320] plot ({2*cos(\x)},{sin(\x)});
  \draw[domain=20:330, postaction={decorate}] plot ({2*cos(\x)},{sin(\x)});
  \draw[white, line width=10] (-.5,0)--(-.5,2.5);
  \draw (-.5,0)--(-.5,2.5);
  \draw[white, line width=10] (-1,0)--(-1,2.5);
  \draw (-1,0)--(-1,2.5);
   \draw[white, line width=10] (1.25,0)--(1.25,2.5);
  \draw (1.25,0)--(1.25,2.5);
  \draw (1.73,-.5) to [out = 50, in = 150] (2.73,.5) to [out=-30, in= 330] (2.25,-.25);
  \draw (0.375,2) node {$\cdots$};
  \draw (0.375,-2) node {$\cdots$};
  \end{tikzpicture}
  \quad $\leftrightarrow$ \qquad
  \begin{tikzpicture}[scale=.5,line width=1,baseline=0cm]
  \draw (-.5,2.5)--(-.5,-2.5);
  \draw (-1,2.5)--(-1,-2.5);
  \draw (1.25,2.5)--(1.25,-2.5);
  \draw[fill=white] (-1.5, -.5) rectangle node {$+1$} (1.75,0.5);
  \draw (0.375,2) node {$\cdots$};
  \draw (0.375,-2) node {$\cdots$};
  \end{tikzpicture}
  \end{center}
    \caption{A negative blowup move, where the box labeled by $+1$ indicates a full twist. The positive move is obtained by changing the sign of both twists.}
\end{figure}
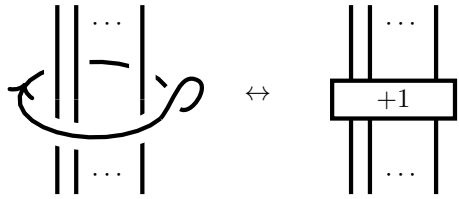

\begin{theorem}
  \label{thm:surgery facts}
  \begin{enumerate}
    \item All compact oriented \(3\)-manifolds arise as surgery on a framed link in \(S^3\).
    \item Two links \(L\) and \(L'\) give the same manifold if and only if they are related by framed isotopy and the positive and negative \defemph{blowup moves}.
    \item We say a blowup move is \defemph{special} if the removed unknot is unlinked from the rest of \(L\).
      If \(\surg{L}\) and \(\surg{L'}\) are homeomorphic then \(L\) and \(L'\) are related by isotopy, negative blowup moves, and special blowup moves of either sign.
      In other words we can assume any positive moves are special.
      \qedhere
  \end{enumerate}
\end{theorem}

Extending these techniques to \(A\)-manifolds is quite simple.
If \(M\) is obtained by surgery along a framed link \(L\) in \(S^3\) then there is a surjection \(\homol[1]{\comp{L}; \ZZ} \to \homol[1]{M; \ZZ}\), so we can describe our class  \(\omega \in \cohomol[1]{M;A}\) by labeling the components of \(L\) with elements of \(A\) as before.
The only new behavior is that \(\cohomol[1]{M; \ZZ}\) now has relations that \(\omega\) must respect.

\begin{definition}
  Let \(L = K_1 \cup \cdots \cup K_c\) be a framed link with \(c\) components.
  The \defemph{linking matrix} of \(L\) is the symmetric \(c \times c\) matrix with entries
  \[
    \lk(L)_{i,j}
    =
    \begin{cases}
      \text{ the linking number of \(K_i\) and \(K_j\) } & i \ne j
      \\
      \text{ the writhe of \(K_i\) } & i = j
    \end{cases}
  \]
\end{definition}

\begin{proposition}
  \(\lk(L)\) is a presentation matrix for \(\homol[1]{\surg{L}, \ZZ}\).
  This means that \(\homol[1]{\surg{L}, \ZZ}\) is the quotient of  \(\ZZ^c\) with relations given by the columns of \(\lk(L)\), or,  thinking of \(\lk(L)\) as the matrix of a linear map \(\ZZ^c \to \ZZ^c\),
  \[
    \homol[1]{\surg{L}, \ZZ} \iso \ZZ^c/ \im \lk(L).
    \qedhere
  \]
\end{proposition}

\begin{corollary}
  \label{thm:A colorings are kernel}
    An \(A\)-coloring \(\omega\) of a framed link \(L\) gives an \(A\)-structure on \(S^3(L)\) if and only if the columns of  \(\lk(L)\) lie in the kernel of \(\omega\).
    Another way to say this is to think of the matrix \(\lk(L)\) as an endomorphism of \(A^{c}\): then our claim is
  \[
    \cohomol[1]{\surg{L}, A} \iso \ker \lk(L).
    \qedhere
  \]
\end{corollary}
\begin{proof}
  \(\ZZ^c/\im \lk(L)\) is a cokernel, so when we pass to cohomology we instead use the kernel.
\end{proof}

\begin{example}
  With respect to the blackboard framing, the positively-linked Hopf link
  
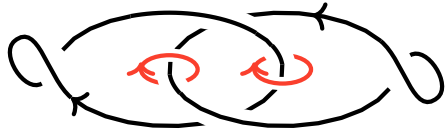
\begin{align*}
   \begin{tikzpicture}[scale=.75,line width=1,baseline=0cm, decoration={
    markings, mark=at position 0.15 with {\arrow{>}}}]
  \begin{scope}[xshift=-2cm,scale=-1]
  \draw[white, line width=10] [domain=250:290] plot ({2*cos(\x)},{sin(\x)});
  \draw[domain=30:340,<-] plot ({2*cos(\x)},{sin(\x)});
  \draw (1.73,.5) to [out = 320, in = 220] (2.73,-.5) to [out=40, in=30] (2.25,.25);
  \end{scope}
  \draw[white, line width=10] [domain=220:290] plot ({2*cos(\x)},{sin(\x)});
  \draw[domain=30:340, postaction={decorate}] plot ({2*cos(\x)},{sin(\x)});
  \draw (1.73,.5) to [out = 320, in = 220] (2.73,-.5) to [out=40, in=30] (2.25,.25);
    \draw[white, line width =5] (0,0) ellipse (.5cm and .25cm);
  \draw[slred] (0,0) ellipse (.5cm and .25cm);
    \draw[white, line width =5] (-2,0) ellipse (.5cm and .25cm);
  \draw[slred] (-2,0) ellipse (.5cm and .25cm);
   \draw[slred,->] (-.5,.125);
   \draw[slred,->] (-2.5,.125);
    \begin{scope}[xshift=-2cm,scale=-1]
  \draw[white, line width=10] [domain=220:290] plot ({2*cos(\x)},{sin(\x)});
  \draw[domain=220:290] plot ({2*cos(\x)},{sin(\x)});
    \draw[white, line width=10] [domain=185:205] plot ({2*cos(\x)},{sin(\x)});
  \draw[domain=185:205] plot ({2*cos(\x)},{sin(\x)});
  \end{scope}
  \draw[white, line width=10] [domain=185:205] plot ({2*cos(\x)},{sin(\x)});
  \draw[domain=185:205] plot ({2*cos(\x)},{sin(\x)});
  \end{tikzpicture}
\end{align*}
  (with meridians shown in \textcolor{slred}{red})
  has linking matrix
  \[
    \begin{bmatrix}
      1 & 1 \\
      1 & -1
    \end{bmatrix}
  \]
  A cohomology class \(\omega \in \homol[1]{S^3 \setminus L; A}\) is determined by two elements \(a_1, a_2 \in A\), the values on the meridians of each component.
  This determines an \(A\)-structure on \(\surg{L}\) if and only if
  \[
    \begin{bmatrix}
      1 & 1 \\
      1 & -1
    \end{bmatrix}
    \begin{bmatrix}
      a_1 \\
      a_2
    \end{bmatrix}
    =
    0
  \]
  that is if \(a_1 = a_2\) and \(2a_1 = 0\).
  We conclude that \(\homol[1]{\surg{L}; A} \iso A / 2A\).
\end{example}

\begin{theorem}
  \Cref{thm:surgery facts} generalizes in the obvious way to \(A\)-manifolds:
  \begin{thmenum}
    \item we can represent any \(A\)-manifold as surgery on an \(A\)-colored link whose coloring is compatible with the framing,
    \item any two presentations are related by isotopy and colored blow-up moves and their inverses, and
    \item when proving invariance it suffices to consider negative blow-up moves and special blow-up moves of both signs. \qedhere
  \end{thmenum}
\end{theorem}
In particular, we observe that when adding a new component during a blowup move its color is automatically determined by the condition in \cref{thm:A colorings are kernel}.

\subsection{Surgery colors and \texorpdfstring{\(A\)}{A}-manifold invariants}
\label{sec:surgery colors and A-mfld invariants}

We can now explain how to construct invariants of \(A\)-manifolds in terms of surgery presentations.

\begin{definition}
  Let \(\mathcal{C}\) be a faithfully \(A\)-graded ribbon fusion category.
  The (graded) \defemph{\(S\)-matrix} of \(\mathcal{C}\) is a square matrix \(s\) whose rows and columns are indexed by isomorphism classes of simple objects in the trivial component \(\mathcal{C}_e\).
  (For the usual \(S\)-matrix we would range over all simple isomorphism classes.)
  Its entries are the Reshetikhin-Turaev invariants 
  \[
    s_{X,Y} = \linkinv{\beta_{Y,X} \circ \beta_{X,Y}}{\mathcal{C}}
  \]
  of Hopf links colored by objects \(X,Y\) in the corresponding isomorphism classes.
  We say that \(\mathcal{C}\) is  \defemph{graded-modular} if the graded \(S\)-matrix is nonsingular.%
  \note{%
    These are a special case of Turaev's \cite{Turaev2010} modular \(A\)-crossed categories:
    we can view \(\mathcal{C}\) as an \(A\)-crossed ribbon category with trivial \(A\)-action.
  }
\end{definition}

If we forget the grading and take \(\mathcal{C}_{e}\) to be all of \(\mathcal{C}\) we recover the usual definition of modular category.
However, for a given category being modular and graded-modular are not necessarily equivalent. For example, any Ising-like modular fusion category $\mathcal{C}$ is $\ZZ/2\ZZ$-graded with $\mathcal{C}_e \simeq \Rep(\ZZ/2\ZZ)$, which is symmetric, hence not modular.

\begin{definition}
  \label{def:global-dim-and-blow-up}
  For a graded-modular category \(\mathcal{C}\), define
  \begin{align}
    \label{eq:blow-up-coeff}
    \Delta_{\pm}
    &\defeq
    \sum_{|X| = e}
    \theta_X^{\pm 1}
    \dim(X)
    \\
    \label{eq:global-dim}
    \globaldim
    &\defeq
    \sum_{|X| = e}
    \dim(X)^2
  \end{align}
  where the sums are over isomorphism classes of simple objects in degree \(e \in A\).
\end{definition}
These are graded versions of the parameters usually appearing in the surgery construction: for example, \(\globaldim\) is the global dimension of the trivial component \(\mathcal{C}_e\) (instead of the whole category).
Graded-modularity implies that none are zero and that \(\Delta_{+}\Delta_{-} = \globaldim^2\).

\begin{definition}
  Let \(\mathcal{C}\) be a graded-modular category.
  For each \(a \in A\) the \defemph{surgery color} is the formal linear combination
  \begin{equation}
    \surgcol a \defeq \sum_{|X| = a} \dim(X) X
  \end{equation}
  over all isomorphism classes of simple object in degree \(a\), weighted by their quantum dimensions.
\end{definition}

\begin{definition}
  \label{def:rt-manifold-definition}
  Let \((M, \omega)\) be a \(3\)-manifold with \(A\)-structure \(\omega\) presented as surgery on a framed \(A\)-colored link \((L, \omega)\) with \(n\) components.
  Let \(\mathcal{C}\) be a graded-modular category.
  Write \(\sigma^{+}(L)\) for the number of positive eigenvalues of the linking matrix of \(L\), and similarly \(\sigma^{-}(L)\) for the number of negative eigenvalues (so the signature is \(\sigma(L) = \sigma^{+}(L) - \sigma^{-}(L)\)).
  For \(\mathcal{C}\) a graded-modular category, we define the \defemph{Reshetikhin-Turaev invariant} of \((M, \omega)\) to be
  \begin{equation}
    \label{eq:rt-manifold-definition}
    \rtinv{M, \omega}{\mathcal{C}}
    \defeq
    \Delta_{+}^{-\sigma^{+}(L)}
    \Delta_{-}^{-\sigma^{-}(L)}
    \globaldim^{-1}
    \linkinv{L; \surgcol{\omega(1)}, \dots, \surgcol{\omega(n)}}{\mathcal{C}}
  \end{equation}
  where the right-hand side is the \(\mathcal{C}\)-coloring of \(L\) assigning a component \(i\) with \(A\)-coloring \(\omega(i)\) the surgery \(\mathcal{C}\)-color \(\surgcol{\omega(i)}\).
\end{definition}

\begin{remark}
  The surgery construction can be viewed as generating a compact oriented \(4\)-manifold \(W\): we think of \(L\) as lying in \(S^3 = \partial B^4\) and add \(2\)-handles instead of solid tori to get a manifold \(W\) with \(\partial M = \surg{L}\).
  The linking matrix of \(L\) is a presentation matrix for the intersection form of \(W\).
  The \((2+1)\)-dimensional Reshetikhin-Turaev surgery TQFT is the boundary theory of a \((3+1)\)-dimensional Crane-Yetter theory whose value on \(W\) is related to the signature term in \cref{eq:rt-manifold-definition}; adding it ensures that the Reshetikhin-Turaev invariant of \(M\) does not depend on the choice of \(W\) with \(\partial W = M\).
\end{remark}

\begin{theorem}
  \( \rtinv{M, \omega}{\mathcal{C}} \) is a homeomorphism invariant of \((M, \omega)\).
  In particular, it does not depend on the choice of surgery presentation \((L, \omega)\).
\end{theorem}
\begin{proof}
  See \cite[Theorem VII.2.3]{Turaev2010}.
  The essential idea is the same as in the ungraded case \cite{Reshetikhin1991}: because \(\linkinv{L}{\mathcal{C}}\) is an isotopy invariant of \(L\) it suffices to consider the blowup moves.
  By semi-simplicity of \(\mathcal{C}\) turn these can be reduced to the case of a single strand, which can be handled directly.
\end{proof}

\subsection{\texorpdfstring{\(A\)}{A}-manifold invariants under zesting}

\begin{definition}
  Suppose \((M, \omega)\) is obtained by surgery along a framed \(A\)-colored link \((L, \omega)\).
  For ribbon zesting data \(\zeta = (\lambda, \nu, t, \epsilon)\) we define
  \begin{equation}
    \label{eq:zest manifold inv def}
    \rtinv{M, \omega}{\zeta} \defeq 
    \zest{e, L, \omega}
    \prod_{i} \epsilon(\omega(i))
  \end{equation}
  where \(\omega(i) \in A\) is the color of the \(i\)th component of \(L\) and the product is over all components.
\end{definition}

Here we choose the base region color to be \(e\), but by \cref{thm:base color independence} any choice would give the same answer.
Below we check the blow-up move with arbitrary region colorings in order to show it holds regardless of where we apply it in a diagram of \(L\).

\begin{remark}
  \label{rem:surgery analogies}
  We can think of \(\rtinv{M, \omega}{\zeta}\) as an analogue of the surgery invariants in \cref{def:rt-manifold-definition}.
  We have the following table of analogies:
  \begin{center}
    \begin{tabularx}{\textwidth}{X X}
      set of iso classes of simples in degree \(a\)
      &
      the single element \(a\)
      \\
      quantum dimension of \(a\)
      &
      \(\epsilon(a)\)
      \\
      twist of \(a\) 
      &
      \(f(a) = t(a) \epsilon(a)\)
      \\
      surgery color for degree \(a\)
      &
      \(
        \surgcol{a}[\zeta] = \epsilon(a) a
      \)
      \\
      link invariant
      &
      \(\zest{L, \omega}\)
      \\
      manifold invariant
      &
      \(\rtinv{L, \omega}{\zeta}\)
      \\
      global dimensions
      &
      \(
      \Delta_{\pm}^{\zeta}
      =
      \globaldim^{\zeta}
      =1
      \)
    \end{tabularx}
  \end{center}
  To see the last claim, observe that for the trivial component \(t(e)=1\) and \(\epsilon(e) = 1\), so 
  \begin{align*}
    \Delta_{\pm}^{\zeta}
    &=
    f(1)^{\pm 1}
    \epsilon(1)
    =
    1
    \\
    \globaldim^{\zeta}
    &=
    \epsilon(1)^2
    =
    1.
  \end{align*}
  There are no sums above or in the surgery color \(\surgcol{a}[\zeta]\) because there is a single ``object'' in each degree.
  This analogy suggests that there is an algebraic object \(\mathcal{P}(\zeta)\) associated to the data \(\zeta\) that can be used to directly define surgery invariants.
\end{remark}

\begin{theorem}
  \label{thm:manifold invariants factor}
  Let \(\mathcal{C}^{\zeta}\) be a ribbon zesting of \(\mathcal{C}\) with universal grading group \(A\) using data \(\zeta\).
  \begin{enumerate}[label={(\alph*)}]
    \item If \(\mathcal{C}\) is graded-modular, then so is \(\mathcal{C}^{\zeta}\).
    \item Let \((M, \omega)\) be a \(3\)-manifold with \(A\)-structure \(\omega\) presented as surgery on a framed \(A\)-colored link \((L, \omega)\).
      Then
      \begin{equation*}
        \rtinv{M, \omega}{\mathcal{C}^{\zeta}}
        =
        \rtinv{M, \omega}{\mathcal{C}}
        \rtinv{L, \omega}{\zeta}
      \end{equation*}
    \item \(\rtinv{L, \omega}{\zeta}\) is an invariant of \((M, \omega)\): in particular, it is independent of the framed link \(L\) used to present \(M\).
      \qedhere
  \end{enumerate}
\end{theorem}
\begin{proof}
  (a) Graded-modularity is determined by the trivially-graded piece, which is invariant under zesting.

  (b) By definition
  \[
    \rtinv{M, \omega}{\mathcal{C}^{\zeta}}
    =
    \linkinv{L; \surgcol{\omega(1)}[\mathcal{C}^{\zeta}], \dots, \surgcol{\omega(n)}[\mathcal{C}^{\zeta}]}{\mathcal{C}^{\zeta}}
  \]
  where \(\omega(i)\) is the value of \(\omega\) on the (meridian of the) \(i\)th component.
  This is a \emph{homogeneous} sum, so
  \begin{align*}
    \rtinv{M, \omega}{\mathcal{C}^{\zeta}}
    &=
    \linkinv{L; \surgcol{\omega(1)}[\mathcal{C}^{\zeta}], \dots, \surgcol{\omega(n)}[\mathcal{C}^{\zeta}]}{\mathcal{C}^{\zeta}}
    \\
    &=
    \sum_{j_1 \in \indset{\omega(1)}}
    \cdots
    \sum_{j_n \in \indset{\omega(n)}}
    \qdim[\mathcal{C}^{\zeta}](X_{j_1})
    \cdots
    \qdim[\mathcal{C}^{\zeta}](X_{j_n})
    \linkinv{L, X_{j_1}; \cdots X_{j_n}}{\mathcal{C}^{\zeta}}
    \\
    &=
    \sum_{j_1 \in \indset{\omega(1)}}
    \cdots
    \sum_{j_n \in \indset{\omega(n)}}
    \qdim[\mathcal{C}](X_{j_1})
    \epsilon(a_{j_1})
    \cdots
    \qdim[\mathcal{C}](X_{j_n})
    \epsilon(a_{j_n})
    \zest{L, \omega}
    \linkinv{L, X_{j_1}; \cdots X_{j_n}}{\mathcal{C}}
    \\
    &=
    \rtinv{L, \omega}{\zeta}
    \rtinv{M, \omega}{\mathcal{C}}
  \end{align*}
  where we used \cref{thm:qdims square to one} to understand the change in quantum dimensions.

  (c)
  By \cref{thm:tangle relations} \(\zest{L, \omega}\) is an isotopy invariant, so it suffices to check the surgery moves.
  First we check the special blowup moves.
  An unlinked, \(\pm 1\)-framed unknot must always have color \(e \in A\).
  By \cref{thm:qdims square to one} the value of \(\zest{}\) on an \(e\)-colored unknot is \(\epsilon(e) = 1\) regardless of the adjacent region color, so adding or removing a \(\pm 1\)-framed unknot will not change the value of \(\zest{}\).
  This agrees with \cref{rem:surgery analogies}: the ``global dimension'' is \(1\).

  For the general case, write \(\operatorname{Bl}^{-}(a_1, \dots, a_n)\) for the \(A\)-colored diagram
  \begin{align*}
\text{Bl}^-(a_1,a_2, \ldots, a_n)=
    \begin{tikzpicture}[scale=.5,line width=1,baseline=-.875cm, decoration={
    markings, mark=at position 0.5 with {\arrow{<}}}]
  \draw (-1,2.5)--(-1,-5) node[below] {\small $a_1 \phantom{a_2}$};
  \draw (-.5,2.5)--(-.5,-5) node[below] {\small $a_2$};
  \draw (1.25,2.5)--(1.25,-5) node[below] {\small $a_n$};
  \draw[white, line width=10] [domain=230:320] plot ({2*cos(\x)},{sin(\x)});
  \draw[domain=20:330, postaction={decorate}] plot ({2*cos(\x)},{sin(\x)});
  \draw[white, line width=10] (-.5,0)--(-.5,2.5);
  \draw (-.5,0)--(-.5,2.5);
  \draw[white, line width=10] (-1,0)--(-1,2.5);
  \draw (-1,0)--(-1,2.5);
   \draw[white, line width=10] (1.25,0)--(1.25,2.5);
  \draw (1.25,0)--(1.25,2.5);
  \draw (1.73,-.5) to [out = 50, in = 150] (2.73,.5) to [out=-30, in= 330] (2.25,-.25);
  \draw (0.375,2) node {$\cdots$};
  \draw (0.375,-2) node {$\cdots$};
    \draw (0.375,-4.5) node {$\cdots$};
   \draw[fill=white] (-1.5, -3.75) rectangle node {$-1$} (1.75,-2.75);
  \end{tikzpicture}
  \end{align*}
  where the box indicates a full twist.
  Compatibility with the framing says that the closed component of \(\operatorname{Bl}^{-}(a_1, \dots, a_n)\) has color \(a^{- 1}\), where \(a = a_1 \cdots a_n\).
  Invariance of \(\rtinv{M}{\mathcal{C}}\) under the blow-up moves follows from showing that \(\linkinv{\operatorname{Bl}^{-}}{\mathcal{C}}\) (with all components colored by surgery colors) is an identity map times some appropriate powers of \(\Delta_{\pm}\) and \(\globaldim\).
  (One also needs to check \(\operatorname{Bl}^{+}\) in the special case with no encircled strands, which is trivial.)
  In parallel we need to show that \(\zest{i, \operatorname{Bl}^{-}(a_{1}, \dots, a_{n})}\) is \(\epsilon(a^{-1}) \prod_{j = 1}^{n} \epsilon(a_{j})\) times the identity map for every \(i, a_{1}, \dots, a_{n} \in A\).
  (Recall that the zest invariant \eqref{eq:zest manifold inv def} of a \(3\)-manifold includes extra factors of \(\epsilon\).)

  We use the same strategy as in \cref{thm:identities preserved}.
  Let \(B\) be a \(\mathcal{C}\)-coloring of \(\operatorname{Bl}^{-}(a_1, \dots, a_n)\) with surgery colors \(\surgcol{a_{i}}\) coloring the open strands and \(\surgcol{a^{-1}}\) on the closed component.
  Then blowup invariance of the \(A\)-manifold Reshetikhin-Turaev invariant gives
  \[
    \linkinv{B}{\mathcal{C}}
    =
    \Delta_{-}
    \id_{X_{1}} \otimes \cdots \id_{X_{n}}
  \]
  (More specifically: this claim is equivalent to the key step \cite[eq.~VII.2.3.a]{Turaev2010} in the proof of \cite[Theorem VII.2.3]{Turaev2010}.)
  If \(S\) is a single strand colored by a homogeneous object \(W\) in degree \(i\) we similarly have
  \[
    \linkinv{W \du B}{\mathcal{C}}
    =
    \Delta_{-}
    \id_{W} \otimes \id_{X_{1}} \otimes \cdots \id_{X_{n}}
  \]
  The same argument shows that
  \(
    \linkinv{W \du B}{\mathcal{C}^{\zeta}}
  \)
  is \(\Delta_{-}\) times an identity map.
  On the other hand by \cref{thm:factorization still holds}
  \begin{equation*}
    \linkinv{W \du B}{\mathcal{C}^{\zeta}}
    =
    \epsilon(a^{-1}) \prod_{j = 1}^{n} \epsilon(a_{j})
    \linkinv{W \du B}{\mathcal{C}}
    \warp
    \zest{i, \operatorname{Bl}^{-}(a_{1}, \dots, a_{n})}
  \end{equation*}
  which is only possible if 
  \begin{equation*}
    \epsilon(a^{-1}) \prod_{j = 1}^{n} \epsilon(a_{j})
    \zest{i, \operatorname{Bl}^{-}(a_{1}, \dots, a_{n})}
  \end{equation*}
  is an identity map as required.
  The extra factors of \(\epsilon\) appear just as in the proof of part (b):
  a surgery color \(\surgcol{b}\) includes factors of \(\dim Y\) for objects \(|Y| = b\) and their quantum dimensions transform by \(\epsilon(b)\) under zesting.
\end{proof}

\subsubsection{Efficient computation of the zest invariant of an \texorpdfstring{$A$}{A}-manifold}
\label{sec:zest manifold algorithm}

\cref{thm:manifold invariants factor} showed that, like simply colored tangles and links, the Reshetikhin-Turaev invariants of $3$-manifolds with $A$-structure differ by an invariant that depends only on the zesting data. The following corollary is obvious from the argument used in \cref{thm: efficient algorithm} to show that $\mathcal{J}_\zeta(L)$ can be computed in polynomial time.

\begin{corollary}
\label{cor: A-mfld invariant efficient} 
Let $\mathcal{C}$ be an $A$-modular fusion category. Then if $\zeta$ is ribbon zesting data for $\mathcal{C}$, the $A$-manifold invariants $\mathcal{Z}_{\mathcal{C}}(M,\omega)$ and  $\mathcal{Z}_{\mathcal{C}^\zeta}(M,\omega)$ are polynomial time Turing equivalent. 
\end{corollary}
\qedhere

For more general $3$-manifold invariants without $A$-structure, it is straightforward to compute the Reshetikhin-Turaev invariant $\mathcal{Z}_{\mathcal{C}^\zeta}(M)$ from $\mathcal{Z}_{\mathcal{C}}(M)$ and the zesting data $\zeta$ using the methods established throughout this paper. 
\begin{remark}
\label{rmk: zested 3-mfld invariant}
Because \(\surgcol{} = \sum_{a \in A} \surgcol{a}\) it is clear that
\begin{equation*}
  \rtinv{M}{\mathcal{C}} = \sum_{\omega \in \homol[1]{M; A}} \rtinv{M, \omega}{\mathcal{C}}
\end{equation*}
so that
\begin{equation}
  \rtinv{M}{\mathcal{C}^{\zeta}} = \sum_{\omega \in \homol[1]{M; A}} \rtinv{M, \omega}{\mathcal{C}} \rtinv{M, \omega}{\zeta}.
\end{equation}
Because of the summation over inhomogeneous colorings understanding the relative complexity is more intricate than for simply-colored links and $A$-manifolds.
\end{remark}

\printbibliography

\appendix

\section{New formulation of ribbon zesting data}
\label{sec:ribbon zesting conditions}
In this appendix we explain why our definition of ribbon zesting is equivalent to the original one; in the process we show braided zestings always admit twist zestings and recover a characterization of pseudo-unitary zestings due to \citeauthor{Galindo2023} \cite[Theorem 3.18]{Galindo2023}.

\begin{theorem}
  \label{thm:twist zesting conditions}
  Fix braided zesting data \(\zeta = (\lambda, \nu, t)\) for a ribbon fusion category \(\mathcal{C}\) and let \(f, \epsilon : A \to \kkm\) be functions with \(f(a) = t(a) \epsilon(a)\) for all \(a \in A\).
  (Recall \(t(a) \defeq \langle t(a,a) \rangle\).)
  Then \(f\) satisfies 
  \begin{equation}
    \label{eq:twist zesting f}
    f(ab) \omega(a,b; ab) \theta_{\lambda(a,b)}= f(a)f(b)t^{(2)}(a,b)
  \end{equation}
  for all \(a, b \in A\) if and only if \(\epsilon\) satisfies \cref{eq:twist zesting epsilon}
  \begin{equation*}
    \tag{\ref{eq:twist zesting epsilon}}
    \frac{
      \epsilon(ab)
    }{
      \epsilon(a) \epsilon(b)
    }
    =
    \dim \lambda(a,b)
  \end{equation*}
  for all \(a, b \in A\).
  When this is the case,
  \begin{equation}
    \label{eq:ribbon zesting f}
    f(a)=f(a^{-1})\omega(a, a^{-1}; a^{-1}) \theta_{\lambda(a,a^{-1})}
  \end{equation}
  for all \(a \in A\) if and only if \(\epsilon\) satisfies \cref{eq:ribbon zesting epsilon}
  \begin{equation*}
    \tag{\ref{eq:ribbon zesting epsilon}}
    \epsilon(a)^2 = 1
  \end{equation*}
  for all \(a \in A\).
\end{theorem}

The theorem follows from a computational lemma given at the end of this section.

In the original definition \cite{Delaney2020,Delaney2021} of zesting a twist zesting is a function \(f\) satisfying \cref{eq:twist zesting f} and a ribbon zesting is a twist zesting additionally satisfying \cref{eq:ribbon zesting f}, so this theorem establishes that our definition in \cref{sec:zesting summary} is equivalent.
The new, simpler characterization has some immediate applications:

\begin{corollary}
  \label{thm:ribbon condition reformulation}
  Let \(\mathcal{C}\) be a braided fusion category with a twist and \(\zeta = (\lambda, \nu, t)\) a braided zesting of \(\mathcal{C}\).
  Define a cocycle \(\delta \in \cohomol{A; \kkm}\) by \(\delta(a,b) = \dim(\lambda(a,b))\).
  \begin{thmenum}
    \item \(\zeta\) always extends to a twist zesting and such extensions are a torsor over \(\hom(A, \kkm) = \widehat{A}\).
    \item If \(\mathcal{C}\) is ribbon, \(\zeta\) extends to a ribbon zesting if and only if \(\delta\) is trivial in \(\cohomol{A, \mathbb{Z}/2\mathbb{Z}}\) and such extensions are a torsor over \(\hom(A, \mathbb{Z}/2\mathbb{Z}) = \widehat{A / 2 A}\).
    \item 
      If \(\mathcal{C}\) is pseudo-unitary then \(\mathcal{C}^{\zeta}\) is pseudo-unitary if and only if \(\epsilon = 1\) is the trivial cochain.
      \qedhere
  \end{thmenum}
\end{corollary}

The torsor characterizations in parts (a) and (b) were previously known \cite[Proposition 5.1(iii)]{Delaney2020}, as was part (c) \cite[Theorem 3.18]{Galindo2023}.
Our existence results are new.

\begin{proof}
  (a)
  \Cref{eq:twist zesting epsilon} is simply the statement that \(\delta\) is the coboundary of \(\epsilon\).
  Because \(\kkm\) is divisible any symmetric \(2\)-cocycle with values in \(\kkm\) is a coboundary.%
  \note{
    Thanks to C\'esar Galindo for the following argument:
    Divisible abelian groups are injective (as \(\mathbb{Z}\)-modules) so if \(A\) is divisible every extension \(0 \to A \to B \to C\) is split.
    Because symmetric \(2\)-cocycles parametrize such extensions we conclude \(0 = \operatorname{Ext}_{\mathbb{Z}}(A,B) = \operatorname{H}_{\text{sym}}^{2}(B,A)\).
  }
  Clearly the ratio of two solutions to \cref{eq:twist zesting epsilon} is a homomorphism \(A \to \kkm\).

  (b)
  \Cref{eq:ribbon zesting epsilon} shows that the only additional requirement for a ribbon zesting is that \(\epsilon\) take values in \(\mathbb{Z}/2\mathbb{Z} \iso \set{1, -1}\).
  
  (c)
  If \(X\) is simple, it is homogeneous, so
  \[
    \dim_{\mathcal{C}^{\zeta}}(X) = \epsilon(|X|) \dim_{\mathcal{C}}(X)
  \]
  and along with \(\epsilon(a) \in \set{1, -1}\) this shows \(\epsilon = 1\) is the only choice preserving positive dimensions.
  If \(\mathcal{C}\) is pseudo-unitary then \(\delta\) is trivial so it is the coboundary of \(\epsilon = 1\) as required.
\end{proof}

\begin{lemma}
\label{lem: computational}
  For any braided zesting and any \(a, b \in A\),
  \begin{align}
    \label{eq:braided zesting consequence}
    \omega(a,b;ab)
    t(ab)
    &=
    t(a)
    t(b)
    t^{(2)}(a,b)
    \theta_{\lambda(a,b)}^{-1}
    \dim(\lambda(a,b))
    \\
    \label{eq:braided zesting consequence ii}
    \omega(a,a^{-1};a)
    &=
    t(a^{-1})^{-2} t^{(2)}(a, a^{-1})^{-1}.
  \end{align}
\end{lemma}

\begin{proof}
  Consider
  \begin{equation}
   \label{eq:braided zesting consequence LHS}
  \scalebox{.9}{$\omega(a,b;ab)$ \begin{tikzpicture}[line width=1,baseline=1.9cm]
    \draw (-1,0) node[below] {\small $\lambda(a,b)$}--(-1,4) node[above] {\small $\lambda(a,b)$};
    \draw (0,0) node[below] {\small $\lambda(a,b)$}--(0,4) node[above] {\small $\lambda(a,b)$};
    \draw (1,0) node[below] {\small $\phantom{aba}\lambda(ab,ab)$}--(1,4) node[above] {\small $\phantom{aba}\lambda(ab,ab)$};
    \begin{scope}[xshift=1cm, yshift=1.33cm]
    \draw[fill=white] (0,.65) circle (.55) node [] {\scriptsize $t(ab,ab)$};
    \end{scope}
    \end{tikzpicture}}.
    \end{equation}
  We can re-write the braided zesting conditions \eqref{fig:braided zesting conditions} as
  \begin{center}
     \scalebox{.9}{\begin{tikzpicture}[line width=1,baseline=1.9cm]
    \draw (0,0) node[below] {\small $\lambda(b,c)$}--(0,4) node[above] {\small $\lambda(b,c)$};
    \draw (1,0) node[below] {\small $\lambda(bc,a)$}--(1,4) node[above] {\small $\lambda(a,bc)$};
    \begin{scope}[xshift=1cm, yshift=1.33cm]
    \draw[fill=white] (0,.65) circle (.5) node [] {\scriptsize $t(a,bc)$};
    \end{scope}
    \end{tikzpicture}
    \quad = \quad
    \begin{tikzpicture}[line width=1,baseline=1.9cm]
     \draw (0,-1.5) node[below] {\small $\lambda(b,c)$}--(0,5.5) node[above] {\small $\lambda(b,c)$};
    \draw (1,-1.5) node[below] {\small $\lambda(b,ac)$}--(1,5.5) node[above] {\small $\lambda(a,bc)$};
    \begin{scope}[yshift=-1.33cm]
\draw[fill=white] (-.25, .415) rectangle node [] {\scriptsize $\nu(b,c,a)^{-1}$}(1.25,.915);
    \end{scope}
    \draw[fill=white] (0,.65) circle (.45) node [] {\scriptsize $t(a,c)$};
    \begin{scope}[yshift=2.66cm]
    \draw[fill=white] (0,.65) circle (.45) node [] {\scriptsize $t(a,b)$};
    \end{scope}
    \begin{scope}[yshift=1.33cm]
    \draw[fill=white] (-.25, .415) rectangle node [] {\scriptsize $\nu(b,a,c)$}(1.25,.915);
    \end{scope}
       \begin{scope}[yshift=3.99cm]
    \draw[fill=white] (-.25, .415) rectangle node [] {\scriptsize $\nu(a,b,c)^{-1}$}(1.25,.915);
    \end{scope}
    \end{tikzpicture}}
     and 
     \scalebox{.9}{\begin{tikzpicture}[line width=1,baseline=1.9cm]
    \draw (0,0) node[below] {\small $\lambda(a,b)$}--(0,4) node[above] {\small $\lambda(a,b)$};
    \draw (1,0) node[below] {\small $\lambda(c,ab)$}--(1,4) node[above] {\small $\lambda(ab,c)$};
    \begin{scope}[xshift=1cm, yshift=1.33cm]
    \draw[fill=white] (0,.65) circle (.5) node [] {\scriptsize $t(ab,c)$};
    \end{scope}
    \end{tikzpicture}
    \quad = $\omega(a,b;c)^{-1}$ \quad
    \begin{tikzpicture}[line width=1,baseline=1.9cm]
     \draw (0,-1.5) node[below] {\small $\lambda(b,c)$}--(0,5.5) node[above] {\small $\lambda(ab,c)$};
     \draw (1,-1.5) node[below] {\small $\lambda(a,b)$}--(1,5.5) node[above] {\small $\phantom{b}\lambda(c,ab)$};
   \begin{scope}[yshift=-1.33cm]
\draw[fill=white] (-.25, .415) rectangle node [] {\scriptsize $\nu(c,a,b)$}(1.25,.915);
    \end{scope}
    \draw[fill=white] (0,.65) circle (.45) node [] {\scriptsize $t(a,c)$};
    \begin{scope}[yshift=2.66cm]
    \draw[fill=white] (0,.65) circle (.45) node [] {\scriptsize $t(b,c)$};
    \end{scope}
    \begin{scope}[yshift=1.33cm]
    \draw[fill=white] (-.25, .415) rectangle node [] {\scriptsize $\nu(a,c,b)^{-1}$}(1.25,.915);
    \end{scope}
       \begin{scope}[yshift=3.99cm]
    \draw[fill=white] (-.25, .415) rectangle node [] {\scriptsize $\nu(a,b,c)$}(1.25,.915);
    \end{scope}
    \end{tikzpicture}}
  \end{center}
  We can then expand \eqref{eq:braided zesting consequence LHS} using the left-hand relation as
  \begin{center}
    \scalebox{.9}{$\omega(a,b;ab)$ \begin{tikzpicture}[line width=1,baseline=1.9cm]
    \draw (-1,0) node[below] {\small $\lambda(a,b)$}--(-1,4) node[above] {\small $\lambda(a,b)$};
    \draw (0,0) node[below] {\small $\lambda(a,b)$}--(0,4) node[above] {\small $\lambda(a,b)$};
    \draw (1,0) node[below] {\small $\phantom{aba}\lambda(ab,ab)$}--(1,4) node[above] {\small $\phantom{aba}\lambda(ab,ab)$};
    \begin{scope}[xshift=1cm, yshift=1.33cm]
    \draw[fill=white] (0,.65) circle (.55) node [] {\scriptsize $t(ab,ab)$};
    \end{scope}
    \end{tikzpicture}}
    =
    \scalebox{.9}{\begin{tikzpicture}[line width=1,baseline=1.9cm]
    \draw (-1,-1.5) node[below] {\small $\lambda(a,b)$}--(-1,5.5) node[above] {\small $\lambda(a,b)$};
     \draw (0,-1.5) node[below] {\small $\lambda(a,b)$}--(0,5.5) node[above] {\small $\lambda(a,b)$};
    \draw (1,-1.5) node[below] {\small $\phantom{aba}\lambda(ab,ab)$}--(1,5.5) node[above] {\small $\phantom{aba}\lambda(ab,ab)$};
    \begin{scope}[yshift=-1.33cm]
\draw[fill=white] (-.30, .415) rectangle node [] {\scriptsize $\nu(ab,a,b)^{-1}$}(1.30,.915);
    \end{scope}
    \draw[fill=white] (0,.65) circle (.5) node [] {\scriptsize $t(ab,b)$};
    \begin{scope}[yshift=2.66cm]
    \draw[fill=white] (0,.65) circle (.5) node [] {\scriptsize $t(ab,a)$};
    \end{scope}
    \begin{scope}[yshift=1.33cm]
    \draw[fill=white] (-.25, .415) rectangle node [] {\scriptsize $\nu(a,ab,b)$}(1.25,.915);
    \end{scope}
       \begin{scope}[yshift=3.99cm]
    \draw[fill=white] (-.30, .415) rectangle node [] {\scriptsize $\nu(a,b,ab)^{-1}$}(1.30,.915);
    \end{scope}
    \end{tikzpicture}}
  \end{center}
  Applying the right-hand relation to \(t(ab, a)\) and \(t(ab, b)\), pulling out scalar factors, and applying some cancellations shows that our quantity is equal to
  \begin{center}
    $\langle t(a,a) \rangle \langle t(b,b) \rangle \langle t^{(2)}(a,b) \rangle$ 
     \scalebox{.9}{\begin{tikzpicture}[line width=1,baseline=1.9cm]
    \draw (-1,-1.5) node[below] {\small $\lambda(a,b)$}--(-1,5.5) node[above] {\small $\lambda(a,b)$};
     \draw (0,-1.5) node[below] {\small $\lambda(a,b)$}--(0,5.5) node[above] {\small $\lambda(a,b)$};
    \draw (1,-1.5) node[below] {\small $\phantom{aba}\lambda(ab,ab)$}--(1,5.5) node[above] {\small $\phantom{aba}\lambda(ab,ab)$};
    \begin{scope}[yshift=-1.33cm]
\draw[fill=white] (-.30, .415) rectangle node [] {\scriptsize $\nu(a,b,ab)^{-1}$}(1.30,.915);
    \end{scope}
    \draw[fill=white] (-1.30, .415) rectangle node [] {\scriptsize $\nu(b,a,b)$}(.30,.915);
    \begin{scope}[yshift=2.66cm]
     \draw[fill=white] (-1.30, .415) rectangle node [] {\scriptsize $\nu(a,b,a)$}(.30,.915);
    \end{scope}
    \begin{scope}[yshift=1.33cm]
    \draw[fill=white] (-.25, .415) rectangle node [] {\scriptsize $\nu(a,ab,b)$}(1.25,.915);
    \end{scope}
       \begin{scope}[yshift=3.99cm]
    \draw[fill=white] (-.30, .415) rectangle node [] {\scriptsize $\nu(ab,a,b)^{-1}$}(1.30,.915);
    \end{scope}
    \end{tikzpicture}}
  \end{center}
  Finally applying the associative zesting condition
  gives
  \begin{center}
    \scalebox{.9}{$\omega(a,b;ab)$ 
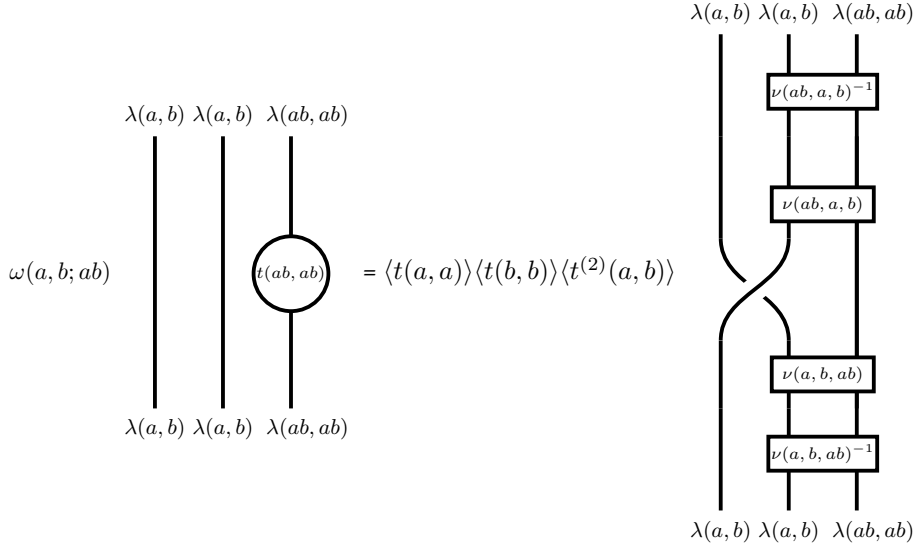
\begin{tikzpicture}[line width=1,baseline=1.9cm]
    \draw (-1,0) node[below] {\small $\lambda(a,b)$}--(-1,4) node[above] {\small $\lambda(a,b)$};
    \draw (0,0) node[below] {\small $\lambda(a,b)$}--(0,4) node[above] {\small $\lambda(a,b)$};
    \draw (1,0) node[below] {\small $\phantom{aba}\lambda(ab,ab)$}--(1,4) node[above] {\small $\phantom{aba}\lambda(ab,ab)$};
    \begin{scope}[xshift=1cm, yshift=1.33cm]
    \draw[fill=white] (0,.65) circle (.55) node [] {\scriptsize $t(ab,ab)$};
    \end{scope}
    \end{tikzpicture}}
    =
     $\langle t(a,a) \rangle \langle t(b,b) \rangle \langle t^{(2)}(a,b) \rangle$ 
     \scalebox{.9}{\begin{tikzpicture}[line width=1,baseline=1.9cm]
    \draw (-1,-1.5) node[below] {\small $\lambda(a,b)$} -- (-1,0);
    \draw (-1, 4)--(-1,5.5) node[above] {\small $\lambda(a,b)$}; 
     \draw (0,-1.5) node[below] {\small $\lambda(a,b)$}-- (0,0);
     \draw (0, 4) --(0,5.5) node[above] {\small $\lambda(a,b)$};
    \draw (1,-1.5) node[below] {\small $\phantom{aba}\lambda(ab,ab)$}--(1,5.5) node[above] {\small $\phantom{aba}\lambda(ab,ab)$};
    \begin{scope}[yshift=-1.33cm]
\draw[fill=white] (-.30, .415) rectangle node [] {\scriptsize $\nu(a,b,ab)^{-1}$}(1.30,.915);
    \end{scope}
       \begin{scope}[yshift=3.99cm]
    \draw[fill=white] (-.30, .415) rectangle node [] {\scriptsize $\nu(ab,a,b)^{-1}$}(1.30,.915);
    \end{scope}
    \begin{scope}[xshift=-1cm]
        \foreach \x in {0,1}
    {
    \draw(\x,0)--(\x,1);
    \draw(\x,2.5)--(\x,4);
    }
   \draw(2,0)--(2,4);
   \begin{scope}[xshift=0cm,yshift=1cm]
   \draw (1,0)  \br (0,1.5);
   \draw[white, line width=10] (0,0) \br (1,1.5);
   \draw (0,0) \br (1,1.5);
   \end{scope}
   \draw[fill=white] (.75,.25) rectangle node {\scriptsize $\nu(a,b,ab)$} (2.25,.75);
   \draw[fill=white] (.75,2.75) rectangle node {\scriptsize $\nu(ab,a,b)$} (2.25,3.25);
    \end{scope}
    \end{tikzpicture}}
  \end{center}
  Taking traces (in \(\mathcal{C}\)) gives
  \[
    \omega(a,b;ab)
    t(ab)
    \dim (\lambda(ab, ab))
    =
    t(a)
    t(b)
    t^{(2)}(a,b)
    \dim (\lambda(a,b)) \dim (\lambda(ab, ab))
    \theta_{\lambda(a,b)}^{-1}
  \]
  and since \(\dim \lambda(a,b)^{2} = 1\) we obtain \cref{eq:braided zesting consequence}.

  For \cref{eq:braided zesting consequence ii} setting \(b = a^{-1}, c = a\) in the first braided zesting condition \eqref{eq:braided zesting I} gives
  \[
    \nu(a, a^{-1}, a) = t(a) t(a, a^{-1})
  \]
  and in the second \eqref{eq:braided zesting II} gives
  \[
    \nu(a, a^{-1}, a)^{-1} \omega(a, a^{-1}; a) t(a)^{-1} = t(a^{-1}, a).
  \]
  Composing gives the identity
  \[
    \omega(a,a^{-1}; a) =
    t(a)^2 t^{(2)}(a,a^{-1})
  \]
  for all \(a \in A\).
  Replacing \(a\) with \(a^{-1}\) gives
  \[
    \omega(a^{-1},a; a^{-1}) =
    t(a^{-1})^2 t^{(2)}(a^{-1},a)
  \]
  and then since \(t^{(2)}(a^{-1},a) = t^{(2)}(a,a^{-1})\) and \(\omega(a^{-1}, a; a^{-1}) = \omega(a, a^{-1}; a^{-1}) = \omega(a, a^{-1}; a)^{-1}\) we conclude
  \[
    \omega(a,a^{-1};a) =
    t(a^{-1})^{-2} t^{(2)}(a, a^{-1})^{-1}.
    \qedhere
  \]
\end{proof}

\begin{proof}[Proof of \cref{thm:twist zesting conditions}]
  Substituting \(f = \epsilon t\) into \cref{eq:twist zesting f} and applying \eqref{eq:braided zesting consequence} shows \eqref{eq:twist zesting f} is equivalent to
  \[
    \epsilon(ab)
    \dim(\lambda(a,b))
    =
    \epsilon(a)
    \epsilon(b)
  \]
  and as \(\dim \lambda(a,b) \in \set{1, -1}\) this is equivalent to \cref{eq:twist zesting epsilon}.

  Clearly \eqref{eq:ribbon zesting f} is equivalent to
  \[
    1 
    =
    \frac{
      f(a^{-1}) \omega(a, a^{-1}; a^{-1}) \theta_{\lambda(a, a^{-1})}
    }{
      f(a).
    }
  \]
  Setting \(b = a^{-1}\) in \cref{eq:braided zesting consequence} gives
  \[
    \theta_{\lambda(a, a^{-1})}
    =
    \epsilon(a) \epsilon(a^{-1}) t(a) t(a^{-1}) t^{(2)}(a, a^{-1}).
  \]
  Therefore
  \begin{align*}
    \frac{
      f(a^{-1}) \omega(a, a^{-1}; a) \theta_{\lambda(a, a^{-1})}
    }{
      f(a)
    }
    &=
    \frac{
      t(a^{-1}) \epsilon(a^{-1}) 
      \epsilon(a) \epsilon(a^{-1}) t(a) t(a^{-1}) t^{(2)}(a, a^{-1})
    }{
      t(a) \epsilon(a) t(a^{-1})^{2} t^{(2)}(a, a^{-1})
    }
    \\
    &=
    \epsilon(a^{-1})^2
  \end{align*}
  which is equal to \(1\) for all \(a \in A\) if and only if  \(\epsilon(a)^2\) is.
\end{proof}

\section{Proof that zesting defines a shadow rack 2-cocycle}
\label{sec:proof of cocycle condition}
First we prove a relation satisfied by the associative zesting data $\nu$.
\begin{lemma}
  \label{thm:assoc-consequence}
  Under our conventions the \(3\)-cochain \(\nu\) of any zesting satisfies
  \begin{equation}
    \label{eq:assoc-consequence}
    \frac{
      \nu(a_0 a_1, a_2, a_3)
    }{
      \nu(a_0 a_1, a_3, a_2)
    }
    =
    \frac{
      \nu(a_0, a_1 a_2, a_3)
      \nu(a_0, a_1, a_2)
      \nu(a_1, a_2, a_3)
    }{
      \nu(a_0, a_1 a_3, a_2)
      \nu(a_0, a_1, a_3)
      \nu(a_1, a_3, a_2)
    }
  \end{equation}
  for all \(a_i \in A\).
\end{lemma}
\begin{proof}
  The associative zesting condition says that
  \[
    \nu(a_0 a_1, a_2, a_3)
    \nu(a_0, a_1, a_2 a_3)
    \beta^{-1}_{\lambda(a_0,a_1), \lambda(a_2,a_3)}
    =
    \nu(a_0, a_1, a_2)
    \nu(a_0, a_1 a_2, a_3)
    \nu(a_1, a_2 , a_3)
  \]
  or, swapping \(a_2\) and \(a_3\),
  \[
    \nu(a_0 a_1, a_3, a_2)
    \nu(a_0, a_1, a_3 a_2)
    \beta^{-1}_{\lambda(a_0,a_1), \lambda(a_3,a_2)}
    =
    \nu(a_0, a_1, a_3)
    \nu(a_0, a_1 a_3, a_2)
    \nu(a_1, a_3 , a_2).
  \]
  But we have chosen \(\lambda(a_2, a_3)\) and \(\lambda(a_3, a_2)\) to be equal (not just isomorphic) so
  \[
    \beta^{-1}_{\lambda(a_0,a_1), \lambda(a_2,a_3)}
    =
    \beta^{-1}_{\lambda(a_0,a_1), \lambda(a_3,a_2)}
  \]
  as well.
  Hence
  \[
    \frac{
      \nu(a_0, a_1, a_2)
      \nu(a_0, a_1 a_2, a_3)
      \nu(a_1, a_2 , a_3)
    }{
      \nu(a_0 a_1, a_2, a_3)
      \nu(a_0, a_1, a_2 a_3)
    }
    =
    \frac{
      \nu(a_0, a_1, a_3)
      \nu(a_0, a_1 a_3, a_2)
      \nu(a_1, a_3 , a_2)
    }{
      \nu(a_0 a_1, a_3, a_2)
      \nu(a_0, a_1, a_3 a_2)
    }
  \]
  and rearranging (and using the commutativity of \(A\)) gives \cref{eq:assoc-consequence}.
\end{proof}

\begin{proof}[Proof of \cref{thm: zest invariant is rack cocycle invariant}]%
The shadow rack 2-cocycle condition holds if
\[
  \phi_i(b,c) \phi_{ib}(a,c)\phi_i(a,b)=\phi_{ic}(a,b)\phi_i(a,c)\phi_{ia}(b,c).
\]
Applying the definition of $\phi$, the left hand side becomes
\[
  \frac{t(b,c) \nu(i,b,c)t(a,c) \nu(ib,a,c) t(a,b) \nu(i,a,b)}{\nu(i,c,b)\nu(ib, c,a)\nu(i,b,a)}
\]
and the right hand side becomes
\[
  \frac{t(a,b) \nu(ic,a,b)t(a,c) \nu(i,a,c) t(b,c) \nu(ia,b,c)}{\nu(ic,b,a)\nu(i, c,a)\nu(ia,c,b)}.
\]
The same factors of \(t\) appear on both sides, so it suffices to show that
\[
  \frac{\nu(i,b,c)\nu(ib,a,c) \nu(i,a,b)}{\nu(i,c,b)\nu(ib, c,a)\nu(i,b,a)}=\frac{ \nu(ic,a,b)\nu(i,a,c) \nu(ia,b,c)}{\nu(ic,b,a)\nu(i, c,a)\nu(ia,c,b)}.
\]
Rearranging to put all terms with a product in the first argument on the lefthand side, we want to show that 
\[
  \frac{\nu(ib,a,c)}{\nu(ib, c,a)}\frac{\nu(ic,b,a)}{\nu(ic,a,b)}\frac{\nu(ia,c,b)}{\nu(ia,b,c)}=\frac{\nu(i,b,a) \nu(i,c,b)\nu(i,a,c) }{ 
  \nu(i,a,b)\nu(i,b,c)\nu(i, c,a)}.
\]

After applying \cref{thm:assoc-consequence}--which relies only the associative zesting condition \cref{eq:associative zesting}--to each of the three ratios on the left hand side, the equation reduces to 
\[
  \frac{\nu(b,a,c)\nu(c,b,a)\nu(a,c,b)}{\nu(b,c,a)\nu(a,b,c)\nu(c,a,b)}= 1.
\]
By rearranging the braided zesting condition \cref{eq:braided zesting I}, we have
\[
  \frac{\nu(b,c,a)\nu(a,b,c)}{\nu(b,a,c)} = \frac{t(a,b)t(a,c)}{t(a,bc)}
  \qquad \text{ and } \qquad
  \frac{\nu(a,c,b)\nu(c,b,a)}{\nu(c,a,b)} = \frac{t(a,c)t(a,b)}{t(a,bc)},
\]
and substituting these in gives the desired result.
\end{proof}

\end{document}